\documentclass{tran-l}
\usepackage{graphicx} 
\usepackage[utf8]{inputenc}
\usepackage{lmodern}
\usepackage[T1]{fontenc}
\usepackage{amsmath,amssymb}
\usepackage{hyperref}
\usepackage{mathrsfs}
\usepackage{amsfonts}
\usepackage{mathtools}
\usepackage{tikz}
\usepackage{lmodern}
\usepackage{bigints}
\usepackage{relsize}
\usepackage{bbm}
\usepackage{amsthm}
\usepackage{geometry}
\usepackage{scalerel,stackengine}
\stackMath
\newcommand\reallywidehat[1]{
\savestack{\tmpbox}{\stretchto{
  \scaleto{
    \scalerel*[\widthof{\ensuremath{#1}}]{\kern-.6pt\bigwedge\kern-.6pt}
    {\rule[-\textheight/2]{1ex}{\textheight}}WIDTH-LIMITED BIG WEDGE
  }{\textheight} 
}{0.5ex}}
\stackon[1pt]{#1}{\tmpbox}
}
\numberwithin{equation}{subsection}
\newcommand{\norm}[1]{\left\lVert#1\right\rVert}
\newtheorem{theorem}{Theorem}[section]
\newtheorem{corollary}[theorem]{Corollary}
\newtheorem{lemma}[theorem]{Lemma}

\newtheorem{definition}[theorem]{Definition}
\newtheorem{notation}[theorem]{Notation}
\newtheorem{remark}[theorem]{Remark}
\theoremstyle{definition}
\DeclareMathOperator{\sppp}{Span}

\DeclareMathOperator{\I}{Im}
\DeclareMathOperator{\Ree}{Re}

\title{ Collision of two solitons for $1d$ Nonlinear Schrodinger Equation with the same mass}
\author{Abdon Moutinho}
\thanks{ The author is a member of the math department of Georgia Tech.
Address: 686 Cherry St NW, Atlanta, GA 30332, USA.
ORCID: 0000-0003-2841-8444.
The author consents the AMS permission to publish this article. The author acknowledges the support of Georgia Tech for this research. The author also thanks the suggestions of Gong Chen and Wilhelm Schlag in his work on this problem. Electronic address: aneto8@gatech.edu}
\date{May 2024}

\begin{document}

\begin{abstract}
 We study the global dynamics of the collision of two opposite solitons having the same mass for one-dimensional Nonlinear Schrödinger models with multi-power nonlinearity. For any natural number $k,$ it is verified that if the incoming speed $v$ between the two solitary waves is small enough, then, due to the interaction force between the solitons, the two solitary waves will move away with an outcoming speed $v_{f}=v+O(v^{k})$ after the collision and the remainder of the solution will also have energy and weighted norms of order  $O(v^{k}).$ This is applied to the one-dimensional Nonlinear Schrödinger equations having an odd polynomial nonlinearity with stable solitons such as the cubic NLS, and cubic-quintic NLS.   
\end{abstract}
\maketitle

\section{Introduction}
 
In this manuscript, we consider the following one-dimensional nonlinear Schrödinger model
\begin{equation}\label{NLS3}
    \Lambda(u)(t)\coloneq iu_{t}+u_{xx} +F^{'}(\vert u\vert^{2})u=0,
\end{equation}
such that $F$ is a real polynomial satisfying a real function satisfying
\begin{equation}\label{H1}\tag{H1}
    F(0)=0,\,F^{'}(0)=0. 
\end{equation}
One particular example of \eqref{NLS3} corresponds to the one-dimensional cubic Schrödinger equation, which is given by
\begin{equation}\label{cubicnls}\tag{Cubic NLS}
    iu_{t}+u_{xx}+2\vert u \vert^{2} u=0.
\end{equation} The partial differential equation \eqref{cubicnls} is known to be completely integrable and that there are explicit formulas for its solutions including the multi-soliton solutions, see the classical work of \cite{cubicinv}. Moreover, all the solutions of \eqref{cubicnls} with finite energy are invariant under infinite invariances including the following
\begin{itemize}
    \item $u_{\mu}(t,x)=\mu u(\mu^{2}t,\mu x)$ (dilation),
    \item $\tau_{\zeta}u(t,x)=u(t,x-\zeta)$ (space translation),
    \item $\iota_{\delta}u(t,x)=u(t+\delta,x)$ (time translation),
    \item $\sigma_{\theta}u(t,x)=e^{i\theta}u(t,x)$  (phase shift),\\
    \item $\Theta_{v}u(t,x)=u(t,x-vt)e^{i(\frac{vx}{2}-\frac{vt^{2}}{4})}$ (Galilean Transformation).
\end{itemize} 
Moreover, it is standard to verify that any strong solution of \eqref{NLS3} is invariant under space translation, time translation, and phase shift.
 \par Furthermore, the integrability of \eqref{cubicnls} also implies that all the strong solutions of the cubic Schrödinger model satisfy infinite conserved quantities. However, since many models of the form \eqref{NLS3} are not integrable, we will only use the following conserved quantities 
\begin{align}\label{H}\tag{Hamiltonian}
    H(u)=\int_{\mathbb{R}}\frac{\vert u_{x} \vert^{2}}{2}-\frac{F( \vert u\vert^{2})}{2},dx,\\ \label{M}\tag{Mass}
    Q(u)=\int_{\mathbb{R}}\vert u \vert^{2}\,dx,\\ \label{P}\tag{Momentum}
    M(u)=\I \int_{\mathbb{R}} \bar{u}u_{x}\,dx. 
\end{align}
\par In addition, the partial differential equation \eqref{NLS3} can be used to describe one-dimensional nonlinear Schrodinger with double power nonlinearity such as cubic-quintic models, which for real parameters $a,\, b$ is given by the following partial differential equation
\begin{equation}\label{gcq}
    iu_{t}+u_{xx}+a\vert u\vert^{2}u+b\vert u\vert^{4}u=0.
\end{equation}
Different from \eqref{cubicnls}, this model is non-integrable and we cannot use the inverse scattering transform to describe explicitly the strong solutions of \eqref{gcq} for all time $t.$ The model \eqref{gcq} has many physical applications, see \cite{phy1}, \cite{phy2}, \cite{phy3}, \cite{phy4}, \cite{opticssoltion} and \cite{phy5} for example. 
\par Moreover, the partial differential equation \eqref{NLS3} also includes the one-dimensional Schrödinger with triple  power nonlinearity such as the following partial differential equation
\begin{equation}\label{triplee}
    iu_{t}+u_{xx}+a\vert u\vert^{2n}u+b\vert u\vert^{2m}u+c\vert u\vert^{2l}u=0,
\end{equation}
for real parameters $a,\,b,\,c.$ The study of the dynamics of soliton and multi-solitons for the model \eqref{triplee} is of great interest to Optical Physics, see for example \cite{optics}, \cite{opticstriple1}. Concerning one-dimensional Schrödinger with four-power nonlinearity, see also the article \cite{fourpower} of the field of Optical physics. 
\par Next, concerning the existence of solutions \eqref{soliton} of the partial differential equation \eqref{NLS3}, we consider the Theorem $5$ from the article \cite{solitoneq} by Berestycki and Lions.
\begin{theorem}\label{ordt} %improve condition of F in my main result
Let $\omega>0,$ if
    \begin{equation*}
    T_{\omega}(y)\coloneq -\omega \frac{y^{2}}{2}+\frac{F(\vert y\vert^{2})}{2}
\end{equation*}
satisfies for some $y_{0}>0$ 
\begin{itemize}
    \item [(H2)] $T(y_{0})=0,$ 
    \item [(H3)] $T^{'}(y_{0})>0$ and $T(y)>0$ for all $y>y_{0},$
\end{itemize}
then the ordinary differential equation
\begin{equation}\label{odeord}
\begin{cases}
{-}\phi_{\omega}^{''}=T^{'}_{\omega}(\phi_{\omega})={-}\omega\phi_{\omega}+F^{'}(\phi_{\omega}^{2})\phi_{\omega},\\
\phi_{\omega}(0)=y_{0}
\end{cases}
\end{equation}
has a unique positive solution $\phi_{\omega}\in H^{1}(\mathbb{R}).$
\end{theorem}
\begin{remark}\label{asyrem}
    Indeed, $\phi_{\omega}$ shall be an even function and under the assumption that $F$ satisfies \eqref{H1}, we can verify that there is a real number $a_{{+}\infty}>0$ satisfying 
    \begin{equation*}
          \lim_{x\to \pm \infty} \phi_{\omega}(x)e^{\pm \sqrt{\omega} x}=a_{{+} \infty}.
    \end{equation*}
\end{remark}
\begin{remark}
\par Furthermore, from the article \cite{solitoneq}, $\phi_{\omega}$ shall satisfy for all $x\geq 0$
\begin{equation}\label{ode}
    \phi^{'}_{\omega}(x)={-}\sqrt{\omega \phi_{\omega}(x)^{2}-F(\phi^{2}_{\omega})}.
\end{equation}
Consequently, using the change of variables $y(x)=e^{{-}\sqrt{\omega}x}$ and considering $\mathcal{G}_{\omega}(y)\coloneqq e^{\sqrt{\omega}x}\phi_{\omega}(x),$ the ordinary differential equation \eqref{ode} can be rewritten for $x>1$ as
\begin{equation}\label{ode2}
\begin{cases}
    \frac{d}{dy}\mathcal{G}_{\omega}(y)=\frac{1}{y}\left(\left[\mathcal{G}_{\omega}(y)^{2}-\frac{F\left(\mathcal{G}_{\omega}(y)^{2}y^{2}\right)}{y^{2}\omega}\right]^{\frac{1}{2}}-\mathcal{G}_{\omega}(y)\right),\\
    \mathcal{G}_{\omega}(0)=a_{+\infty}.
\end{cases}    
\end{equation}
Consequently, since $F$ satisfies \eqref{H1}, the function
\begin{equation*}
    H(z,y)=\frac{1}{y}\left(\left[z^{2}-\frac{F\left(z^{2}y^{2}\right)}{y^{2}\omega}\right]^{\frac{1}{2}}-z\right)
\end{equation*}
has an analytic extension over some set of the form $\{(z,y)\in\mathbb{C}^{2}\vert\,\vert z-a_{+\infty}\vert+\vert y \vert<\delta\},$ we can verify using Piccard-Lindelöf Theorem that there exists a unique holomorphic function $\mathcal{G}_{\omega}(y)$ satisfying \eqref{ode2} over a set $\{\vert y\vert\leq \delta\}.$ This implies for a $\delta_{\omega}>0$ depending on $\omega$ and $F$ the existence of a real analytic function $\mathcal{P}$ satisfying $\mathcal{P}(0)=0,\,\mathcal{P}^{'}(0)\neq 0$ and
\begin{equation}\label{domega}
    \phi_{\omega}(x)=\mathcal{P}\left(e^{-{\sqrt{\omega}\vert x\vert}}\right) \text{, if $\vert x\vert\geq \delta_{\delta_{\omega}}.$}
\end{equation}
Moreover, since $\phi_{\omega}$ also satisfies \eqref{odeord}, we can verify by induction that $\mathcal{P}$ is an odd analytic function.
\end{remark}
In this paper, we are only going to consider the partial differential equations \eqref{NLS3} such that there exists $\omega>0$ satisfying Theorem \ref{ordt}.
As a consequence, if $\omega>0$ satisfies all the hypotheses in the statement of Theorem \ref{ordt}, then the following function
\begin{equation}\label{soliton}\tag{Standing wave}
    u(t,x)=e^{i\omega t}\phi_{\omega}(t)
\end{equation}
is a strong solution of the partial differential equation \eqref{NLS3} in the space $C(\mathbb{R},H^{1}(\mathbb{R},\mathbb{C})).$  Moreover, under the hypotheses of Theorem \ref{ordt}, we can obtain using the Galilean transformation in the solution \eqref{soliton} the following set of solutions
\begin{equation*}
    u(t,x)=\phi_{\omega}(x-vt-y)e^{i(\omega t+\gamma)}e^{i\left(\frac{v x}{2}-\frac{v^{2}t}{4}\right)},
\end{equation*}
for any $v,\,y,\,\gamma\in\mathbb{R},$ which are denominated by solitary waves.
\par The stability theory of the solitons for Nonlinear Schrödinger models was studied by Grillakis, Shatah, and Strauss in the articles \cite{solisch}, \cite{solisch2}. In addition, there have been several research about the asymptotic stability of solitons for one-dimensional Nonlinear Schrödinger models. See for example \cite{asympt0}, \cite{asympt1}, \cite{asympt2}, \cite{asympt3}, \cite{asympt4}, \cite{asympt5} and \cite{asympt6}.
\par Moreover, it was proved the following proposition in article \cite{solisch}.
\begin{lemma}\label{L0}
If the function $\phi_{\omega}(x)$ satisfies
\begin{equation}\label{H4}\tag{H4}
\frac{d}{d\omega}\int_{\mathbb{R}}\vert \phi_{\omega}(x) \vert^{2}>0,
\end{equation}
for some $\omega=\omega_{1},$ then the soliton $\phi_{\omega}(x)$ is orbital stable in the space $H^{1}(\mathbb{R}).$ Otherwise, if
\begin{equation}\label{H3}
\frac{d}{d\omega}\int_{\mathbb{R}}\vert \phi_{\omega}(x) \vert^{2}<0,
\end{equation}
then the soliton $\phi_{\omega}$ is unstable.
\end{lemma}
\par Furthermore, in the article \cite{ohta}, using Lemma \ref{L0}, Ohta studied the stability and instability of one-dimensional Nonlinear Schrödinger models with double-power nonlinearity such as the one-dimensional cubic-quintic Schrödinger equation given in \eqref{gcq}. 
\par We can now state our main theorems.

\begin{theorem}\label{main}
Let $F$ be any real polynomial and $\omega>0$ both satisfying hypotheses $(H1),\,(H2),\,(H3),\,(H4).$ There exists constants $C>0,\,M>1$ and, for any $k\in\mathbb{N},$  there is $\delta_{k}\in(0,1)$ such that for any $v_{1}\in\mathbb{R}$ if $0<v<\delta_{k},$ and $u$ is the unique solution of \eqref{NLS3} 
satisfying 
\begin{align}\label{initialcondition}
\lim_{t\to{+\infty}}\norm{e^{{-}i\omega t}u(t,x)-\phi_{\omega}(x-(v+v_{1})t)e^{i\left(\frac{(v+v_{1})x}{2}-\frac{(v+v_{1})^{2}t}{4}\right)}+\phi_{\omega}(x+(v-v_{1})t)e^{i\left(-\frac{(v-v_{1})x}{2}-\frac{(v-v_{1})^{2}t}{4}\right)}}_{H^{1}}e^{ct}\\ \nonumber=0,
\end{align}
  for a $c>0,$
  then there are functions $\gamma,\,\zeta$ of class $C^{1}$ satisfying for all $t<\frac{{-}2Ck\vert\ln{v}\vert}{v}$
\begin{align}\nonumber
    \norm{e^{{-}i\gamma(t)}u(t)-\left[\phi_{\omega}(x-\zeta(t)-v_{1}t)e^{i\left({-}\frac{(v-v_{1})x}{2}-\frac{v_{1}^{2} t}{4}\right)}-\phi_{\omega}(x+\zeta(t)-v_{1}t)e^{i\left(\frac{(v+v_{1})x}{2}-\frac{v_{1}^{2}t}{4}\right)}\right]}_{L^{2}}<&v^{k},\\ \label{almoelres}
    \norm{\frac{\partial}{\partial x}\left[e^{{-}i\frac{v_{1}x}{2}}u(t)-\left(\phi_{\omega}(x-\zeta(t)-v_{1}t)e^{i\left(\gamma(t)-\frac{vx}{2}-\frac{v_{1}^{2} t}{4}\right)}-\phi_{\omega}(x+\zeta(t)-v_{1}t)e^{i\left(\gamma(t)+\frac{vx}{2}-\frac{v_{1}^{2}t}{4}\right)}\right)\right]}_{L^{2}}<& v^{k},
\end{align}
and 
\begin{align}\label{modass}
    \left\vert \dot \gamma(t)-\omega+\frac{v^{2}}{4}\right\vert+
    \vert \dot \zeta(t)+v \vert<v^{k},
\end{align}
for all $t\leq \frac{{-}\vert\ln{v}\vert^{\frac{4}{3}}}{v}.$
\par Moreover, for any $l\in\mathbb{N},$ there exists $0<\delta_{k,l}<\delta_{k}$ such that if $0<v<\delta_{k,l},$ then
\begin{equation}\label{weightednorm}
    \norm{(1+\vert x\vert^{l})\left[e^{{-}i\gamma+\frac{v_{1}t^{2}}{4}}u(t)-\phi_{\omega}(x-\zeta(t)-v_{1}(t))e^{i\left(\frac{{-}(v+v_{1})x}{2}\right)}+\phi_{\omega}(x+\zeta(t)-v_{1}t)e^{i\left(\frac{(v+v_{1})x}{2}\right)}\right]}_{L^{2}}<v^{k}
\end{equation}
for any $\frac{{-}2\vert\ln{v}\vert^{\frac{4}{3}}}{v}\leq t\leq \frac{{-}\vert\ln{v}\vert^{\frac{4}{3}}}{v}.$
\end{theorem}
\begin{remark}
 From Theorems \ref{energyestimatetheor}, \ref{app lemma} and Lemma \ref{lambdasol}, we have that $\zeta(t)=-vt+c_{\omega}\ln{\frac{1}{v}}+O(v^{2})$ as $t$ approach ${-}\infty$ for some real constant $c_{\omega}.$ The shift $c_{\omega}\ln{\frac{1}{v}}$ follows from the repulsive force between the two solitons having opposite phases, see Lemma \ref{lambdasol}.     
\end{remark}
\begin{remark}
Concerning the case where the two solitons have a small difference in their masses, the behavior of the solution is expected to be different because of the symmetry break of the ordinary differential equations associated with the parameters $\zeta,\,\gamma,\,\omega $
 and $v.$ This argument was used in \cite{collision1} by Martel and Merle to describe the collision of two solitons for the quartic $gKdV,$ see also the work \cite{perelman} of Perelman to describe the collision of two different solitons for \eqref{NLS3} in the case where $F$ is $C^{2}$ and $F^{''}(0)\neq 0.$
 \end{remark}
\begin{remark}
 From the estimates \eqref{weightednorm} considered in Theorem \ref{main}, we might expect to be able to verify new phenomena as the scattering of the remainder for example. In the article \cite{asympt0} of Collot and Germain, similar hypotheses in the weighted norm were considered for the proof of the asymptotic stability of a single soliton for \eqref{NLS3}.   
\end{remark}
\begin{remark}\label{v1=0}
Moreover, because of the Galilean Transformation invariance, it is enough to prove the Theorem \ref{main} when $v_{1}=0.$  \end{remark}
\begin{theorem}\label{uniq}
  There is $\delta_{0}>0$ such that if $0<v<\delta_{0},$ then the solution of \eqref{NLS3} satisfying for $T>1,\,c>0$ and any $t\geq T$
\begin{equation}\label{moq1}
    \norm{u(t,x)-\phi_{\omega}(x-vt)e^{i(\omega t+\frac{vx}{2}-\frac{v^{2}t}{4})}+\phi_{\omega}(x+vt)e^{i(\omega t+\frac{{-}vx}{2}-\frac{v^{2}t}{4})}}_{H^{1}}=O(e^{{-}c\vert t\vert}),
\end{equation}
is unique. 
Furthermore, there is $\delta_{l,m}>0$ such that if $0<v<\delta_{l,m},$ then the solution $u$ satisfies for any $m,\,l \in\mathbb{N}$ and $t\geq 10 (l+m+1) T,$ then
\begin{equation}\label{pureweight}
    \norm{(1+x^{2})^{\frac{l}{2}}\left[u(t,x)-\phi_{\omega}(x-vt)e^{i(\omega t+\frac{vx}{2}-\frac{v^{2}t}{4})}+\phi_{\omega}(x+vt)e^{i(\omega t+\frac{{-}vx}{2}-\frac{v^{2}t}{4})}\right]}_{H^{m}}
    =O\left(e^{{-}c\vert t\vert}\right).
\end{equation}
\end{theorem}
\begin{remark}\label{runiq}
 From the proof of Theorem \ref{uniq}, we obtain $c=\frac{3\sqrt{\omega}v}{4}$ and $T=\frac{4\ln{\left(\frac{1}{v}\right)}}{\sqrt{\omega}v}.$  
\end{remark}
\begin{remark}
It is expected that we can repeat the argument used in the proof of Theorem \ref{uniq} to verify the uniqueness of multi-solitons for \eqref{NLS3} when the norm remainder has exponential decay. The existence of multi-solitons for one-dimensional Nonlinear Schrödinger models was already verified in \cite{mnls} by Martel and Merle. However, the parameters $c,\,T$ obtained in the article \cite{mnls} satisfying \ref{moq1} are not good for our analysis of the long-time behavior of the collision of the two solitons. Therefore, we consider Theorem \ref{uniq} for the collision of two solitons. Therefore, we consider Theorem \ref{uniq} for our article.   
\end{remark}
\par The Theorem \ref{main} describes globally the collision between two stable solitons with the same mass and opposite phases for any model of the form \eqref{NLS3} such as the cubic \eqref{cubicnls}, the cubic-quintic \eqref{gcq} which is non-integrable and also one-dimensional models with multi power nonlinearity such as the \eqref{triplee}. Moreover, the estimates \eqref{almoelres} and \eqref{modass} imply that the collision between the two solitons is almost elastic, because, for any $k\in\mathbb{N},$ the energy norm in the remainder and the change in the size of the speed of propagation of each soliton after the collision can of order $O(v^{k}).$ This conclusion is quite surprising since this is not much expected for non-integrable models of the form \eqref{almoelres}.

\par Concerning the study of the interaction between solitons for Nonlinear Schrödinger models, there exist previous works. In \cite{perelman}, Perelman studied the collision between a large soliton and a small soliton and concluded that the solution doesn't preserve the two solitons' structure during a finite long-time interval after the collision. In the article \cite{holmerlin}, Holmer and Lin studied the interaction between two solitons with the same mass for the model \eqref{cubicnls} having the same or opposite phase. Concerning collision between solitons of Nonlinear Schrödinger models having high speed, see the article \cite{sigcoll} by Salem, Fröhlich, and Sigal.  
\par In this paper, using the methods from \cite{second}, \cite{third}, we are going to analyze the collision of two identical stable solitons of \eqref{NLS3} having low differences in their speeds.

Since similar methods were used by the author in \cite{third} in a non-integrable one-dimensional nonlinear wave equation, we believe the method used in this paper has applications in a large set of one-dimensional nonlinear dispersive models. 
\par The mathematical research of the collision phenomena between solitons hasn't been restricted only to Schrödinger models. In the articles \cite{stabcol} and \cite{collision1}, Martel and Merle studied the collision between two solitons for the quartic $gKdV,$ they proved in \cite{collision1} that no solution of the quartic $gKdV$ is a pure two multi-soliton, and the collision is inelastic and the $H^{1}$ norm of the error in the approximate two solitons solution is of order cubic in the speed of the two solitary waves before the collision, see Theorem $1$ from \cite{collision1}. Moreover, in \cite{munozin}, Muñoz studied the collision between two solitons of different sizes for $gKdV$ models and obtained that the collision is inelastic when the model is non-integrable, see also \cite{munozkdv2} for information about the collision between solitons for slow-varying $gKdV.$
\par In \cite{second} and \cite{third}, the author studied the collision between two solitons denominated kinks for the one-dimensional nonlinear wave equation known as the $\phi^{6}$ model. Moreover, the main result of the paper \cite{third} is similar to Theorem \ref{main}, the collision preserves the two solitons' structure and the energy norm of the defect can be of order $O(v^{k})$ for any $k\in\mathbb{N}$ when $v>0$ is small enough, where $v$ is the incoming speed of the solitons before the collision.
\par Furthermore in \cite{collisiondidier},  we also cite the recent work of Pilod and Valet in the description of the collision of two nearly equal solitary waves for the Zakharov-Kuznetsov partial differential equation in dimensions $2$ and $3.$ 
\par Before, we move to the next section, we need to consider the following notation.
\begin{notation}
\label{notsigma}
For any function $f:\mathbb{R}\to\mathbb{C},$ and any real  functions $\zeta,v,\gamma:\mathbb{R}\to\mathbb{R}$ we denote
\begin{equation*}
    f(v,\zeta,\gamma)
\end{equation*}
by the following function
\begin{equation*}
     e^{\frac{iv(t)(x-\frac{\zeta(t)}{2})}{2}+i\gamma(t)+\omega t}f(x-\zeta(t)),
\end{equation*}
for all $(t,x)\in\mathbb{R}^{2}.$
\par We also consider the following
\begin{equation}\label{symoo}
    Sym(f)(x)\coloneqq f(x)-f({-}x),
\end{equation}
for any function $f:\mathbb{R}\to\mathbb{C}.$ In addition, we for any two functions $f,\,g:\mathbb{R}\to\mathbb{C},$ we consider
\begin{equation}\label{gesym}
Sym\left[f(\cdot)g(\cdot)\right](x)=f(x)g(x)-f({-}x)g({-}x).
\end{equation}
Moreover, for any $f:\mathbb{R}^{2}\to \mathbb{C},$ we denote
\begin{equation*}
Sym(f)(t,x)\coloneqq f(t,x)-f(t,{-}x) \text{, for any $(t,x)\in\mathbb{R}^{2}.$}    
\end{equation*}
In particular, we are going to use the notation \eqref{gesym} to describe many expressions in Section \ref{refapp}.   
\par Next, for any $\zeta\in\mathbb{R},$ we consider the following notation for space translation:
\begin{equation*}
    \tau_{\zeta}f(x)\coloneq f(x-\zeta).
\end{equation*}
\par In this paper, we will only consider the following dot product restricted to the space $L^{2}\left(\mathbb{R},\mathbb{C}\right)$
\begin{equation}\label{dotproduct}
\left\langle f,g\right\rangle\coloneqq \Ree \int_{\mathbb{R}}f(x)\bar{g}(x)\,dx.
\end{equation}
\par Furthermore, in this manuscript, all the expressions with $\sum$ only represent a finite sum, and, we say that a real or complex function $f$ with a domain contained $D$ in $\mathbb{R}$ is of order $O(g(t))$ for some real positive function $g$ if there exists a $C>0$ satisfying
\begin{equation*}
\vert f(t)\vert\leq C g(t )\text{, for all $t\in D.$}
\end{equation*}
In all this manuscript, for any real function $f(t)>0,$ the partial differential inequality
\begin{equation*}
    i\partial_{t}r(t,x)+\partial^{2}_{x}r(t,x)=F(t,x)+O\left(f(t)\right),
\end{equation*}
means for some $C>1$ that $r$ is a strong solution of a partial differential equation
\begin{equation*}
    i\partial_{t}r(t,x)+\partial^{2}_{x}r(t,x)=F(t,x)+G(t,x),
\end{equation*}
such that $\norm{G(t)}_{H^{1}}<Cf(t)$ for all $t\in\mathbb{R}$ or all $t$ in the domain of $r(t,x).$
\par Finally, for any function $f(t,r)$ over $\mathbb{C}$ and defined for $t\in\mathbb{R}$ and $r\in \mathbb{H},$ $\mathbb{H}$ being some Hilbert Space, we define the function $\dot f(t,r)$ to be
\begin{equation*}
    \frac{\partial f(t,r)}{\partial t},
\end{equation*}
when the partial derivative above is well-defined on $t.$
\end{notation}

%comments about application of main theorem any polynomials nonlinearity, many non-integrable models, applications to optics
% Comment similar results of collision of solitons
% Conjectures: wieghted norm of inelasticity? Extension of asymptotic stability for multi-solitons in weighted norms.
% Resume of the proof
\subsection{Resume of the proof}
\par The proof of Theorem \ref{main} is similar to the demonstration of Theorem $1.2$ from the paper \cite{third}, and it also requires the techniques from \cite{second} which are going to be revised in Sections \ref{back} and \ref{refapp}. 
\par First, in Section \ref{refapp}, we focus on constructing a sequence of approximate solutions $(\varphi_{k})_{k\in\mathbb{N}}$ of \eqref{NLS3} satisfying for any $s\geq 0$
\begin{equation}\label{000}
\norm{\frac{\partial^{l}}{\partial t^{l}}\Lambda(\varphi_{k})(t,x)}_{H^{s}_{x}(\mathbb{R})}\leq C(s,l)v^{2k+2+l}\left(\vert t\vert v+\ln{\frac{1}{v}}\right)^{c(k)}e^{{-}2\sqrt{\omega}\vert t\vert v},
\end{equation}

for some $t_{k}\in\mathbb{R},$ and for $v>0$ small enough. The functions $\varphi_{k}$ are of following form:
\begin{align}\label{ge0}
    \varphi_{k}(t,x)=&\left[e^{i\frac{v_{k}}{2}(x-\frac{\zeta_{k}}{2})+i\gamma_{k}}\phi_{\omega}(x-\zeta_{k})-e^{{-}i\frac{v_{k}}{2}(x+\frac{\zeta_{k}}{2})+i\gamma_{k}}\phi_{\omega}({-}x-\zeta_{k})\right]\\ \nonumber &{+}\sum_{j\in I_{k}}g_{j}(t)\left[e^{i\frac{v_{k}}{2}(x-\frac{\zeta_{k}}{2})+i\gamma_{k}}p_{j,\omega}(x-\zeta_{k})-e^{{-}i\frac{v_{k}}{2}(x+\frac{\zeta_{k}}{2})+i\gamma_{k}}p_{j,\omega}({-}x-\zeta_{k})\right],
\end{align}
such that $\zeta_{k}(t)>\frac{1}{2\sqrt{\omega}}\ln{\left(\frac{1}{v\omega^{\frac{1}{4}}}\right)},$ see Remark \ref{remestimate}, and all the functions $g_{j},\,p_{j,\omega}$ are in $\mathscr{S}(\mathbb{R},\mathbb{C})$ having exponential decay. However, since the general formula of $\varphi_{k}$ is slightly complicated, we shall explain briefly the method to obtain each function $\varphi_{k}.$ 
\par From the fact that the soliton $e^{i\omega t}\phi_{\omega}(x)$ is a solution of \eqref{NLS3}, we consider our initial approximate solution satisfying \eqref{000} to be
\begin{equation*}
    \varphi_{0}(t,x)=e^{i\omega t}e^{\frac{i\dot d(t)}{2}\left(x-\frac{d(t)}{2}\right)}\phi_{\omega}(x-d(t))-e^{i\omega t}e^{{-}\frac{i\dot d(t) }{2}\left(x+\frac{d(t)}{2}\right)}\phi_{\omega}(x+d(t)),
\end{equation*}
  for some large function $d(t)>1$ to be chosen carefully. More precisely, it is possible to find a smooth function $d(t)>1$ such that $\varphi_{0}$ satisfies \eqref{000} and
  \begin{equation}\label{dotsmall}
     \left\langle \Lambda(\varphi_{0}(t,x)),ie^{i\omega t}e^{\frac{i\dot d(t)}{2}\left(\pm x-\frac{d(t)}{2}\right)} \phi_{\omega}(x\mp d(t)) \right\rangle=O\left(v^{4}\right),
  \end{equation}
when $v>0$ is small enough.
The estimate \eqref{dotsmall} happens when $d(t)$ satisfies the following ordinary differential equation
\begin{equation*}
\begin{cases}
    \ddot d(t)=Ce^{-2\sqrt{\omega} d(t)},\\
    \lim_{t\to{+}\infty}\left\vert d(t)-vt-c_{\omega} \ln{v}\right\vert=0,\, \lim_{t\to{+}\infty}\left\vert \dot d(t)-v \right\vert=0,
\end{cases}
\end{equation*}
see Lemma \ref{ODE}. From this choice of $d(t),$ we deduce that
\begin{equation}
\norm{\frac{\partial^{l}}{\partial t^{l}}\Lambda(\varphi_{0})(t,x)}_{H^{1}}=O(v^{2+l}).
\end{equation}
The ordinary differential equation above was also studied in \cite{second} to describe the collision between two kinks for the $\phi^{6}$ model in \cite{third}.
\par Furthermore, using estimate \eqref{dotsmall}, we can find a better approximate solution $\varphi_{1}$ of \eqref{NLS3}.
But, since the construction of $\varphi_{1}$ from $\varphi_{0}$ is similar to the process to obtain $\varphi_{k+1}$ from $\varphi_{k}$ for any $k\in\mathbb{N},$ we  shall explain the construction of $\varphi_{k}$ for general $k\in\mathbb{N}.$
\par Using the exponential decay of the functions $g_{j},\, p_{j}$ and that $\zeta_{k}>1$ is very large, we find a finite set of complex Schwartz functions $r_{j},\,\rho_{j}$
such that
\begin{equation*}
\Lambda(\varphi_{k})(t,x)=\sum r_{j}(t)Sym\left[\rho_{j}(\cdot-\zeta_{k}(t))e^{i\frac{v_{k}(t)}{2}\left(\cdot-\frac{\zeta_{k}(t)}{2}\right)}e^{i\gamma_{k}(t)}\right](x)+O(v^{2k+4}).
\end{equation*}
Next, we need the properties of the operator $S_{\omega}$ defined by

\begin{equation*}
    S_{\omega}(\rho)\coloneqq -\rho^{''}+\omega \rho-F^{'}(\phi_{\omega}^{2})\rho-F^{''}(\phi_{\omega}^{2})\phi_{\omega}^{2}\left[\rho+\bar{\rho}\right],
\end{equation*}
which is invertible in the orthogonal complement of $\sppp\{\phi^{'}_{\omega},i\phi_{\omega},\partial_{\omega}\phi_{\omega},ix\phi_{\omega}\},$ see Lemmas \ref{coerc} and \ref{inverhol}. The operator $S_{\omega}$ comes from the linear part of equation \eqref{NLS3} when  $u(t)=\left[\phi_{\omega}(x)+\rho(t,x)\right]e^{i\omega t}.$ 
\par Next, using ordinary differential equation methods, we can replace the smooth parameters $\zeta_{k},\,v_{k},\,\gamma_{k}$ with $\zeta_{k+1},\,v_{k+1},\,\gamma_{k+1}$ such that the new function $\varphi_{k,0}$ obtained satisfies for some $c>0$
\begin{equation}\label{badd}
\Lambda(\varphi_{k,0})(t,x)=\sum r_{j}(t)Sym\left[\Pi^{\perp}\left(\rho_{j}\right)(\cdot-\zeta_{k+1}(t))e^{i\frac{v_{k+1}(t)}{2}\left(\cdot-\frac{\zeta_{k+1}(t)}{2}\right)}e^{i\gamma_{k+1}(t)}\right](x)+O\left(v^{2k+4}\left(\ln{\frac{1}{v}}\right)^{c}\right),
\end{equation}
where $\Pi^{\perp}$ is the orthogonal projection operator in the orthogonal complement of the subspace $\sppp\{\phi^{'}_{\omega},i\phi_{\omega},\partial_{\omega}\phi_{\omega},ix\phi_{\omega}\}$ in $L^{2}(\mathbb{R}).$
More precisely, using the invertible property of $S_{\omega},$ we can consider the following function of correction 
\begin{align*}
    Corr(t,x)=&\sum r_{j}(t)Sym\left[S^{{-}1}_{\omega}\left(\Pi^{\perp}\left(\rho_{j}\right)\right)(\cdot-\zeta_{k+1}(t))e^{i\frac{v_{k+1}(t)}{2}\left(\cdot-\frac{\zeta_{k+1}(t)}{2}\right)}e^{i\gamma_{k+1}(t)}\right](x)\\
    &{+}\sum \dot r_{j}(t)Sym\left[S^{{-}1}_{\omega}\left(i\Pi^{\perp}\left(\rho_{j}\right)\right)(\cdot-\zeta_{k+1}(t))e^{i\frac{v_{k+1}(t)}{2}\left(\cdot-\frac{\zeta_{k+1}(t)}{2}\right)}e^{i\gamma_{k+1}(t)}\right](x)
\end{align*}
for the removal of the main estimate in \eqref{badd}, the function $Corr$ is well define because of Lemma \ref{inverhol} from Section \ref{back}. Therefore, we conclude that the approximate solution $\varphi_{k+1}(t,x)=\varphi_{k,0}(t,x)+Corr(t,x)$ satisfies 
\begin{equation*}
\norm{\frac{\partial^{l}}{\partial t^{l}}\Lambda(\varphi_{k+1})(t,x)}_{H^{s}_{x}(\mathbb{R})}\leq C_{1}(s,l)v^{2k+4+l}\left(\vert t\vert v+\ln{\frac{1}{v}}\right)^{c(k+1)}e^{{-}2\sqrt{\omega}\vert t\vert v},
\end{equation*}
for any $l\in\mathbb{N}$ if $v>0$ is small enough.
\par In Section \ref{s}, we study the long-time stability of the approximate solutions $\varphi_{k}$ using energy estimate methods, this is very similar to the approach in \cite{holmerlin} and \cite{third}. In the remaining sections, we prove the main theorem from the results of Section \ref{s}. See also Section \ref{back} for background information about the techniques used in this manuscript. 
\par In Section \ref{orbsec}, we study the orbital stability of two opposite solitary waves with the distance between their centers sufficiently large. The results in \eqref{orbsec} will allow us to prove the estimate \eqref{almoelres} for all $t\leq \frac{{-}\vert\ln{v}\vert^{\frac{4}{3}}}{v}$ when $v>0$ is small enough.
\par The proof of Theorem \ref{main} is written in Section $6.$ The proof of Theorem \ref{uniq} is written in the Appendix Section, and it is completely similar to the approach in the article \cite{multison} by Jendrej and Chen to describe the uniqueness of kink Networks.
\section{Background}\label{back} 
\par First, we consider the following lemma which we are going to use several times in the main body of this paper to compute our estimates.

\begin{lemma}\label{interactt}
For any real numbers $x_{2},x_{1}$, such that $\zeta=x_{2}-x_{1}>0$ and $\alpha,\,\beta,\,m>0$ with $\alpha\neq \beta$ the following bound holds:
\begin{equation*}
    \int_{\mathbb{R}}\vert x-x_{1}\vert ^{m} e^{-\alpha(x-x_{1})_{+}}e^{-\beta(x_{2}-x)_{+}}\lesssim_{\alpha,\beta,m} \max\left(\left(1+\zeta^{m}\right)e^{-\alpha \zeta},e^{-\beta \zeta}\right),
\end{equation*}
For any $\alpha>0$, the following bound holds
\begin{equation*}
    \int_{\mathbb{R}}\vert x-x_{1}\vert^{m} e^{-\alpha(x-x_{1})_{+}}e^{-\alpha(x_{2}-x)_{+}}\lesssim_{\alpha}\left[1+\zeta^{m+1}\right] e^{-\alpha \zeta}.
\end{equation*}
\end{lemma}
\begin{proof}
    Elementary computations.
\end{proof}
Next, we consider the following proposition obtained from Taylor's Expansion Theorem.
\begin{lemma}\label{taylorlemma}
 Let $F$ be a smooth function satisfying \eqref{H1}, if $c>0$ and $u,\,v$ are real numbers satisfying $\vert u\vert+\vert v\vert\leq c,$ then for any $k\in\mathbb{N}$
 \begin{align}\label{taylor1}
     \left\vert F^{'}(u+v)-F^{'}(v)-F^{'}(u)\right\vert\lesssim_{c}& \vert u  v \vert,\\ \label{taylor2}
     \left\vert F^{(k)}(u+v)-\sum_{m=1}^{l} \frac{F^{(k+l)}(u)v^{m}}{m !}\right\vert\lesssim_{c} &\vert v^{l+1}\vert.
\end{align}
\end{lemma}
\begin{proof}
    The inequality \eqref{taylor2} is a direct consequence of Taylor's Expansion Theorem, so we only need to verify the first inequality. 
    \par From the hypothesis \eqref{H1}, we have $F^{'}(0)=0.$ In conclusion, since $F$ is smooth, we obtain that
    \begin{multline*}
        \begin{aligned}
        F^{'}(u+v)-F^{'}(u)-F^{'}(v)= &F^{'}(u+v)-F^{'}(u)-F^{'}(v)-F^{'}(0)\\
        =& \int_{0}^{1}\left[F^{''}(u+\theta v)-F^{''}(\theta v) \right]v\,d\theta\\
        =& \left[\int_{0}^{1}\int_{0}^{1}F^{(3)}(\alpha u+\theta v)\,d\theta\,d\alpha\right]u v\\
        \lesssim_{c} &\vert u v \vert,
        \end{aligned}
    \end{multline*}
    and the last inequality above follows from $F^{(3)}$ being a bounded function in any bounded interval $[{-}c,c].$
\end{proof}
Based on the approach at \cite{second}, we consider the following space:
\begin{definition}\label{s11}
     Let $\delta_{\omega}>0$ be the value defined in Remark \eqref{asyrem} satisfying \eqref{domega}. $\mathcal{S}^{+},\,\mathcal{S}^{-}$ are linear sub-spaces of $L^{\infty}(\mathbb{R})$ such that $f\in \mathcal{S}^{+},\,g\in \mathcal{S}^{-}$ if $f,\,g\in C^{\infty}(\mathbb{R})$ and
      \begin{align*}
    f(x)=&\sum_{j=0}^{+\infty}a_{j}e^{(2j+1)x} \text{, for all $x\leq {-}\delta_{\omega},$}\,
     f(x)=\sum_{j=0}^{+\infty}b_{j}e^{{-}(2j+1)x} \text{, for all $x\geq \delta_{\omega},$}\\
     g(x)=&\sum_{j=0}^{+\infty}c_{j}e^{2 j x} \text{, for all $x\leq {-}\delta_{\omega},$}\,
     g(x)=\sum_{j=0}^{+\infty}q_{j}e^{{-}2jx} \text{, for all $x\geq \delta_{\omega},$}
     \end{align*}
     
     and the functions
    \begin{align*}
        A(z)=\sum_{j=0}^{+\infty}a_{j}z^{(2j+1)},\,
        B(z)=\sum_{j=0}^{+\infty}b_{j}z^{(2j+1)},\,
        C(z)=\sum_{j=1}^{+\infty}c_{j}z^{2j},\,Q(z)=\sum_{j=1}^{+\infty}q_{j}z^{2j}
    \end{align*}
    are analytic on the open unitary disk $\mathbb{D}\subset\mathbb{C}.$
\end{definition}
\begin{remark}\label{phis+}
    If $\phi_{\omega}$ satisfies Theorem \ref{ordt}, it is not difficult to verify using Remark \ref{asyrem} that the function $\Phi_{\omega}(x)\coloneqq \phi_{\omega}(\frac{x}{\sqrt{\omega}})$ is an element of $\mathcal{S}^{+},$ therefore it satisfies
    \begin{equation*}
       \left\vert \frac{d^{l}\phi_{\omega}(x)}{dx^{l}}\right\vert=O\left(e^{{-}\sqrt{\omega}\vert x\vert}\right), 
    \end{equation*}
for any $l\in\mathbb{N}.$
\end{remark}
\begin{definition}\label{S}
For any $n\in \mathbb{N}\cup\{0\},$ the linear spaces
$\mathcal{S}^{+,n}=\{x^{n}f(x)|\,f(x)\in \mathcal{S}^{+}\cap \mathscr{S}(\mathbb{R})\}$
and
$\mathcal{S}^{-,n}=\{x^{n}f(x)|\,f(x)\in \mathcal{S}^{-}\cap \mathscr{S}(\mathbb{R})\},$
and for any $m\in\mathbb{N}\cup\{0\},$ we define 
\begin{equation*}
    \mathcal{S}^{+}_{m}=\bigoplus_{n=0}^{m} \mathcal{S}^{+,n},\,\mathcal{S}^{-}_{m}=\bigoplus_{n=0}^{m} S^{-,n},\,
    S^{+}_{\infty}=\bigoplus_{n=0}^{+\infty} \mathcal{S}^{+,n},\,\mathcal{S}^{-}_{\infty}=\bigoplus_{n=0}^{+\infty} \mathcal{S}^{-,n}.
\end{equation*}
\end{definition}
\begin{remark}
The spaces $\mathcal{S}^{+}$ are going to be used to construct approximate solutions of \eqref{NLS3} that behave like the function $u$ satisfying the conclusion of Theorem \ref{main}. 
\end{remark}
Moreover, from the Definitions \ref{s11} and \ref{S}, we can verify the following elementary propositions.
\begin{lemma}\label{deriv}
    For any $l\in\mathbb{N},$ if $f\in \mathcal{S}^{+}\cup \mathcal{S}^{+}_{\infty}$ and $g\in \mathcal{S}^{-}\cup \mathcal{S}^{-}_{\infty},$ then  
    \begin{equation*}
        \frac{d^{l}f(x)}{dx^{l}}\in \mathcal{S}^{+}_{\infty}, \, \frac{d^{l}g(x)}{dx^{l}}\in \mathcal{S}^{-}_{\infty},
    \end{equation*}
    for any natural number $l\geq 1.$
\end{lemma}
\begin{lemma}\label{alg}
    For any $l\in\mathbb{N},$ if $f_{1},\,f_{2},\,f_{3}\in  \mathcal{S}^{+}_{\infty}$ and $g_{1},\,g_{2}\in \mathcal{S}^{-}_{\infty},$ then  
    \begin{equation*}
        f_{1}(x)f_{2}(x)f_{3}(x)\in \mathcal{S}^{+}_{\infty}, \, g_{1}(x)g_{2}(x)\in \mathcal{S}^{-}_{\infty}.
    \end{equation*}
    In particular, if $f\in \mathcal{S}^{+}_{\infty},$ then
    for any natural number $l\geq 1$
    \begin{equation*}
        f(x)^{2l-1}\in \mathcal{S}^{+}_{\infty}, \, f(x)^{2l}\in \mathcal{S}^{-}_{\infty}.
    \end{equation*}
\end{lemma}

Similarly to the proof of Proposition $2.4$ from the article \cite{second}, we can verify the following proposition.
\begin{lemma}\label{separation}
If $f \in \mathcal{S}^{+},\,g \in \mathcal{S}^{-}$ and $l,m\in\mathbb{N},$ then there exist a unique sequence of pairs $(h_{n},d_{n})_{n\geq 1}$ and a set $\Delta\subset \mathbb{N}$ such that $(d_{n})_{n\geq 1}\subset \mathbb{N}$ is a strictly increasing sequence, for all $n\in \Delta,\,h_{n}(-x)$ is in $\mathcal{S}^{+}\cap \mathscr{S}(\mathbb{R}),$ for all $n\in \Omega=\mathbb{N}\setminus \Delta,$ $h_{n}(x)$ is in $\mathcal{S}^{+}\cap \mathscr{S}(\mathbb{R}),$ and for any $\mathcal{M} \in \mathbb{N}$ and any $\zeta\geq 1$  
\begin{equation}\label{idsepl}
    f(x-\zeta)g(x)=\sum_{\substack{1 \leq n\leq \mathcal{M},\\ n\in  \Delta}}h_{n}(x-\zeta)e^{-d_{n}\zeta}+\sum_{\substack{1 \leq n\leq \mathcal{M},\\ n\in \mathbb{N}\setminus \Delta}}h_{n}(x)e^{-d_{n}\zeta}+e^{-d_{\mathcal{M}} \zeta}f_{\mathcal{M}}(x-\zeta)g_{\mathcal{M}}(x),
\end{equation}
where either $f_{\mathcal{M}} \in \mathcal{S}^{+} \cap\mathscr{S}(\mathbb{R}),\,g_{\mathcal{M}}\in \mathcal{S}^{-}$ or $f_{\mathcal{M}} \in \mathcal{S}^{-} \cap\mathscr{S}(\mathbb{R}),\,g_{\mathcal{M}}\in \mathcal{S}^{+}.$
\end{lemma}

\begin{remark}\label{remaindersep}
    Using Lemma \ref{interactt} and interpolation, we can verify that
    \begin{equation*}
        \norm{f_{\mathcal{M}}(x-\zeta)g_{\mathcal{M}}(x)}_{H^{s}_{x}(\mathbb{R})}\lesssim_{s} (1+\zeta)e^{-{\zeta}},
    \end{equation*}
for any real number $s.$
\end{remark}
From now on, we are going to study the properties of the following linear operator
\begin{align}\label{sw}\tag{$S_{\omega}$}
    S_{\omega}(\rho)\coloneqq &-\rho^{''}+\omega \rho-F^{'}(\phi_{\omega}^{2})\rho-F^{''}(\phi_{\omega}^{2})\phi_{\omega}^{2}\left[\rho+\bar{\rho}\right]
\end{align}
on the function space $\mathcal{S}^{+}.$ Indeed, it is not difficult to verify that ${-}e^{i\omega t}S_{\omega}(\rho)$ is the linear part of the expression   \begin{equation*}
iu_{t}+u_{xx}+F^{'}(\vert u\vert^{2})u
\end{equation*}
around $\rho$ for $u(t,x)=\left(\phi_{\omega}(x)+\rho(x)\right)e^{i\omega t}.$
Similarly to the approach in the article \cite{holmerlin}, we can verify the following property of $S.$
\begin{lemma}\label{kersw}
Assuming that $\omega>0$ and $\phi_{\omega}$ satisfy Theorem \ref{ordt}, the kernel of the map $S_{\omega}$ is the following subspace of $L^{2}(\mathbb{R})$
\begin{equation*}
\ker{S_{\omega}}=\sppp{ \left\{\phi^{'}_{\omega},\,i\phi_{\omega}\right\}}.
\end{equation*}
\end{lemma}
Moreover, from the article \cite{solisch}, we also have the coercivity property satisfied by the operator $S_{\omega}.$
\begin{lemma}\label{coer}
 There exists a constant $c>0$ such that if $g\in H^{1}(\mathbb{R},\mathbb{C})$ is orthogonal in $L^{2}(\mathbb{R},\mathbb{C})$ to the functions $\phi^{'}_{\omega},\,i\phi_{\omega},\,\partial_{\omega}\phi_{\omega},$ then
 \begin{equation*}
     \left\langle S_{\omega}(g),g\right\rangle\geq c\norm{g}_{H^{1}}^{2}. 
 \end{equation*}
Moreover, for any $g\in H^{1}(\mathbb{R},\mathbb{C}),$ we have
\begin{equation*}
 \left\langle S_{\omega}(g),g\right\rangle\geq  c\norm{g}_{H^{1}}^{2}-C\left[\left\langle g,\phi^{'}_{\omega}\right\rangle^{2}+\left\langle g,i\phi_{\omega}\right\rangle^{2}+\left\langle g,\partial_{\omega}\phi_{\omega}\right\rangle^{2}\right],
\end{equation*}
for some constant $C>1.$
\end{lemma}
\begin{remark}\label{coerremark}
Moreover, we also can verify for some constant $c>0$ that if $g\in H^{1}(\mathbb{R},\mathbb{C})$ is orthogonal in $L^{2}(\mathbb{R},\mathbb{C})$ to $x\phi_{\omega}(x),\,i\partial_{\omega}\phi_{\omega}$ and $\phi_{\omega},$ then
\begin{equation}\label{efinal}
     \left\langle S_{\omega}(g),g\right\rangle\geq c\norm{g}_{H^{1}}^{2}. 
 \end{equation}
 \par This follows because for $g_{1}$ the orthogonal projection of $g$ onto $\sppp\{\phi^{'}_{\omega},\,i\phi_{\omega},\,\partial_{\omega}\phi_{\omega}\}^{\perp},$ we have that
\begin{equation}\label{eqh}
    g=g_{1}+a_{1}\phi^{'}_{\omega}+a_{2}i\phi_{\omega}+a_{3}\partial_{\omega}\phi_{\omega},
\end{equation}
and $\left\langle S_{\omega}(g),g\right\rangle\geq c\norm{g_{1}}_{H^{1}}^{2}.$
\par Consequently, using equation \eqref{eqh} and the fact that $\phi_{\omega},\,\partial_{\omega}\phi_{\omega}$ are even functions, we deduce for some $C>0$ that
\begin{align*}
    \left\vert\left\langle g, x\phi_{\omega}(x)\right\rangle\right\vert^{2}+\norm{g_{1}}_{H^{1}}^{2}\geq &C a_{1}^{2},\\
    \left\vert\left\langle g, i\partial_{\omega}\phi_{\omega}(x)\right\rangle\right\vert^{2}+\norm{g_{1}}_{H^{1}}^{2}\geq &Ca_{2}^{2},\\
    \left\vert\left\langle g, \phi_{\omega}(x)\right\rangle\right\vert^{2}+\norm{g_{1}}_{H^{1}}^{2}\geq &Ca_{3}^{2}.
\end{align*}
Therefore, we have for some $K>0$ that
\begin{equation*}
    \norm{g}_{H^{1}}^{2}\leq K\left[\norm{g_{1}}_{H^{1}}^{2}+\norm{\Pi_{1}(g)}_{L^{2}}^{2}\right]=K\norm{g_{1}}_{H^{1}}^{2},
\end{equation*}
where $\Pi_{1}$ is the orthogonal projection of $L^{2}(\mathbb{R},\mathbb{C})$ into $\sppp\{x\phi_{\omega}(x),\,\phi_{\omega}(x),\,i\partial_{\omega}\phi_{\omega}(x)\}.$
\par In conclusion, using Lemma \ref{coer} implies that
\begin{equation}\label{impoco}
\left\langle S_{\omega}(g),g\right\rangle\geq c\norm{g_{1}}_{H^{1}}^{2}\geq \frac{c}{K} \norm{g}_{H^{1}}^{2}-\norm{\Pi_{1}(g)}_{L^{2}}^{2}=\frac{c}{K} \norm{g}_{H^{1}}^{2}.
\end{equation}

\end{remark}
\par Next, we consider the following proposition which can be verified directly using the differential equations' formulas \eqref{NLS3} and \eqref{odeord}, see also Section $4$ of the article in \cite{holmerlin} for a proof restricted to \eqref{cubicnls}.
\begin{lemma}\label{invertexpfor}
Let $\rho\in\mathscr{S}(\mathbb{R};\mathbb{C})$ and $h(x)=S_{\omega}(\rho)(x).$ Then, for any real smooth functions $f,\, \zeta,\,v,\,\gamma$ such that the functions
\begin{align*}
    \rho_{\theta}(t,x)\coloneq & f(t)\rho(v(t),\zeta(t),\gamma(t)),\\
    \rho_{\sigma}(t,x)\coloneq & \rho(v(t),\zeta(t),\gamma(t)),\\
    \phi_{\omega,\sigma}(t,x)\coloneq &\phi_{\omega}\left(x-\zeta(t)\right)e^{i\left(\gamma (t)+\frac{v(t)}{2}\left(x-\frac{\zeta(t)}{2}\right)+\omega t\right)}
\end{align*}
satisfy
\begin{multline*}
    i\partial_{t}\rho_{\theta}+\partial^{2}_{x}\rho_{\theta}+F^{'}(\vert \phi_{\omega,\sigma}\vert^{2})\rho_{\theta}+F^{''}(\vert \phi_{\omega,\sigma}\vert^{2}) \left\vert\phi_{\omega,\sigma}\right\vert^{2}\rho_{\theta}+F^{''}(\vert \phi_{\omega,\sigma}\vert^{2})\phi_{\omega,\sigma}^{2}\overline{\rho_{\theta}}
    \\=i\dot f(t)\rho_{\sigma}(t,x)-f(t)h\left(v(t),\zeta(t),\gamma(t)\right)
    {+}if(t)\dot\gamma(t)\partial_{\gamma}\rho_{\sigma}(t,x)\\{+}i\dot v(t) f(t)\partial_{v}\rho_{\sigma}(t,x){+}if(t)(\dot \zeta(t)-v(t))\partial_{\zeta}\rho_{\zeta}(t,x).\end{multline*}
\end{lemma}
%put remark from the separation lemma
\par We also need to consider the following statement.
\begin{lemma}\label{inverhol}
 There exist $C,\,c,\,\delta>0$ such that for any $f\in L^{2}$ satisfying
 \begin{align}\label{ee1}
     \left\langle f,\frac{\partial \phi_{\omega}(v,\zeta,\gamma)}{\partial \zeta}\big\vert_{v=\zeta=\gamma=0} \right\rangle=0,\\ \label{ee2}
     \left\langle f,\frac{\partial \phi_{\omega}(v,\zeta,\gamma)}{\partial\gamma}\big\vert_{v=\zeta=\gamma=0} \right\rangle=0,
 \end{align}
$S_{\omega}^{{-}1}f$ is defined and satisfies
$\norm{S_{\omega}^{{-}1}f}_{H^{1}}<C\norm{f}_{L^{2}},$ and
\begin{align*}
     \left\langle S_{\omega}^{{-}1}f,\frac{\partial \phi_{\omega}(v,\zeta,\gamma)}{\partial \zeta}\big\vert_{v=\zeta=\gamma=0} \right\rangle=0,\\
     \left\langle S_{\omega}^{{-}1}f,\frac{\partial \phi_{\omega}(v,\zeta,\gamma)}{\partial\gamma}\big\vert_{v=\zeta=\gamma=0} \right\rangle=0.
 \end{align*}
Furthermore, if $\omega>\delta,$ $f\in H^{1}_{x}(\mathbb{R},\mathbb{C})$ and
 \begin{align}\label{ll1}
 \left\langle f,\frac{\partial \phi_{\omega}(v,\zeta,\gamma)}{\partial\omega}\big\vert_{v=\zeta=\gamma=0} \right\rangle=0,\\ \label{ll2}\left\langle f,\frac{\partial \phi_{\omega}(v,\zeta,\gamma)}{\partial v}\big\vert_{v=\zeta=\gamma=0} \right\rangle=0,
\end{align}
 then
\begin{align}\label{ff1}
\left\langle S_{\omega}^{{-}1}f,i\frac{\partial \phi_{\omega}(v,\zeta,\gamma)}{\partial\zeta}\big\vert_{v=\zeta=\gamma=0} \right\rangle=0,\\ \label{ff2}\left\langle S_{\omega}^{{-}1}f,i\frac{\partial \phi_{\omega}(v,\zeta,\gamma)}{\partial \gamma}\big\vert_{v=\zeta=\gamma=0} \right\rangle=0.
\end{align}
\end{lemma}
\begin{proof}[Proof of Lemma \ref{inverhol}]
First, the operator $S_{\omega}$ is self-adjoint in the Hilbert Space $L^{2}(\mathbb{R},\mathbb{C})$ having the dot product given in \eqref{dotproduct}. 
\par Since 
\begin{equation*}
    \frac{d}{d\omega}\norm{\phi_{\omega}}_{L^{2}}^{2}>0,
\end{equation*}
we have that $S_{\omega}$ has a unique negative eigenvalue $\lambda_{\omega}$ with its eigenspace being a one-dimensional subspace $B_{\omega}=\{m\rho_{\omega}\vert\,m\in\mathbb{R}\}$ of $L^{2}(\mathbb{R},\mathbb{C})$ satisfying$\ \left\langle \rho_{\omega},\partial_{\omega}\phi_{\omega} \right\rangle\neq 0,$ this is a consequence of Theorem $3$ from
\cite{solisch}.
\par Furthermore, we have that
\begin{equation*}
    \ker S_{\omega}=\sppp\{i\phi_{\omega},\,\phi^{'}_{\omega}\},
\end{equation*}
see Theorem $3.3$ from \cite{solisch}. Moreover, since $\phi_{\omega}$ has exponential decay, we have from Weyl's Theorem that $\sigma_{ess}(S_{\omega})=[\omega,{+}\infty),$ see Theorem $14.6$ from the book \cite{specbook}. Therefore, for all $f\in H^{1}(\mathbb{R},\mathbb{C})\cap \left(\ker S_{\omega}\right)^{\bot},$ there exists $f_{1}\in\mathbb{R}$ such that $f=a\rho_{\omega}+f_{1},\,\left\langle f_{1},\rho_{\omega}\right\rangle=0.$ 
\par Next, for any $g_{1}\in H^{1}(\mathbb{R},\mathbb{C})\cap \left(\ker S_{\omega} \bigoplus B_{\omega}\right)^{\perp},$ we have
\begin{equation}\label{pppq}
    \left\langle S_{\omega} g_{1},g_{1}\right\rangle\geq c\norm{g_{1}}_{L^{2}}^{2},
\end{equation}
for some constant $c>0$ depending only on $\omega.$ Moreover, since $\phi_{\omega}\in L^{\infty},$ we have that there exists $C>1$ satisfying
\begin{equation}\label{pppq2}
    \left\langle S_{\omega} g,g\right\rangle\geq \norm{g^{'}}_{L^{2}}^{2}-C\norm{g}_{L^{2}}^{2},
\end{equation}
for any $g\in H^{1}(\mathbb{R},\mathbb{C}).$
\par As a consequence, we deduce from inequalities \eqref{pppq} and \eqref{pppq2} that there exists $c>0$ depending only on $\omega$ satisfying
\begin{equation*}
    \left\langle S_{\omega} g_{1},g_{1}\right\rangle\geq c\norm{g_{1}}_{H^{1}}^{2},
\end{equation*}
for all $g_{1}\in H^{1}(\mathbb{R},\mathbb{C})\cap \left(\ker S_{\omega}\bigoplus B_{\omega}\right)^{\perp},$ and so  $\norm{g_{1}}_{H^{1}}\leq \frac{1}{c}\norm{S_{\omega}g_{1}}\leq \frac{1}{c}\norm{S_{\omega}g}.$ 
\par Therefore, we can verify the existence of a constant $C>1$ and a bounded linear map $S_{\omega}^{{-}1}$ satisfying $S_{\omega}^{{-}1}\left(\left(\ker S_{\omega}\right)^{\bot}\right)\subset \left(\ker S_{\omega}\right)^{\bot}$ and
\begin{equation*}
\norm{S_{\omega}^{{-}1} f}_{H^{1}}\leq C\norm{f}_{L^{2}},
\end{equation*}
for all $f\in \left(\ker S_{\omega}\right)^{\bot}.$
\par The identities \eqref{ff1} and \eqref{ff2} follow from the fact that $S_{\omega}$ is a self-adjoint operator and from the following identities
\begin{align}\label{somega1}
    S_{\omega}\left(\frac{\partial \phi_{\omega}(v,\zeta,\gamma)}{\partial\omega}\big\vert_{v=\zeta=\gamma=0}\right)=&S_{\omega}(\partial_{\omega}\phi_{\omega})={-}\phi_{\omega},\\ \label{somega2}
    S_{\omega}\left(\frac{\partial \phi_{\omega}(v,\zeta,\gamma)}{\partial v}\big\vert_{v=\zeta=\gamma=0}\right)=&S_{\omega}\left(i\frac{x}{2}\phi_{\omega}(x)\right)={-}i\phi^{'}_{\omega}.
\end{align}
\end{proof}

Moreover, we also consider the following standard lemmas to simplify our reasoning in the next sections.
\begin{lemma}\label{ProP}
    There exist unique orthogonal projections $\Pi,\,\Pi_{1}:L^{2}_{x}(\mathbb{R})\to L^{2}_{x}(\mathbb{R})$ such that the range of  $\Pi$ is  equal to $\sppp\{\phi^{ \prime}_{\omega},ix\phi_{\omega},i\phi_{\omega},\partial_{\omega}\phi_{\omega}\},$ the range of  $\Pi_{1}$ is  equal to $\sppp\{i\phi^{ \prime}_{\omega},x\phi_{\omega},\phi_{\omega},i\partial_{\omega}\phi_{\omega}\},$ and for any $f\in L^{2}_{x}(\mathbb{R},\mathbb{C})$  
    \begin{align*}
        \left\langle f-\Pi(f)(x),\phi_{\omega}^{'}(x)\right\rangle=&0,\,\left\langle  f-\Pi(f)(x),i\phi_{\omega}(x)\right\rangle=0,\\
        \left\langle f-\Pi(f)(x),ix\phi_{\omega}(x)\right\rangle=&0,\,\left\langle f-\Pi(f)(x),\partial_{\omega}\phi_{\omega}(x)\right\rangle=0,
    \end{align*}
 and
 \begin{align*}
        \left\langle f-\Pi_{1}(f)(x),i\phi_{\omega}^{'}(x)\right\rangle=&0,\,\left\langle f-\Pi_{1}(f)(x),\phi_{\omega}(x)\right\rangle=0,\\
        \left\langle f-\Pi_{1}(f)(x),x\phi_{\omega}(x)\right\rangle=&0,\,\left\langle f-\Pi_{1}(f)(x),i\partial_{\omega}\phi_{\omega}(x)\right\rangle=0.
    \end{align*}
\end{lemma}
\begin{remark}
    The orthogonal projection $\Pi$ is going to be used in Section $3$ to construct the approximate solutions. The projection $\Pi_{1}$ is going to be used only in Section $4$ to estimate the energy norm of the remainder of the approximate solution during a large time interval.
\end{remark}
Furthermore, similarly to the proof of Lemma $2.11$ from \cite{second}, we can verify the following statement.
\begin{lemma}\label{invertan}
    For any $n\in\mathbb{N}\cup\{0\},$ there exists a natural number $m_{n}\geq n$ such that if $S_{\omega}(\rho)\left(\frac{\cdot}{\sqrt{\omega}}\right)\in \mathcal{S}^{+}_{m},$ then $\rho\left(\frac{\cdot}{\sqrt{\omega}}\right)\in \mathcal{S}^{+}_{m_{n}}.$
\end{lemma}
\begin{remark}\label{oddinve}
    Furthermore, since $\phi_{\omega}$ is a real function, we can verify that if $\rho$ is a real function in the domain of $S^{{-}1}_{\omega},$ then $S^{{-}1}_{\omega}(\rho)$ is a real function. If $i\rho$ is a real function, and $\rho$ is in the domain of $S^{{-}1}_{\omega},$ then $iS^{{-}1}_{\omega}(\rho)$ is a real function. From this fact, we can verify Lemma \ref{invertan} similarly to the approach in the proof of Lemma $2.11$ from \cite{second}. 
\end{remark}
\begin{remark}
Similarly to the Lemma $2.9$ from \cite{second}, we can verify that the sum of the sets $x^{l}\mathcal{S}^{+}$ has to be a direct sum.
\end{remark}

\section{Approximate solutions}\label{refapp}
Following the approach made in \cite{second}, we are going to construct a sequence of approximate solutions of the following form
\begin{align}\label{uk}
    \varphi_{k}=&\left[e^{i\frac{v_{k}}{2}(x-\frac{\zeta_{k}}{2})+i\gamma_{k}}\phi_{\omega}(x-\zeta_{k})-e^{{-}i\frac{v_{k}}{2}(x+\frac{\zeta_{k}}{2})+i\gamma_{k}}\phi_{\omega}({-}x-\zeta_{k})\right]\ \\ \nonumber
    &{+}f_{\omega,k}(t)\left[e^{i\frac{v_{k}}{2}(x-\frac{\zeta_{k}}{2})+i\gamma_{k}}\partial_{\omega}\phi_{\omega}(x-\zeta_{k})-e^{{-}i\frac{v_{k}}{2}(x+\frac{\zeta_{k}}{2})+i\gamma_{k}}\partial_{\omega}\phi_{\omega}({-}x-\zeta_{k})\right]
    \\
    &{+}\sum_{j\in J} g_{j}(t)\left[e^{i\frac{v_{k}}{2}(x-\frac{\zeta_{k}}{2})+i\gamma_{k}}p_{j,\omega}(x-\zeta_{k})-e^{{-}i\frac{v_{k}}{2}(x+\frac{\zeta_{k}}{2})+i\gamma_{k}}p_{j,\omega}({-}x-\zeta_{k})\right],
\end{align}
where the functions $p_{1,\omega}\left(\frac{\cdot}{\sqrt{\omega}}\right),\, p_{2,\omega}\left(\frac{\cdot}{\sqrt{\omega}}\right),\, p_{3,\omega}\left(\frac{\cdot}{\sqrt{\omega}}\right)$ and $p_{j,\omega}\left(\frac{\cdot}{\sqrt{\omega}}\right)$ are in $\bigoplus_{l=0}^{{+}\infty} x^{l}\mathcal{S}^{+}.$ Moreover, all the functions $g_{j}$ and $f_{\omega,k}$ are Schwartz functions with exponential decay and $J_{k}$ is a finite set with size depending on $k.$ 
\par 
More precisely:
\begin{theorem}\label{app lemma}
    There exist a constant $c>0,$ and finite subsets $J,\, I$ of $\mathbb{N},$ and for each $k\in\mathbb{N},$ there are finite sets $J_{k},\, I_{k}$ and $\delta_{k}\in(0,1),\,c(k)\in\mathbb{R}$ such that if $0<v<\delta_{k},$ then $e^{{-}2\zeta_{k}(v,t)}\leq c(k)v^{2}$ for all $t$ and there exist functions $\sigma_{k}(t)=(\zeta_{k},\,v_{k},\,\gamma_{k},f_{\omega,k}):(0,1)\times \mathbb{R}\to\mathbb{R}^{4},\,g_{odd,j},\,g_{even,j}:(0,1)\times\mathbb{R}\to\mathbb{R}$ and real functions $p_{j,\omega}\in \mathcal{S}^{+}_{\infty}$ such that all functions $g_{even,j}(v,\cdot)$ and $g_{odd,j}(v,\cdot)$ are Schwartz, $\zeta_{k},\,f_{\omega,k}$ are even on $t,\,$ $v_{k},\,\gamma_{k}$ are on odd $t,$
    \begin{equation*} 
        \left\vert \frac{\partial^{l} f_{\omega,k}(v,t)}{\partial t^{l}} \right\vert\leq C(l)v^{2+l}\left(\ln{\left(\frac{1}{v}\right)}+\vert t\vert\right)^{c}e^{{-}2\sqrt{\omega}\vert t\vert v},
    \end{equation*}
for a positive constants $c,$ and the function $\varphi_{k}(t,x)=P_{k}(t,x,\sigma_{k}(t))$ defined by
\begin{align}\label{phik}  
P_{k}(t,x,\sigma_{k}(t))\coloneqq&\left[e^{i\frac{v_{k}}{2}(x-\frac{\zeta_{k}}{2})+i\gamma_{k}}\phi_{\omega}(x-\zeta_{k})-e^{{-}i\frac{v_{k}}{2}(x+\frac{\zeta_{k}}{2})+i\gamma_{k}}\phi_{\omega}({-}x-\zeta_{k})\right]\ \\ \nonumber
    &{+}f_{\omega,k}(v,t)\left[e^{i\frac{v_{k}}{2}(x-\frac{\zeta_{k}}{2})+i\gamma_{k}}\partial_{\omega}\phi_{\omega}(x-\zeta_{k})-e^{{-}i\frac{v_{k}}{2}(x+\frac{\zeta_{k}}{2})+i\gamma_{k}}\partial_{\omega}\phi_{\omega}({-}x-\zeta_{k})\right]
     \\\nonumber
&{+}e^{{-}2\sqrt{\omega}\zeta_{0}}\left[e^{i\frac{v_{k}}{2}(x-\frac{\zeta_{k}}{2})+i\gamma_{k}}p_{1,\omega}(x-\zeta_{k})-e^{{-}i\frac{v_{k}}{2}(x+\frac{\zeta_{k}}{2})+i\gamma_{k}}p_{1,\omega}({-}x-\zeta_{k})\right]
    \\ \nonumber
    &{+}\sum_{j\in J} ig_{odd,j}(v,t)\left[e^{i\frac{v_{k}}{2}(x-\frac{\zeta_{k}}{2})+i\gamma_{k}}p_{j,\omega}(x-\zeta_{k})-e^{{-}i\frac{v_{k}}{2}(x+\frac{\zeta_{k}}{2})+i\gamma_{k}}p_{j,\omega}({-}x-\zeta_{k})\right]\\ 
    \nonumber
    &{+}\sum_{j\in I} g_{even,j}(v,t)\left[e^{i\frac{v_{k}}{2}(x-\frac{\zeta_{k}}{2})+i\gamma_{k}}p_{j,\omega}(x-\zeta_{k})-e^{{-}i\frac{v_{k}}{2}(x+\frac{\zeta_{k}}{2})+i\gamma_{k}}p_{j,\omega}({-}x-\zeta_{k})\right]\\ 
    \nonumber
    &{+}\sum_{j\in J_{k}}i g_{odd,j}(v,t)\left[e^{i\frac{v_{k}}{2}(x-\frac{\zeta_{k}}{2})+i\gamma_{k}}p_{j,\omega}(x-\zeta_{k})-e^{{-}i\frac{v_{k}}{2}(x+\frac{\zeta_{k}}{2})+i\gamma_{k}}p_{j,\omega}({-}x-\zeta_{k})\right]\\ \nonumber
    &{+}\sum_{j\in I_{k}} g_{even,j}(v,t)\left[e^{i\frac{v_{k}}{2}(x-\frac{\zeta_{k}}{2})+i\gamma_{k}}p_{j,\omega}(x-\zeta_{k})-e^{{-}i\frac{v_{k}}{2}(x+\frac{\zeta_{k}}{2})+i\gamma_{k}}p_{j,\omega}({-}x-\zeta_{k})\right],
\end{align}
all functions $g_{even,j}$ are even, all functions $g_{odd,j}$ are odd, and 
\begin{equation}\label{ukest}
\norm{\frac{\partial^{l}}{\partial t^{l}}\Lambda(\varphi_{k})(t,x)}_{H^{s}_{x}(\mathbb{R})}\leq C(s,l)v^{2k+2+l}\left(\vert t\vert v+\ln{\frac{1}{v}}\right)^{c(k)}e^{{-}2\sqrt{\omega}\vert t\vert v},
\end{equation}
for all $l\in\mathbb{N}$ and $t,\,s\in\mathbb{R},$ and there is $t_{k}=O\left(\frac{1}{v}\ln{\frac{1}{v}}\right)$ such that
\begin{equation*}
   \lim_{t\to{+}\infty}\norm{\varphi_{k}(t+t_{k})-\phi_{\omega}(x-vt)e^{iv(x-vt)+i\omega t+i\frac{v^{2}t}{2}}+\phi_{\omega}(x+vt)e^{{-}iv(x+vt)+i\omega  t+i\frac{v^{2}t}{2}}}_{H^{1}_{x}(\mathbb{R})}=0,
\end{equation*}
Moreover, there exist $c>0,$ and   $c_{k}>0$ depending on $k$ such that if $v>0$ is small enough, then
\begin{align*}
    \max_{j\in J}\left\vert \frac{\partial^{l}}{\partial l^{l}} g_{odd,j}(v,t)\right\vert+\max_{j\in I}\left\vert \frac{\partial^{l}}{\partial l^{l}} g_{even,j}(v,t)\right\vert\lesssim_{k,l} &v^{2+l}\left(\vert t\vert v+\ln{\frac{1}{v}}\right)^{c}e^{{-}2\sqrt{\omega}\vert t\vert v},\\
    \max_{j\in J_{k}}\left\vert \frac{\partial^{l}}{\partial l^{l}} g_{odd,j}(v,t)\right\vert+\max_{j\in I_{k}}\left\vert \frac{\partial^{l}}{\partial l^{l}} g_{even,j}(v,t)\right\vert\lesssim_{k,l} &v^{4+l}\left(\vert t\vert v+\ln{\frac{1}{v}}\right)^{c_{k}}e^{{-}2\sqrt{\omega}\vert t\vert v}.
\end{align*}
\end{theorem}
\begin{remark}\label{induest}
 Furthermore, we are going to verify that the functions $f_{\omega,k},\,v_{k},\,\gamma_{k}$ and $\zeta_{k}$ satisfy for some $s_{k}>0$ the following asymptotic expansion
\begin{align*}
   f_{\omega,k}(v,t)=&f_{\omega,k-1}(v,t)+O\left(v^{2k+2}\left(\ln{\frac{1}{v}}\right)^{s_{k}}\right),\,v_{k}(t)=v_{k-1}(v,t)+O\left(v^{2k+1}\left(\ln{\frac{1}{v}}\right)^{s_{k}}\right),\\
   \gamma_{k}(v,t)=&\gamma_{k-1}(v,t)+O\left(v^{2k+1}\left(\ln{\frac{1}{v}}\right)^{s_{k}}\right),\,\zeta_{k}(v,t)=\zeta_{k-1}(v,t)+O\left(v^{2k}\left(\ln{\frac{1}{v}}\right)^{s_{k}}\right),
\end{align*}    
for any $k\in\mathbb{N}_{\geq 1}$ if $v>0$ is small enough.
\end{remark}
\begin{remark}\label{remestimate}
 From the proof of Theorem \ref{app lemma}, we have that there is $c>0$ such that $\zeta_{k}(t)=d(t)+O(v^{2}(\ln{\frac{1}{v}})^{c}),$
 for $d(t)$ satisfying \eqref{d0},
 and    \begin{equation}\label{popopopo}
   v^{4}\left\vert \frac{\partial^{l}}{\partial t^{l}}\zeta_{k}(v,t)\right\vert+v^{3}\left\vert \frac{\partial^{l}}{\partial t^{l}}v_{k}(v,t)\right\vert+v^{3}\left\vert \frac{\partial^{l}}{\partial t^{l}}\left[\gamma_{k}(v,t)-\omega t\right]\right\vert\leq Cv^{4+l}\left(\vert t\vert v+\ln{\frac{1}{v}}\right)^{c}e^{{-}2\sqrt{\omega}\vert t\vert v},
\end{equation}
for any $l\in\mathbb{N}_{\geq 1}$ if $v>0$ is sufficiently small.
\end{remark}
\begin{remark}
    In the Section \ref{s}, we are going to use the notation $P_{k}(t,x,\sigma_{k}(t))$  from the statement of Theorem \ref{app lemma} instead of $\varphi_{k}(t,x).$
\end{remark}
In the notation of Theorem \ref{app lemma}, to simplify our reasoning, we consider for any smooth function $\rho$ that 
\begin{equation}\label{trk}
    \rho_{k}(t,x)\coloneqq e^{i\left(\gamma_{k}(t)+\frac{v_{k}(t)}{2}\left[x-\frac{\zeta_{k}(t)}{2}\right]\right)}\rho(x-\zeta_{k}(t))=e^{i\alpha_{k}(t,x)}\rho(x-\zeta_{k}(t)).
\end{equation}
\par Next, before we start the proof of Theorem \ref{app lemma}, we state a useful proposition for us to estimate of $\Lambda(u_{k})$ with high precision for every $k\in\mathbb{N}.$ More precisely:
\begin{lemma}
    \label{represent1} In notation of Theorem \ref{app lemma}, if there is $c>0$ such that $\norm{\partial^{l}_{t}\Lambda(\varphi_{k})(t,x)}_{H^{1}(\mathbb{R})}=O\left(v^{2\mathcal{M}+l}\left(\ln{\left(\frac{1}{v}\right)+\vert t\vert v}\right)^{c}e^{{-}2\sqrt{\omega}v\vert t\vert}\right)$ for all $v>0$ sufficiently small,then there exist natural numbers $N_{1},N_{2}$ satisfying, for any $0<v\ll 1,$ the following estimate
\begin{equation*}
    \Lambda(\varphi_{k})(t,x)=\sum_{i=1}^{N_{1}}s_{i}(v,t)Sym\left(\left(\mathcal{R}_{i}\right)_{k}\right)(t,x)+Rest_{v}(t,x),
\end{equation*}
such that $Rest_{v}(t,x)$ satisfies for some $c_{2}>0$ and any $l\in\mathbb{N}$
\begin{equation*}
    \norm{\partial^{l}_{t}Rest_{v}(t,x)}_{H^{1}(\mathbb{R})}=O\left(v^{2\mathcal{M}+4+l}\left(\ln{\left(\frac{1}{v}\right)+\vert t\vert v}\right)^{c_{2}}e^{{-}2\sqrt{\omega}v\vert t\vert}\right),
\end{equation*}
and for all $1\leq i,j\leq N_{1}$ we have $\left\langle \mathcal{R}_{i},\,\mathcal{R}_{j}\right\rangle=\delta_{i,j},\,\mathcal{R}_{i}\in \mathcal{S}^{+}_{\infty},$  $s_{i,v}\in C^{\infty}(\mathbb{R})$ satisfies, for all $l\in\mathbb{N},$ $\left\vert\frac{\partial^{l}}{\partial t^{l}}s_{i}(v,t)\right\vert\lesssim_{l}v^{2\mathcal{M}}\left[v\vert t \vert+\ln{\left(\frac{1}{v^{2}}\right)}\right]^{n_{\mathcal{M}}}e^{-2\sqrt{\omega}v\vert t\vert},$ for all $t\in\mathbb{R}$ and $v\in(0,1),$ if $0<v\ll 1.$ 
\end{lemma}
\begin{proof}
    The proof of Lemma \ref{represent1} is completely analogous to the proof of Lemma $5.3$ from \cite{second}.
\end{proof}

From now on, to simplify more our notation, we are going to consider $v\in(0,1)$ small enough and denote the functions $\zeta_{k}(v,t),\,v_{k}(v,t),\,\gamma_{k}(v,t),\,f_{\omega,k}(v,t)$ by $\zeta_{k}(t),\,v_{k}(t),\,\gamma_{k}(t),\,f_{\omega,k}(t)$ respectively. 
\par Next, we consider the following result which is going to be essential in the estimate of $\Lambda(u_{k})(t,x)$ with high precision.
\begin{lemma}\label{oddeven}
In notation of Theorem \ref{app lemma}, let $\rho:\mathbb{R}^{2}\to\mathbb{C}$ be the following function
\begin{multline}
    \rho(t,x)=Sym(\phi_{\omega,k})(t,x)+i\sum_{j=1}^{m}f_{odd,j}(v,t)Sym((q_{j})_{k})(t,x)+\sum_{j=1}^{n}f_{even,j}(v,t)Sym((p_{j})_{k})(t,x),
\end{multline}
where all the real functions $q_{j}\left(\frac{\cdot}{\sqrt{\omega}}\right),\,p_{j}\left(\frac{\cdot}{\sqrt{\omega}}\right)$ are in $\mathcal{S}^{+}_{\infty},$ all the real functions $f_{odd,j}$ are smooth and odd on the variable $t,$ all the real functions $f_{even,j}$ are smooth and even on the variable $t,$ and 
\begin{align}\label{odd}
    \left\vert \frac{\partial^{l}}{\partial t^{l}}f_{odd,j}(v,t)\right\vert\lesssim_{l} &v^{2+l}\left(\vert t\vert v+\ln{\frac{1}{v}}\right)^{c_{j}}e^{{-}2\sqrt{\omega}\vert t\vert v},\\ \label{even}
    \left\vert \frac{\partial^{l}}{\partial t^{l}}f_{even,j}(v,t)\right\vert\lesssim_{l} &v^{2+l}\left(\vert t\vert v+\ln{\frac{1}{v}}\right)^{d_{j}}e^{{-}2\sqrt{\omega}\vert t\vert v},
\end{align}
for any $l\in\mathbb{N}$ and $v\in(0,1)$ small enough, and $c_{j},\,d_{j}$ are positive numbers.
Then, if $e^{{-}2\sqrt{\omega}\zeta_{k}}=O(v^{2}),\,v_{k}=O(v),$ for any $l\in\mathbb{N}$
\begin{equation*}
   v^{4}\left\vert \frac{\partial^{l}}{\partial t^{l}}\zeta_{k}(v,t)\right\vert+v^{3}\left\vert \frac{\partial^{l}}{\partial t^{l}}v_{k}(v,t)\right\vert+v^{3}\left\vert \frac{\partial^{l}}{\partial t^{l}}[\gamma_{k}(v,t)-\omega t]\right\vert\leq Cv^{4+l}\left(\vert t\vert v+\ln{\frac{1}{v}}\right)^{c}e^{{-}2\sqrt{\omega}\vert t\vert v},
\end{equation*}
and the function $\zeta_{k}$ is even on $t$ and $v_{k},\,\gamma_{k}$ are odd on $t,$ then, for any $M\in\mathbb{N},$ there exist numbers $M_{k},\,N_{k}\in\mathbb{N},$ and a finite set of Schwartz functions $P_{j},\,Q_{j}\in C^{\infty}_{x}(\mathbb{R},\mathbb{R})$ such that
\begin{equation*}
    \Lambda(\rho_{k})(t,x)=\sum_{j=1}^{M_{k}}ig_{odd,j}(v,t)Sym((Q_{j})_{k})(t,x)+\sum_{j=1}^{N_{k}}g_{even,j}(v,t)Sym((P_{j})_{k})(t,x)+O(v^{2M}),
\end{equation*}
where each real function $g_{odd,j}$ satisfies \eqref{odd}, each real function $g_{even}$ satisfies \eqref{even}, $P{j}\left(\frac{\cdot}{\sqrt{\omega}}\right)$ and $Q_{j}\left(\frac{\cdot}{\sqrt{\omega}}\right)$ belong to $\mathcal{S}^{+}_{\infty},$ and the term $O(v^{2\mathcal{M}})$ means a smooth function $r(t,x)$ satisfying
\begin{equation}\label{remmmm}
    \norm{\frac{\partial^{l}}{\partial t^{l}}r(t,x)}_{H^{s}_{x}(\mathbb{R})}\lesssim_{s,l} v^{2M+l}\left(\vert t\vert v+\ln{\frac{1}{v}}\right)^{\sigma}e^{{-}2\sqrt{\omega}\vert t\vert v},
\end{equation}
for some constant $\sigma>0,$ and any $s\geq 0,$ any $l\in\mathbb{N}.$
\end{lemma}
\begin{proof}
\par First, we can verify by the definition \eqref{trk} that
\begin{align*}
i\frac{\partial}{\partial t} \rho_{k}(t,x)=&{-}ie^{i\left(\gamma_{k}(t)+\frac{v_{k}(t)}{2}\left[x-\frac{\zeta_{k}(t)}{2}\right]\right)}\dot \zeta_{k}(t)\rho^{'}\left(x-\zeta_{k}(t)\right)\\&{-}(\dot \gamma_{k}(t)-\omega)e^{i\left(\gamma_{k}(t)+\frac{v_{k}(t)}{2}\left[x-\frac{\zeta_{k}(t)}{2}\right]\right)}\rho\left(x-\zeta_{k}(t)\right)\\&{-}\frac{\dot v_{k}(t)}{2}e^{i\left(\gamma_{k}(t)+\frac{v_{k}(t)}{2}\left[x-\frac{\zeta_{k}(t)}{2}\right]\right)}(x-\zeta_{k}(t))\rho\left(x-\zeta_{k}(t)\right)\\
&{-}\frac{\dot v_{k}(t) \zeta_{k}(t) }{4}e^{i\left(\gamma_{k}(t)+\frac{v_{k}(t)}{2}\left[x-\frac{\zeta_{k}(t)}{2}\right]\right)}\rho\left(x-\zeta_{k}(t)\right).
\end{align*}
Consequently, using Lemma \ref{invertexpfor} for $\rho=\phi_{\omega},$ and the chain rule of the derivative, we can verify that
$\left[i\partial_{t}+\partial^{2}_{x}\right]\rho_{k}(t,x)+F^{\prime}(\vert \phi_{\omega,k}(t,x)\vert^{2})\phi_{\omega,k}(t,x)-F^{\prime}(\vert \phi_{\omega,k}(t,{-}x)\vert^{2})\phi_{\omega,k}(t,{-}x)$ is equal to the following finite sum 
\begin{equation}\label{Msf}
    i\sum_{j=1}^{N}g_{odd,j}(v,t)Sym((Q_{j})_{k})(t,x)+\sum_{j=1}^{n}g_{even,j}(v,t)Sym((P_{j})_{k})(t,x),
\end{equation}
for some functions $g_{odd,j},\,g_{even,j},\,\mathcal{Q}_{j},\,\mathcal{P}_{j}$ satisfying all the properties of Lemma \ref{oddeven}.
\par Moreover, $F^{'}(\phi^{2})\phi$ is an odd polynomial. Consequently, since $\rho_{k}$ is defined by \eqref{trk} and 
\begin{equation*}
    \alpha_{k}(t,x)-\alpha_{k}(t,{-}x)=v_{k}x,
\end{equation*} 
 we can verify using Lemma \ref{alg} and the hypotheses satisfied by $f_{odd,j}$ and $f_{even,j}$ that $F^{'}(\vert\rho_{k}\vert^{2})\rho_{k}-F^{\prime}(\vert \phi_{\omega,k}(t,x)\vert^{2})\phi_{\omega,k}(t,x)+F^{\prime}(\vert \phi_{\omega,k}(t,{-}x)\vert^{2})\phi_{\omega,k}(t,{-}x)$ is a finite sum of terms having one of the following forms
\begin{align*}
    i \alpha_{odd,j}(t)Sym\Big(&(\mathcal{Q}_{1})_{k}\Big)(t,x),\\
    \alpha_{even,j}(t)Sym\Big(&(\mathcal{Q}_{1})_{k}\Big)(t,x),\\
    i\alpha_{odd}(t) \Big[&\mathcal{Q}_{1}(x-\zeta_{k}(t))\mathcal{Q}_{2}(x+\zeta_{k}(t))e^{i\alpha_{k}(t,x)}-\mathcal{Q}_{1}({}-x-\zeta_{k}(t))\mathcal{Q}_{2}({-}x+\zeta_{k}(t))e^{i\alpha_{k}(t,{-}x)}\Big],\\
    \alpha_{even}(t) \Big[&\mathcal{Q}_{1}(x-\zeta_{k}(t))\mathcal{Q}_{2}(x+\zeta_{k}(t))e^{i\alpha_{k}(t,x)}-\mathcal{Q}_{1}({}-x-\zeta_{k}(t))\mathcal{Q}_{2}({-}x+\zeta_{k}(t))e^{i\alpha_{k}(t,{-}x)}\Big],\\
i\alpha_{odd}(t)\Big[&\mathcal{Q}_{1}(x-\zeta_{k}(t))\mathcal{Q}_{2}(x+\zeta_{k}(t))e^{i\alpha_{k}(t,x)}\cos{(mv_{k}x)}\\&{-}\mathcal{Q}_{1}(-x-\zeta_{k}(t))\mathcal{Q}_{2}({-}x+\zeta_{k}(t))e^{i\alpha_{k}(t,{-}x)}\cos{(mv_{k}x)}\Big],\\
\alpha_{even}(t)\Big[&\mathcal{Q}_{1}(x-\zeta_{k}(t))\mathcal{Q}_{2}(x+\zeta_{k}(t))e^{i\alpha_{k}(t,x)}\cos{(mv_{k}x)}\\&{-}\mathcal{Q}_{1}(-x-\zeta_{k}(t))\mathcal{Q}_{2}({-}x+\zeta_{k}(t))e^{i\alpha_{k}(t,{-}x)}\cos{(mv_{k}x)}\Big],\\
\alpha_{odd}(t)\Big[&\mathcal{Q}_{1}(x-\zeta_{k}(t))\mathcal{Q}_{2}(x+\zeta_{k}(t))e^{i\alpha_{k}(t,x)}\sin{(mv_{k}x)}\\&{+}\mathcal{Q}_{1}(-x-\zeta_{k}(t))\mathcal{Q}_{2}({-}x+\zeta_{k}(t))e^{i\alpha_{k}(t,{-}x)}\sin{(mv_{k}x)}\Big],\\
i\alpha_{even}(t)\Big[&\mathcal{Q}_{1}(x-\zeta_{k}(t))\mathcal{Q}_{2}(x+\zeta_{k}(t))e^{i\alpha_{k}(t,x)}\sin{(mv_{k}x)}\\&{+}\mathcal{Q}_{1}(-x-\zeta_{k}(t))\mathcal{Q}_{2}({-}x+\zeta_{k}(t))e^{i\alpha_{k}(t,{-}x)}\sin{(mv_{k}x)}\Big]
\end{align*}
such that $m$ and $l$ are integers, $\mathcal{Q}_{1}\left(\frac{\cdot}{\sqrt{\omega}}\right)\in \mathcal{S}^{+}_{\infty},\,\mathcal{Q}_{2}\left(\frac{\cdot}{\sqrt{\omega}}\right)\in \mathcal{S}^{-}_{\infty},$ $\alpha_{even}$ is an even function, $\alpha_{odd}$ is an odd function both satisfying estimates \eqref{odd} and \eqref{even} respectively. 
\par Furthermore, the functions $\cos,\,\sin$ are entire and we have
\begin{align*}
\cos{\left(mv_{k}x\right)}=&\cos{\left(mv_{k}(x-\zeta_{k})\right)}\cos{\left(mv_{k}\zeta_{k}\right)}-\sin{\left(mv_{k}(x-\zeta_{k})\right)}\sin{\left(mv_{k}\zeta_{k}\right)},\\
\sin{\left(mv_{k}x\right)}=&\sin{\left(mv_{k}(x-\zeta_{k})\right)}\cos{\left(mv_{k}\zeta_{k}\right)}{+}\cos{\left(mv_{k}(x-\zeta_{k})\right)}\sin{\left(mv_{k}\zeta_{k}\right)}
\end{align*}
\par Consequently, for any $M\in\mathbb{N}, $we can verify from
Lemma \ref{separation} and Taylor's Expansion Theorem for the functions $\sin,\,\cos$ that $F^{'}(\vert \rho_{k}(t,x) \vert^{2})\rho_{k}(t,x)$ is equal to a expression of the form \eqref{Msf} plus a remainder $r(t,x)$ satisfying \eqref{remmmm}.
\end{proof}
From now on, we consider\begin{equation}\label{phiomegak}
    \phi_{\omega,k}(t,x)\coloneqq \phi_{\omega}(x-\zeta_{k}(t))e^{i\frac{v_{k}(t)}{2}\left(x-\frac{\zeta_{k}(t)}{2}\right)}e^{i\gamma_{k}(t)}.
\end{equation}We also use the following notation 
\begin{equation}
    \alpha_{k}(t)\coloneqq\frac{v_{k}(t)}{2}\left[x-\frac{\zeta_{k}(t)}{2}\right]+\gamma_{k}(t).
\end{equation}
Using the identity
\begin{equation*}
  \phi^{''}_{\omega}-\omega \phi_{\omega}+F^{'}(\phi^{2}_{\omega})\phi_{\omega}=0,
\end{equation*}
it is not difficult to verify that
\begin{align}\label{lambdasol}
    \Lambda\left(\phi_{\omega,k}\right)(t,x)=&i\dot v_{k}(t)\partial_{v_{k}}\phi_{\omega,k}+i\left(\dot\zeta_{k}(t)-v_{k}(t)\right)\partial_{\zeta_{k}}\phi_{\omega,k}\\ \nonumber &{+}i\left(\dot\gamma_{k}(t)-\omega\right)\partial_{\gamma_{k}}\phi_{\omega,k}.
\end{align}
Next, before we start to prove Theorem \ref{app lemma}, we are going to consider the following proposition which is equivalent to the Theorem \ref{app lemma} when $k=1.$
\begin{lemma}\label{k=1}
There exist constants $C,\,\sigma>0\in\mathbb{R}$ such that for $c_{\omega}=\frac{\sqrt{C}}{\omega^{\frac{1}{4}}}$ if $0<v\ll 1$ and $d$ is the unique solution of
\begin{equation}\label{ODE}
\begin{cases}
    \ddot d(t)=Ce^{-2\sqrt{\omega} d(t)},\\
    \lim_{t\to{+}\infty}\left\vert d(t)-vt-\frac{1}{\sqrt{\omega}} \ln{\frac{c_{\omega}}{v}}\right\vert=0,\, \lim_{t\to{+}\infty}\left\vert \dot d(t)-v \right\vert=0,
\end{cases}
\end{equation}
then
\begin{equation}\label{u11}
   u_{1}(t,x)= e^{i\omega t}e^{\frac{i\dot d(t)}{2}\left(x-\frac{d(t)}{2}\right)}\phi_{\omega}(x-d(t))-e^{i\omega t}e^{{-}\frac{i\dot d(t) }{2}\left(x+\frac{d(t)}{2}\right)}\phi_{\omega}(x+d(t))
\end{equation}
satisfies 
\begin{align*}
\norm{\frac{\partial^{l}}{\partial t^{l}}\Lambda(u_{1})(t,x)}_{H^{s}_{x}(\mathbb{R})}\leq & C(s,l)v^{2+l}e^{{-}2\sqrt{\omega}\vert t\vert v},\\ \left\vert\left\langle\frac{\partial^{l}}{\partial t^{l}}\Lambda(u_{1})(t,x), \phi^{\prime}_{\omega}(x\pm d(t))e^{i\alpha(t,\mp x)}\right\rangle\right\vert\leq & C(l)v^{4+l}(1+\vert t\vert v)^{\sigma}e^{{-}2\sqrt{\omega}\vert t\vert v}.
\end{align*}
\end{lemma}
\begin{proof}[Proof of Lemma \ref{k=1}]
First, to simplify the notation used in the proof, we consider for all $(t,x)\in\mathbb{R}^{2}$
\begin{equation*}
    \alpha(t,x)=\omega t+\frac{\dot d(t)}{2}\left(x-\frac{d(t)}{2}\right),
\end{equation*}
using the identity
\begin{equation*}
    \phi_{\omega}^{''}(x)-\omega \phi_{\omega}(x)+F^{'}\left(\phi_{\omega}(x)^{2}\right)\phi_{\omega}(x)=0,
\end{equation*}
we can verify that
\begin{align}\label{lambdau}
    \Lambda\left(u_{1}\right)(t,x)=&-\frac{\ddot d(t)}{2}\left[\left(x-\frac{d(t)}{2}\right)\phi_{\omega}\left(x-d(t)\right)e^{i\alpha(t,x)}+\left(x+\frac{d(t)}{2}\right)\phi_{\omega}\left(x+d(t)\right)e^{i\alpha(t,{-}x)}\right]\\ \nonumber
    &{+}F^{'}\left(\vert u_{1}(t,x)\vert^{2}\right)u_{1}(t,x)-F^{'}\left(\phi_{\omega}(x-d(t))^{2}\right)\phi_{\omega}(x-d(t))e^{i\alpha(t,x)}\\ \nonumber &{+}F^{'}\left(\phi_{\omega}(x+d(t))^{2}\right)\phi_{\omega}(x+d(t))e^{i\alpha(t,{-}x)}.
\end{align}
The ordinary differential equation that $d$ will satisfy is similar to the one obtained in the statement of Lemma $3.1$ from \cite{holmerlin}. Indeed, the ordinary differential equation obtained in \cite{holmerlin} is equivalent to \eqref{ODE} when $F^{'}(z)\equiv z.$
\par Before the construction of the ordinary differential equation \eqref{ODE}, we need to estimate the contribution coming from the nonlinearity of \eqref{NLS3}. From Taylor's Expansion Theorem, we deduce that
\begin{multline}\label{interF}
    F^{'}\left(\vert u_{1}(t,x)\vert^{2}\right)u_{1}(t,x)-F^{'}\left(\phi_{\omega}(x-d(t))^{2}\right)\phi_{\omega}(x-d(t))e^{i\alpha(t,x)}+F^{'}\left(\phi_{\omega}(x+d(t))^{2}\right)\phi_{\omega}(x+d(t))e^{i\alpha(t,{-}x)}\\
    \begin{aligned}
    =&{-}2\cos{(\dot d(t)x)}F^{''}\left(\phi_{\omega}(x+d(t))^{2}+\phi_{\omega}(x-d(t))^{2}\right)u_{1}(t,x)\phi_{\omega}(x-d(t))\phi_{\omega}(x+d(t))
    \\&{+}F^{'}\left(\phi_{\omega}(x+d(t))^{2}+\phi_{\omega}(x-d(t))^{2}\right)u_{1}(t,x)-F^{'}\left(\phi_{\omega}(x-d(t))^{2}\right)\phi_{\omega}(x-d(t))e^{i\alpha(t,x)}\\&{+}F^{'}\left(\phi_{\omega}(x+d(t))^{2}\right)\phi_{\omega}(x+d(t))e^{i\alpha(t,{-}x)}\\
    &{+}O\left(\phi_{\omega}(x+d(t))^{2}\phi_{\omega}(x-d(t))^{2}\right),
\end{aligned}
\end{multline}
where the expression $O\left( \phi_{\omega}(x+d(t))^{2}\phi_{\omega}(x-d(t))^{2}\right)$ means a smooth function $p(t,x)$ satisfying
\begin{equation*}
    \vert p(t,x) \vert\lesssim \ \phi_{\omega}(x+d(t))^{2}\phi_{\omega}(x-d(t))^{2}.
\end{equation*}
Moreover, using estimate \eqref{taylor1} from Lemma \ref{taylorlemma}, we obtain from \eqref{interF} that %the mistake is here
\begin{multline}\label{interF0}
    F^{'}\left(\vert u_{1}(t,x)\vert^{2}\right)u_{1}(t,x)-F^{'}\left(\phi_{\omega}(x-d(t))^{2}\right)\phi_{\omega}(x-d(t))e^{i\alpha(t,x)}
    {+}F^{'}\left(\phi_{\omega}(x+d(t))^{2}\right)\phi_{\omega}(x+d(t))e^{i\alpha(t,{-}x)}
    \\
    \begin{aligned}
    =&{-}2\cos{(\dot d(t)x)}F^{''}\left(\phi_{\omega}(x+d(t))^{2}+\phi_{\omega}(x-d(t))^{2}\right)u_{1}(t,x)\phi_{\omega}(x-d(t))\phi_{\omega}(x+d(t))
    \\&{-}F^{'}\left(\phi_{\omega}(x-d(t))^{2}\right)\phi_{\omega}(x+d(t))e^{i\alpha(t,{-}x)}+F^{'}\left(\phi_{\omega}(x+d(t))^{2}\right)\phi_{\omega}(x-d(t))e^{i\alpha(t,x)}
   \\
    &{+}O\left(\phi_{\omega}(x+d(t))^{2}\phi_{\omega}(x-d(t))^{2}\right),
\end{aligned}
\end{multline}
\par Therefore, from Lemma \ref{interactt}, identity \eqref{lambdau}, estimate \eqref{interF0} and the inequalities 
\begin{align*}
    \left\vert \cos{\left(\dot d(t) x\right)}-1 \right\vert\lesssim \vert \dot d(t) \vert^{2} \vert x\vert^{2},\,\,\left\vert F^{'}(\phi_{\omega}(x)^{2})\right\vert\lesssim \left\vert \phi_{\omega}(x)^{2}\right\vert,
\end{align*}
we obtain using inequality \eqref{taylor1} of Lemma \ref{taylorlemma} that there exists $\sigma>0$ satisfying
\begin{multline*}
  \begin{aligned}
  \left\langle\Lambda(u_{1})(t,x),\, \phi^{'}_{\omega}(x-d(t))e^{i\alpha(t,x)}\right\rangle=&\frac{\ddot d(t)}{4}\left[\norm{\phi_{\omega}}_{L^{2}}^{2}+O\left((1+\vert d(t) \vert)^{2}e^{{-}2\sqrt{\omega}d(t)}\right)\right]
\\&{-}2\int_{\mathbb{R}}F^{''}\left(\phi_{\omega}(x)^{2}\right)\phi_{\omega}(x)^{2}\phi^{'}_{\omega}(x)\phi_{\omega}(x+2d)\,dx\\
&{-}\int_{\mathbb{R}}F^{'}\left(\phi_{\omega}(x)^{2}\right)\phi^{'}_{\omega}(x)\phi_{\omega}(x+2d)\,dx\\
  &{+}O\left(e^{{-}4\sqrt{\omega}d}\left(1+\vert d\vert\right)^{\sigma}+\dot d^{2}e^{{-}2\sqrt{\omega}d}\left(1+\vert d\vert\right)^{\sigma}\right).
\end{aligned}
\end{multline*}
In conclusion, we obtain using integration by parts that
\begin{multline}\label{lambdauphi1}
  \begin{aligned}
  \left\langle\Lambda(u_{1})(t,x),\, \phi^{'}_{\omega}(x-d(t))e^{i\alpha(t,x)}\right\rangle =&\frac{\ddot d(t)}{4}\left[\norm{\phi_{\omega}}_{L^{2}}^{2}+O\left((1+\vert d(t) \vert)^{2}e^{{-}2\sqrt{\omega}d(t)}\right)\right]\\&{+}\int_{\mathbb{R}}F^{'}\left(\phi_{\omega}(x)^{2}\right)\phi_{\omega}(x)\phi^{'}_{\omega}(x+2d)\,dx\\
  &{+}O\left(e^{{-}4\sqrt{\omega}d}\left(1+\vert d\vert\right)^{\sigma}+\dot d^{2}e^{{-}2\sqrt{\omega}d}\left(1+\vert d\vert\right)^{\sigma}\right).
  \end{aligned}
\end{multline}
for a constant $\sigma>0.$
\par Furthermore, since $\phi_{\omega}\left(\frac{\cdot}{\sqrt{\omega}}\right)\in\mathcal{S}^{+},$ we have for the $a_{{+}\infty}>0$ defined on Remark \ref{asyrem}
\begin{equation*}
    \left\vert \phi^{'}_{\omega}(x)+a_{{+}\infty}\sqrt{\omega}e^{{-}\sqrt{\omega}x} \right\vert\lesssim \min\left(e^{{-}\sqrt{\omega}x},\,e^{{-}2\sqrt{\omega}x}\right).
\end{equation*}
Therefore, we can deduce from the estimate \eqref{lambdauphi1} that
\begin{multline*}
\begin{aligned}
  \left\langle\Lambda(u_{1})(t,x),\, \phi^{'}_{\omega}(x-d(t))e^{i\alpha(t,x)}\right\rangle=&\frac{\ddot d(t)}{4}\left[\norm{\phi_{\omega}}_{L^{2}}^{2}+O\left((1+\vert d(t) \vert)^{2}e^{{-}2\sqrt{\omega}d(t)}\right)\right]\\&{-}a_{{+}\infty}\sqrt{\omega}e^{{-}2\sqrt{\omega}d}\int_{\mathbb{R}}F^{'}\left(\phi_{\omega}(x)^{2}\right)\phi_{\omega}(x)e^{{-}\sqrt{\omega}x}\,dx\\
  &{+}O\left(e^{{-}4\sqrt{\omega}d}\left(1+\vert d\vert\right)^{\sigma}+\dot d^{2}e^{{-}2\sqrt{\omega}d}\left(1+\vert d\vert\right)^{\sigma}\right),
  \end{aligned}
\end{multline*}
\par Concerning the estimate above, we choose the following ordinary differential equation
\begin{equation}\label{d edo}
    \ddot d(t)=\frac{4a_{{+}\infty}}{\norm{\phi_{\omega}}_{L^{2}}^{2}}\sqrt{\omega}e^{{-}2\sqrt{\omega}d}\int_{\mathbb{R}}F^{'}\left(\phi_{\omega}(x)^{2}\right)\phi_{\omega}(x)e^{{-}\sqrt{\omega}x}\,dx=Ce^{{-}2\sqrt{\omega}d}.
\end{equation}
Moreover, since $\phi_{\omega}^{''}-\omega \phi_{\omega}={-}F^{'}(\phi_{\omega}^{2})\phi_{\omega},$ we have from integration by parts that
\begin{equation*}
    \int_{\mathbb{R}}F^{'}\left(\phi_{\omega}(x)^{2}\right)\phi_{\omega}(x)e^{{-}\sqrt{\omega}x}\,dx=\lim_{x\to{-}\infty}\phi^{'}_{\omega}(x)e^{{-}\sqrt{\omega}x}+\sqrt{\omega}\phi_{\omega}(x)e^{{-}\sqrt{\omega}x}=2a_{{+}\infty}\sqrt{\omega},
\end{equation*}
so $C$ is positive and depends only on $\omega.$
Indeed, the function 
\begin{equation}\label{d0}
    d(t)=\frac{1}{\sqrt{\omega}}\ln{\left[\frac{\sqrt{C}\left(\cosh{\left(\sqrt{\omega}vt\right)}\right)}{\omega^{\frac{1}{4}}v}\right]},
\end{equation}
satisfies \eqref{ODE} for $c_{\omega}=\frac{\sqrt{C}}{\omega^{\frac{1}{4}}}.$ Similarly to the proof Lemma $3.1$ from \cite{second}, we can verify by induction that if $v>0$ is small enough, then
\begin{equation}\label{decd}
    \vert \dot d(t) \vert=O(v),\,\left\vert d^{(l)}(t) \right\vert\lesssim_{l} v^{l} e^{{-}2\sqrt{\omega}v\vert t\vert} \text{ for all $l\geq 2.$} 
\end{equation}
\par In conclusion, similarly to the proof of Theorem $4.1$ from \cite{second}, we can verify using Lemma \ref{interactt}, identity \eqref{lambdau} and estimates \eqref{decd} that $u_{1}$ defined by \eqref{u11} satisfies Theorem \ref{app lemma} for $k=0.$
\end{proof}
\begin{remark}\label{ffff}
 Furthermore, applying Lemmas \ref{interactt}, \ref{separation} in estimate \eqref{interF0}, we can verify that
 \begin{multline}\label{interF00}
    F^{'}\left(\vert u_{1}(t,x)\vert^{2}\right)u_{1}(t,x)-F^{'}\left(\phi_{\omega}(x-d(t))^{2}\right)\phi_{\omega}(x-d(t))e^{i\alpha(t,x)}
    {+}F^{'}\left(\phi_{\omega}(x+d(t))^{2}\right)\phi_{\omega}(x+d(t))e^{i\alpha(t,{-}x)}
    \\
    \begin{aligned}
=&{-}2a_{{+}\infty}e^{{-}2\sqrt{\omega} d(t)} Sym\left[F^{''}\left(\phi_{\omega}(\cdot-d(t))^{2}\right)\phi_{\omega}(\cdot-d(t))^{2}e^{i\alpha(t,\cdot)}e^{{-}\sqrt{\omega}(\cdot-d(t))}\right](x)
    \\&{-}a_{{+}\infty}e^{{-}2\sqrt{\omega}d(t)}Sym\left[F^{'}\left(\phi_{\omega}(\cdot-d(t))^{2}\right)e^{{-}\sqrt{\omega}(\cdot-d(t))}e^{i\alpha(t,{-}\cdot)}\right](x) \\
&{+}O\left(\phi_{\omega}(x+d(t))^{2}\phi_{\omega}(x-d(t))^{2}\right)\\
&=Forc(t,x)+O\left(\phi_{\omega}(x+d(t))^{2}\phi_{\omega}(x-d(t))^{2}\right).
\end{aligned}
\end{multline}
Consequently, for 
\begin{equation*}
    \alpha_{0}(t,x)=\gamma_{1}(t)+\frac{\dot d(t)}{2}\left(x-\frac{d(t)}{2}\right)
\end{equation*}
we can verify that similarly to the proof of estimate \eqref{lambdau} for any smooth functions $f_{\omega,{1}}(t),\,f_{\zeta}$ satisfying
\begin{align}\label{standeasy}
    \left\vert \frac{d^{l}}{dt^{l}}f_{\zeta}(t) \right\vert+\left\vert \frac{d^{l}}{dt^{l}}f_{\omega,{1}}(t) \right\vert=O\left(v^{2+l}\left(\vert t\vert v+\ln{\frac{1}{v}}\right)^{\sigma}e^{{-}\sqrt{2}\vert t\vert v}\right)
\end{align} that
\begin{align*}
u_{2}(t,x)=&e^{i\alpha_{0}(t,x)}\phi_{\omega}(x-d(t)-f_{\zeta}(t))-e^{i\alpha_{0}(t,{-}x)}\phi_{\omega}(x+d(t)+f_{\zeta}(t))\\&{+}f_{\omega,1}(t)\left[e^{i\alpha_{0}(t,x)}\partial_{\omega}\phi_{\omega}(x-d(t)-f_{\zeta}(t))-e^{i\alpha_{0}(t,{-}x)}\partial_{\omega}\phi_{\omega}(x+d(t)+f_{\zeta}(t))\right]
\end{align*}
satisfies
\begin{multline}\label{main22}
    \Lambda(u_{2})(t,x)\\
    \begin{aligned}
        =&{-}\frac{\ddot d(t)}{2}Sym\left[\left(\cdot-\frac{d(t)}{2}\right)\phi_{\omega}\left(\cdot-d(t)-f_{\zeta}(t)\right)e^{i\alpha_{0}(t,\cdot)}\right](x)\\&{+}i\dot f_{\omega,1}(t)Sym\left[\partial_{\omega}\phi_{\omega}(\cdot-d(t)-f_{\zeta}(t))e^{i\alpha_{0}(t,\cdot)}\right](x)\\
        &{-}f_{\omega,1}(t)Sym\left[S_{\omega}(\partial_{\omega}\phi_{\omega})(\cdot-d(t)-f_{\zeta}(t))e^{i\alpha_{0}(t,\cdot)}\right](x)
        {-}(\dot \gamma_{1}(t)-\omega)Sym\left[\phi_{\omega}(\cdot-d(t)-f_{\zeta}(t))e^{i\alpha_{0}(t,\cdot)}\right](x)\\
&{-}i\dot f_{\zeta}(t)Sym\left[\phi^{\prime}_{\omega}(\cdot-d(t)-f_{\zeta}(t))e^{i\alpha_{0}(t,\cdot)}\right](x)+Forc(t,x)+O\left(v^{4}\left(\vert t\vert v+\ln{\frac{1}{v}}\right)^{\sigma}e^{{-}\sqrt{2}\vert t\vert v}\right),
    \end{aligned}
\end{multline}
where the remainder above means a smooth function $p(t,x)$ satisfying for any $l\in\mathbb{N}$
\begin{equation*}
\norm{\partial_{l}p(t,x)}_{H^{1}}=O\left(v^{4+l}\left(\vert t\vert v+\ln{\frac{1}{v}}\right)^{\sigma}e^{{-}\sqrt{2}\vert t\vert v}\right),
\end{equation*}
the estimates above follow from Lemmas \ref{interactt} and \ref{separation}.
Considering the following ordinary differential system
\begin{align*}
    \dot f_{\omega,1}(t)\langle\partial_{\omega}\phi_{\omega},\phi_{\omega}\rangle=&{-}a_{{+}\infty}e^{{-}2\sqrt{\omega}d(t)}\dot d(t)\left\langle F^{\prime}(\phi_{\omega}(x)^{2})e^{{-}\sqrt{\omega}x},x\phi_{\omega}(x) \right\rangle\\&{-}a_{{+}\infty}d(t)\dot d(t)e^{{-}2\sqrt{\omega}d(t)}\left\langle F^{\prime}(\phi_{\omega}(x)^{2})e^{{-}\sqrt{\omega}x},\phi_{\omega} \right\rangle\\
    (\dot\gamma_{1}(t)-\omega+\frac{\ddot d(t)d(t)}{4})\left\langle\partial_{\omega}\phi_{\omega},\phi_{\omega}\right\rangle=&{-}f_{\omega,1}(t)\left\langle S_{\omega}(\partial_{\omega}\phi_{\omega}),\phi_{\omega}\right\rangle\\&{-}2a_{{+}\infty}e^{{-}2\sqrt{\omega}d(t)}\left\langle F^{''}(\phi_{\omega}(x)^{2})\phi_{\omega}(x)^{2}e^{{-}\sqrt{\omega}x},\partial_{\omega}\phi_{\omega} \right\rangle\\
&{-}a_{{+}\infty}e^{{-}2\sqrt{\omega}d(t)}\left\langle F^{\prime}(\phi_{\omega}(x)^{2})e^{{-}\sqrt{\omega}x},\partial_{\omega}\phi_{\omega}\right\rangle
    ,\\
    \frac{1}{2}\dot f_{\zeta}(t)\norm{\phi_{\omega}}_{L^{2}}^{2}=&{-}a_{{+}\infty}\dot d(t)e^{{-}2\sqrt{\omega}d(t)}\left\langle F^{ \prime}(\phi_{\omega}(x)^{2})e^{{-}\sqrt{\omega}x},x^{2}\phi_{\omega}(x)\right\rangle\\&{-}a_{{+}\infty}d(t)\dot d(t)e^{{-}2\sqrt{\omega}d(t)}\left\langle F^{ \prime}(\phi_{\omega}(x)^{2})e^{{-}\sqrt{\omega}x}, x\phi_{\omega}(x)\right\rangle,\\
\lim_{t\to{+}\infty}f_{\omega,1}(t)=\lim_{t\to{+}\infty}f_{\zeta}(t)=&
\lim_{t\to{+}\infty} \gamma_{1}(t)-\omega t=0,
\end{align*}
 it is not difficult to verify that \eqref{standeasy} holds for some $\sigma,$ because of \eqref{decd} and the Fundamental Theorem of Calculus.
 \par Furthermore,
 since we have from the estimate \eqref{main22} that
\begin{align*}
    \left\langle \Lambda(u_{2})(t,x),i\phi_{\omega}(x-d(t)-f_{\zeta}(t))e^{i\alpha_{0}(t,x)}\right\rangle=&\dot f_{\omega,1}(t)\langle\partial_{\omega}\phi_{\omega},\phi_{\omega}\rangle\\&{+}\langle Forc(t,x),i\phi_{\omega}(x-d(t))e^{i\alpha_{0}(t,x)} \rangle\\&{+}O\left(v^{4}\left(\vert t\vert v+\ln{\frac{1}{v}}\right)^{\sigma}e^{{-}2\sqrt{\omega}\vert t\vert v}\right),\\
    \left\langle \Lambda(u_{2})(t,x),\partial_{\omega}\phi_{\omega}(x-d(t)-f_{\zeta}(t))e^{i\alpha_{1}(t,x)}\right\rangle=& {-}(\dot\gamma_{1}(t)-\omega+\frac{\ddot d(t)d(t)}{4})\left\langle\partial_{\omega}\phi_{\omega},\phi_{\omega}\right\rangle\\&{-}f_{\omega,1}(t)\left\langle S_{\omega}(\partial_{\omega}\phi_{\omega}),\phi_{\omega}\right\rangle\\
    &{+}\left\langle Forc(t,x),\partial_{\omega}\phi_{\omega}(x-d(t)-f_{\zeta}(t))e^{i\alpha_{0}(t,x)} \right\rangle\\
    &{+}O\left(v^{4}\left(\vert t\vert v+\ln{\frac{1}{v}}\right)^{\sigma}e^{{-}2\sqrt{\omega}\vert t\vert v}\right),\\
    \left\langle \Lambda(u_{2})(t,x),i(x-d(t)-f_{\zeta}(t))\phi_{\omega}(x-d(t)-f_{\zeta}(t))e^{i\alpha_{0}(t,x)}\right\rangle=&\frac{1}{2}\dot f_{\zeta}(t)\norm{\phi^{\prime}_{\omega}}_{L^{2}}^{2}
    \\
    &{+}\left\langle Forc(t,x),i(x-d(t))\phi_{\omega}(x-d(t))e^{i\alpha_{0}(t,x)} \right\rangle\\
    &{+}O\left(v^{4}\left(\vert t\vert v+\ln{\frac{1}{v}}\right)^{\sigma}e^{{-}2\sqrt{\omega}\vert t\vert v}\right),
\end{align*}
we can verify from the ordinary differential system above and Lemma \ref{k=1} that
\begin{align*}
    \left\vert \frac{d^{l}}{dt^{l}}\left\langle \Lambda(u_{2})(t,x),i(x-d(t))\phi_{\omega}(x-d(t))e^{i\alpha_{0}(t,x)}\right\rangle\right\vert=&O\left(v^{4+l}\left(\vert t\vert v+\ln{\frac{1}{v}}\right)^{\sigma}e^{{-}2\sqrt{\omega}\vert t\vert v}\right),\\
    \left\vert \frac{d^{l}}{dt^{l}}\left\langle \Lambda(u_{2})(t,x),i\phi_{\omega}(x-d(t))e^{i\alpha_{0}(t,x)}\right\rangle\right\vert=&O\left(v^{4+l}\left(\vert t\vert v+\ln{\frac{1}{v}}\right)^{\sigma}e^{{-}2\sqrt{\omega}\vert t\vert v}\right),\\
    \left\vert \frac{d^{l}}{dt^{l}}\left\langle \Lambda(u_{2})(t,x),\partial_{\omega}\phi_{\omega}(x-d(t))e^{i\alpha_{0}(t,x)}\right\rangle\right\vert=&O\left(v^{4+l}\left(\vert t\vert v+\ln{\frac{1}{v}}\right)^{\sigma}e^{{-}2\sqrt{\omega}\vert t\vert v}\right),\\
    \left\vert \frac{d^{l}}{dt^{l}}\left\langle \Lambda(u_{2})(t,x),\phi^{\prime}_{\omega}(x-d(t))e^{i\alpha_{0}(t,x)}\right\rangle\right\vert=&O\left(v^{4+l}\left(\vert t\vert v+\ln{\frac{1}{v}}\right)^{\sigma}e^{{-}2\sqrt{\omega}\vert t\vert v}\right),
\end{align*}
for any $l\in\mathbb{N}.$
\end{remark}
\begin{proof}[Proof of Theorem \ref{app lemma}]

\par \textbf{Step 1.}(Construction of $\varphi_{1}.$) In notation of Lemma \ref{k=1} and Remark \ref{ffff}, we consider
\begin{equation*}
    \alpha_{1}(t,x)=\gamma_{1}(t)+\frac{\dot d(t)+\dot f_{\zeta}(t)}{2}\left(x-\frac{d(t)+f_{\zeta}(t)}{2}\right),
\end{equation*}
and we recall the space translation function $\tau_{d+f_{\zeta}}$ defined by
\begin{equation*}
    \tau_{d+f_{\zeta}}\rho(t,x)\coloneqq \rho(t,x-d-f_{\zeta}).
\end{equation*}
In notation of Lemma \ref{ProP}, we consider $\Pi^{\perp}\coloneqq Id -\Pi,$ which is the projection in the orthogonal complement of $\{\phi_{\omega},i\phi^{'}_{\omega},\,i\partial_{\omega}\phi_{\omega},x\phi_{\omega}\}$ in $L^{2}(\mathbb{R},\mathbb{C}).$ Moreover, we consider
\begin{align*}
    \varphi_{1,0}(t,x)=&e^{i\alpha_{1}(t,x)}\phi_{\omega}(x-d(t)-f_{\zeta}(t))-e^{i\alpha_{1}(t,{-}x)}\phi_{\omega}(x+d(t)+f_{\zeta}(t))\\&{+}f_{\omega,1}(t)\left[e^{i\alpha_{1}(t,x)}\partial_{\omega}\phi_{\omega}(x-d(t)-f_{\zeta}(t))-e^{i\alpha_{1}(t,{-}x)}\partial_{\omega}\phi_{\omega}(x+d(t)+f_{\zeta}(t))\right]
\end{align*}
From Lemma \ref{k=1} and Remark \ref{ffff}, we obtain for any $l\in \mathbb{N}$ that there exists $\sigma>0$ satisfying
\begin{equation*}
\norm{\partial^{l}_{t}\Lambda(\varphi_{1,0})(t,x)}_{H^{1}} \lesssim v^{2+l}\left(\vert t\vert v+\ln{\frac{1}{v}}\right)^{\sigma}e^{{-}2\sqrt{\omega}\vert t\vert v}.
\end{equation*}
\par Furthermore, using Remark \ref{ffff}, estimates \eqref{standeasy} and Lemma \ref{oddeven}, 
it is not difficult to verify that there exist real functions $p_{j}\in\mathcal{S}^{+}_{\infty},$ and real functions $f_{even,j},\,f_{odd,j}$ satisfying

\begin{align}\label{pp1}
\Lambda(\varphi_{1,0}(t,x))=&\sum_{j\in K_{0}}f_{even,j}(t)Sym\left[\tau_{d(t)+f_{\zeta}(t)}p_{j}(\cdot)e^{i\alpha_{1}(t,\cdot)}\right](x)\\ \nonumber &{+}\sum_{j\in M_{0}}f_{odd,j}(t)Sym\left[\tau_{d(t)+f_{\zeta}(t)}ip_{j}(\cdot)e^{i\alpha_{1}(t,\cdot)}\right](x)\\ \nonumber
&{+}O\left(v^{20}\left(\ln{\frac{1}{v}}+\vert t\vert v\right)^{\sigma}e^{{-}2\sqrt{\omega}\vert t\vert v}\right),
\end{align}
such that $\vert \frac{d^{l}}{dt^{l}}f_{even,j}(t) \vert+\vert\frac{d^{l}}{dt^{l}}f_{odd,j}(t) \vert\lesssim_{l}v^{2+l}\left(\vert t\vert v+\ln{\frac{1}{v}}\right)^{\sigma}e^{{-}2\sqrt{\omega}\vert t\vert v},$ and
\begin{equation*}
    \norm{\sum_{j\in K_{0}}\frac{d^{l}}{dt^{l}}f_{even,j}(t)\Pi\left[p_{j}\right](x)+\sum_{j\in M_{0}}\frac{d^{l}}{dt^{l}}f_{odd,j}(t)\Pi\left[ip_{j}\right](x)}_{H^{1}}\lesssim_{l}v^{4+l}\left(\vert t\vert v+\ln{\frac{1}{v}}\right)^{\sigma}e^{{-}2\sqrt{\omega}\vert t\vert v},
\end{equation*}
because of Remark \ref{ffff}.
\par Next, using Lemma \ref{inverhol}, we can consider the following additional terms
\begin{align*}
   Corr_{1,1}(t,x)=& \sum_{j\in K_{0}}f_{even,j}(t)Sym\left[e^{\alpha_{1}(t,\cdot)}\tau_{d(t)+f_{\zeta}}S^{{-}1}_{\omega}\left\{\Pi^{\perp}\left[p_{j}\right]\right\}(\cdot)\right](x)\\&{+}\sum_{j\in M_{0}}f_{odd,j}(t)Sym\left[e^{\alpha_{1}(t,\cdot)}\tau_{d(t)+f_{\zeta}}S^{{-}1}_{\omega}\left\{\Pi^{\perp}\left[ip_{j}\right]\right\}(\cdot)\right](x),
\end{align*}
 and
 \begin{align*}
   Corr_{2,1}(t,x)=& \sum_{j\in K_{0}}\dot f_{even,j}(t)Sym\left[e^{\alpha_{1}(t,\cdot)}\tau_{d(t)+f_{\zeta}}S^{{-}1}_{\omega}i S^{{-}1}_{\omega}\left\{\Pi^{\perp}\left[p_{j}\right]\right\}(\cdot)\right](x)\\&{+}\sum_{j\in M_{0}}\dot f_{odd,j}(t)Sym\left[e^{\alpha_{1}(t,\cdot)}\tau_{d(t)+f_{\zeta}}S^{{-}1}_{\omega}i S^{{-}1}_{\omega}\left\{\Pi^{\perp}\left[ip_{j}\right]\right\}(\cdot)\right](x).
\end{align*}
 For any function $\rho$ with domain $\mathbb{R},$ let
 \begin{equation}
     \rho_{1}(t,x)\coloneqq e^{i\alpha_{1}(t,x)}\rho(x-d(t)-f_{\zeta}(t)) \text{, for any $(t,x)\in\mathbb{R}^{2}.$}
 \end{equation}
 Using Lemma \ref{invertexpfor}, the estimates from Remark \ref{ffff} and \eqref{decd}, we can verify for any $j\in M_{0}\cup K_{0}$ that
\begin{multline}\label{Lcorr}
  i\partial_{t}\left[f_{j}(t)\left(S^{{-}1}_{\omega}(p_{j})\right)_{1}(t,x)\right]+\partial^{2}_{x}\left[f_{j}(t)\left(S^{{-}1}_{\omega}(p_{j})\right)_{1}(t,x)\right]{+}f_{j}(t)F^{''}\left(\vert\phi_{\omega,1}\vert^{2}\right)f_{j}(t)\left(S^{{-}1}_{\omega}(p_{j})\right)_{1}(t,x)\\\begin{aligned}
  {+}f_{j}(t)F^{''}(\vert\phi_{\omega,k}\vert^{2})\vert\phi_{\omega,k}\vert^{2}e^{2i\alpha_{1}(t)}\overline{\left(S^{{-}1}_{\omega}(p_{j})\right)_{1}}(t,x)=&{-}f_{j}(t)(p_{j})_{1}(t,x)+i\dot f_{j}(t)\left(S^{{-}1}_{\omega}(p_{j})\right)_{1}(t,x)\\
    &{+}O\left(v^{4}\left(\vert t\vert v+\ln{\left(\frac{1}{v}\right)}\right)^{\sigma}e^{{-}2\sqrt{\omega}v\vert t\vert}\right),
  \end{aligned}
\end{multline}
the remainder above means a smooth function $\rho(t,x)$ satisfying
\begin{equation*}
\norm{\partial^{l}_{t}\rho(t,x)}_{H^{1}}\lesssim_{l}v^{4+l}\left(\vert t\vert v+\ln{\left(\frac{1}{v}\right)}\right)^{\sigma}e^{{-}2\sqrt{\omega}v\vert t\vert}.
\end{equation*}

\par Furthermore, Lemmas \ref{interactt}, \ref{inverhol} and \eqref{decd} imply for any $i,\,j\in\{1,2,3\}$ that there exists a constant $\sigma>0$ satisfying
\begin{equation*}
    \norm{p_{j,\omega}(x-d(t))p_{i,\omega}(x+d(t))}_{H^{1}_{x}(\mathbb{R})}+\norm{\phi_{\omega}(x\pm d)p_{j,\omega}(x\mp d)}_{H^{1}_{x}(\mathbb{R})}=O\left(v^{2}\left(\vert t\vert v+\ln{\left(\frac{1}{v}\right)}\right)^{\sigma}e^{{-}2\sqrt{\omega}\vert t\vert v}\right).
\end{equation*}
From now on, we choose the expression from equation \eqref{phik} in the statement of Theorem \ref{app lemma}
\begin{multline}\label{correq}
    \sum_{j\in J} ig_{odd,}(t)\left[e^{i\alpha_{1}(t,x)}p_{j,\omega}(x-d(t)-f_{\zeta}(t))-e^{{-}\alpha_{1}(t,x)}p_{j,\omega}({-}x-d(t)-f_{\zeta}(t))\right]\\ 
    {+}\sum_{j\in I}  g_{even,j}(t)\left[e^{i\alpha_{1}(t,x)}p_{j,\omega}(x-d(t)-f_{\zeta}(t))-e^{{-}\alpha_{1}(t,x)}p_{j,\omega}({-}x-d(t)-f_{\zeta}(t))\right]
\end{multline}
to be equal to
$Corr_{1}(t.x)+Corr_{2}(t,x).$ 
\par Consequently, we deduce from estimates \eqref{lambdau}, \eqref{interF0}, \eqref{pp1} and \eqref{Lcorr} that
the function $\varphi_{1}$ defined by
\begin{align*} 
 \varphi_{1}(t,x)\coloneqq &\varphi_{1,0}(t,x)
     {+}Corr_{1}(t,x)+Corr_{2}(t,x)
\end{align*}
satisfies Theorem \ref{app lemma} for $k=1.$ The oddness and evenness of each real function $f_{even,j},\,f_{odd,j}$ in $Corr_{1}$ and $Corr{2}$ follows from Remark \ref{oddinve}.
 
\par Furthermore, the estimates of Remark \ref{ffff} imply for some $\sigma>0$ that
\begin{align*}    \left\vert\frac{d^{l}}{dt^{l}}\left\langle \Lambda(\varphi_{1})(t,x),\, e^{i\alpha_{1}(t,\pm x)}\phi^{'}_{\omega}(\pm x-d(t)-f_{\zeta}) \right\rangle\right\vert=&O\left(v^{4+l}\left(\vert t\vert v+\ln{\left(\frac{1}{v}\right)}\right)^{\sigma}e^{{-}2\sqrt{\omega}v\vert t\vert}\right),\\
    \left\vert\frac{d^{l}}{dt^{l}}\left\langle \Lambda(\varphi_{1})(t,x),\, i e^{i\alpha_{1}(t,\pm x)}\phi_{\omega}(\pm x-d(t)-f_{\zeta}(t)) \right\rangle\right\vert=&O\left(v^{4+l}\left(\vert t\vert v+\ln{\left(\frac{1}{v}\right)}\right)^{\sigma}e^{{-}2\sqrt{\omega}v\vert t\vert}\right),\\
    \left\vert\frac{d^{l}}{dt^{l}}\left\langle \Lambda(\varphi_{1})(t,x),\, e^{i\alpha_{1}(t,\pm x)}\partial_{\omega}\phi_{\omega}(\pm x-d(t)-f_{\zeta}) \right\rangle\right\vert=&O\left(v^{4+l}\left(\vert t\vert v+\ln{\left(\frac{1}{v}\right)}\right)^{\sigma}e^{{-}2\sqrt{\omega}v\vert t\vert}\right),\\
    \left\vert\frac{d^{l}}{dt^{l}}\left\langle \Lambda(\varphi_{1})(t,x),\, e^{i\alpha_{1}(t,\pm x)}(\pm x-d(t)-f_{\zeta})\phi_{\omega}(\pm x-d(t)-f_{\zeta}) \right\rangle\right\vert=&O\left(v^{4+l}\left(\vert t\vert v+\ln{\left(\frac{1}{v}\right)}\right)^{\sigma}e^{{-}2\sqrt{\omega}v\vert t\vert}\right)
\end{align*} 
for all $l\in\mathbb{N},$ if $v>0$ is small enough. 
\\
\par \textbf{Step 2.}(Estimate of $\Lambda(\varphi_{k}).$) From now on, we consider the existence of functions $v_{k},\,\zeta_{k},\,\gamma_{k}$ and a smooth function $\varphi_{k}:\mathbb{R}^{2}\to\mathbb{R}$ satisfying all the properties of Theorem \ref{app lemma} until $k=k_{0}\geq 1.$  
 To simplify our notation, we consider
\begin{equation}\label{alpha}
    \alpha_{k}(t,x)=\gamma_{k}(t)+\frac{v_{k}(t)}{2}\left(x-\frac{\zeta_{k}(t)}{2}\right).
\end{equation}
\par First, using Lemma \ref{invertexpfor}, we can obtain the following identity for any smooth function $\rho$
\begin{multline*}
   \left[\frac{\partial^{2}}{\partial x^{2}}-\omega\right]\rho_{k}(t,x)+\left(F^{'}\left(\vert \phi_{\omega,k}\vert^{2}\right)+F^{''}\left(\vert\phi_{\omega,k}\vert^{2}\right)\vert \phi_{\omega,k}\vert^{2}\right)\rho_{k}(t,x)+F^{''}(\vert\phi_{\omega,k}\vert^{2})\vert\phi_{\omega,k}\vert^{2}e^{2i\alpha_{k}(t)}\overline{\rho}_{k}(t,x)\\
    ={-}\left(S_{\omega}\rho\right)_{k}(t,x)-\frac{v_{k}(t)^{2}}{4}\rho_{k}(t,x)+ i v_{k}(t)(\rho^{'})_{k}(t,x).
\end{multline*}
Consequently, we obtain that
\begin{multline}\label{general formula}
i\frac{\partial \rho_{k}(t,x)}{\partial t}+\frac{\partial^{2}\rho_{k}(t,x)}{\partial x^{2}}+\left(F^{'}\left(\vert\phi_{\omega,k}\vert^{2}\right)+F^{''}\left(\vert\phi_{\omega,k}\vert^{2}\right)\vert\phi_{\omega,k}\vert^{2}\right)\rho_{k}(t,x){+}F^{''}(\phi_{\omega,k}^{2})\phi_{\omega,k}^{2}\overline{\rho}_{k}(t,x)\\
\begin{aligned}
=&{-}\left(S_{\omega}\rho\right)_{k}(t,x)-i\left[\dot\zeta_{k}(t)-v_{k}(t)\right](\rho^{'})_{k}(t,x)\\
&{-}\frac{\dot v_{k}(t)}{2}\left[x-\frac{\zeta_{k}(t)}{2}\right]\rho_{k}(t,x)+\left[\dot \zeta_{k}(t)-v_{k}(t)\right]\frac{v_{k}(t)}{4}\rho_{k}(t,x)\\&{-}(\dot\gamma_{k}(t)-\omega)\rho_{k}(t,x).
\end{aligned}
\end{multline}
Therefore, since $F$ is a polynomial satisfying \eqref{H1}, using the estimate \eqref{Lcorr} from Step $1,$ the fact that $Corr_{1}(t,x)+Corr_{2}(t,x)$ is equal to \eqref{correq}, the estimate \eqref{interF00} and Lemmas \ref{deriv}, \ref{alg}, and \ref{separation}, we can verify the existence of a finite set of numbers $a_{l}\in\mathbb{N},\,j_{l}\in\mathbb{Z}$ independent of $v,\,t,\,x$ such that for any $(t,x)\in\mathbb{R}^{2}$
\begin{multline}\label{lambdauk}
\begin{aligned}
    \Lambda(\varphi_{k})(t,x)= &i\dot v_{k}(t) \left(\partial_{v}\phi_{\omega,k}(t,x)-\partial_{v}\phi_{\omega,k}(t,-x)+\sum_{j}f_{j}(t)\partial_{v} Sym\left((p_{j,\omega})_{k}\right)(t,x)\right)\\
    &{+}i\dot v_{k}(t)f_{\omega,k}(t)\partial_{v}Sym\left(\partial_{\omega}\phi_{\omega}\right)(t,x)\\
    &{+}i(\dot \zeta_{k}(t)-v_{k}(t))\left(\partial_{\zeta}\phi_{\omega,k}(t,x)-\partial_{\zeta}\phi_{\omega,k}(t,-x)+\sum_{j}f_{j}(t)\partial_{\zeta} Sym\left((p_{j,\omega})_{k}\right)(t,x)\right)\\
    &{+}i(\dot \zeta_{k}(t)-v_{k}(t))f_{\omega,k}(t)\partial_{\zeta}Sym\left(\partial_{\omega}\phi_{\omega}\right)(t,x)\\
    &{+}i(\dot \gamma_{k}(t)-\omega)\left(\partial_{\gamma}\phi_{\omega,k}(t,x)-\partial_{\gamma}\phi_{\omega,k}(t,{-}x)+\sum_{j}f_{j}(t)\partial_{\gamma} Sym\left((p_{j,\omega})_{k}\right)(t,x)\right)\\
    &{+}i(\dot\gamma_{k}(t)-\omega)f_{\omega,k}(t)\partial_{\gamma}Sym\left(\partial_{\omega}\phi_{\omega}\right)(t,x)
    \\
   &{+}i \dot f_{\omega,k}(t)\left[\partial_{\omega}\phi_{\omega,k}(t,x)-\partial_{\omega}\phi_{\omega,k}(t,{-}x)\right]-f_{\omega,k}(t)\left[\left(S_{\omega}(\partial_{\omega}\phi_{\omega})\right)_{k}(t,x)-\left(S_{\omega}(\partial_{\omega}\phi_{\omega})\right)_{k}(t,{-}x)\right]
    \\
&{-}2F^{''}\left(\left\vert \phi_{\omega,k}(t,x)\right\vert^{2}+\left\vert\phi_{\omega,k}(t,{-}x)\right\vert^{2}\right)\left[Sym\left(\phi_{\omega,k}\right)(t,x)\right]\left\vert \phi_{\omega,k}(t,x)\phi_{\omega,k}(t,{-}x)\right\vert
    \\ &{+}F^{'}\left(\left\vert\phi_{\omega,k}(t,{-}x)\right\vert^{2}\right)\phi_{\omega,k}(t,x)-F^{'}\left(\left\vert\phi_{\omega,k}(t,x)\right\vert^{2}\right)\phi_{\omega,k}(t,{-}x)\\
    &{-}\sum_{j\in K_{0}}f_{even,j}(t)Sym\left[e^{i\alpha_{k}(t,\cdot)}\tau_{\zeta_{k}(t)}\Pi^{\perp}p_{j}(\cdot)\right](x)\\&{-}\sum_{j\in M_{0}}f_{odd,j}(t)Sym\left[e^{i\alpha_{k}(t,\cdot)}\tau_{\zeta_{k}(t)}\Pi^{\perp}p_{j}(\cdot)\right](x)
    \\&{+}\sum_{n}f_{n}(v,t)f_{\omega,k}(t)^{a_{n}}e^{{-}2\sqrt{\omega}\zeta_{k}l_{n}}Sym(e^{i j_{n} v_{k}(t) (\cdot)}r_{n,k})(t,x)+O(e^{{-}20k\sqrt{\omega}\zeta_{k}}),
\end{aligned}
\end{multline}
%correct notation re_{l,k} p_{l,k} later
where $l_{n}\in\mathbb{N}\setminus \{0\},$ and all the sums above are finite and all the functions $f_{n}$ depend only on $v$ and $t$ and satisfy
\begin{equation}\label{estre}
    \max_{n}v^{2a_{n}+2l_{n}}\left\vert f^{(j)}_{n}(v,t) \right\vert\lesssim_{j} v^{4+j}\left(v\vert t\vert+\ln{\frac{1}{v}}\right)^{c(k)}e^{{-}2\sqrt{\omega}\vert t\vert v},
\end{equation}
for all $t\in\mathbb{R}$ and $\in\mathbb{R}$ if $v>0$ is sufficiently 
small.
\par Furthermore, to simplify more our notation, we consider 
\begin{multline}\label{mod}
 \begin{aligned}   
    Mod_{k}(t,x)=&i\dot v_{k} \left(\partial_{v}\phi_{\omega,k}(t,x)-\partial_{v}\phi_{\omega,k}(t,-x)+\sum_{j}f_{j}(t)\partial_{v} Sym\left((p_{j,\omega})_{k}\right)(t,x)\right)\\
    &{+}i\dot v_{k}(t)f_{\omega,k}(t)\partial_{v}Sym\left(\partial_{\omega}\phi_{\omega}\right)(t,x)\\
    &{+}i(\dot \zeta_{k}-v_{k})\left(\partial_{\zeta}\phi_{\omega,k}(t,x)-\partial_{\zeta}\phi_{\omega,k}(t,-x)+\sum_{j}f_{j}(t)\partial_{\zeta} Sym\left((p_{j,\omega})_{k}\right)(t,x)\right)\\
    &{+}i(\dot \zeta_{k}-v_{k})f_{\omega,k}(t)\partial_{\zeta}Sym\left(\partial_{\omega}\phi_{\omega}\right)(t,x)\\
    &{+}i(\dot \gamma_{k}(t)-\omega)\left(\partial_{\gamma}\phi_{\omega,k}(t,x)-\partial_{\gamma}\phi_{\omega,k}(t,{-}x)+\sum_{j}f_{j}(t)\partial_{\gamma} Sym\left((p_{j,\omega})_{k}\right)(t,x)\right)\\
    &{+}i(\dot\gamma_{k}(t)-\omega)f_{\omega,k}(t)\partial_{\gamma}Sym\left(\partial_{\omega}\phi_{\omega}\right)(t,x)
    \\
   &{+}i \dot f_{\omega,k}(t)\left[\partial_{\omega}\phi_{\omega,k}(t,x)-\partial_{\omega}\phi_{\omega,k}(t,{-}x)\right]\\&{-}f_{\omega,k}(t)\left[\left(S_{\omega}(\partial_{\omega}\phi_{\omega})\right)_{k}(t,x)-\left(S_{\omega}(\partial_{\omega}\phi_{\omega})\right)_{k}(t,{-}x)\right],
\end{aligned}
\end{multline}
for any $(t,x)\in\mathbb{R}^{2},$ and $k\in\mathbb{N}.$
\par Next, since $F$ is a polynomial, we deduce applying Lemma \ref{separation} and using Remarks \ref{asyrem}, \ref{phis+} that
\begin{multline}\label{m}
    {-}2F^{''}\left(\left\vert \phi_{\omega,k}(x)\right\vert^{2}+\left\vert\phi_{\omega,k}({-}x)\right\vert^{2}\right)\left[Sym\left(\phi_{\omega,k}\right)(t,x)\right]\left\vert \phi_{\omega,k}(t,x)\phi_{\omega,k}(t,{-}x)\right\vert\\
    \begin{aligned}
        &={-}2 a_{{+}\infty}e^{{-}2\sqrt{\omega}\zeta_{k}}Sym\left[e^{i\alpha_{k}(t,\cdot)}F^{''}(\phi_{\omega}(\cdot)^{2})\phi_{\omega}(\cdot)^{2}e^{{-}\sqrt{\omega}(\cdot)}\right](x-\zeta_{k})+O\left(e^{{-}4\sqrt{\omega}\zeta_{k}}\right),
    \end{aligned}
\end{multline}
where the term of order $O\left(e^{{-}4\sqrt{\omega}\zeta_{k}}\right)$ above means a smooth function $f(t,x)$ satisfying 
\begin{equation*}
\norm{\frac{\partial^{l}}{\partial t^{l}}f(t,x)}_{H^{s}_{x}(\mathbb{R})}\lesssim_{l,s} v^{l}e^{{-}4\sqrt{\omega}\zeta_{k}}
\end{equation*}
for any $l\in\mathbb{N}$ and $s\geq 0.$ Indeed, from Lemma \ref{separation}, we can also describe this function $f(t,x)$ with more precision as a finite sum of elements of the form $e^{{-}2d_{m}\sqrt{\omega}\zeta_{k}}r_{m}(\pm x-\zeta_{k})$ such that $r_{m}(\pm\cdot) \in \mathcal{S}^{+}$ plus a remainder function $g(t,x)$ satisfying $\norm{\frac{\partial^{l}}{\partial t^{l}}g(t,x)}_{H^{1}}\lesssim_{l} v^{l+20k}e^{{-}4\sqrt{\omega}\zeta_{k}}$ for any $l\in\mathbb{N}.$ \par Similarly, we can verify that
\begin{multline}\label{po2}
    F^{'}\left(\left\vert\phi_{\omega,k}(t,{-}x)\right\vert^{2}\right)\phi_{\omega,k}(t,x)-F^{'}\left(\left\vert\phi_{\omega,k}(t,x)\right\vert^{2}\right)\phi_{\omega,k}(t,{-}x)\\
    \begin{aligned}
     &={-}a_{{+}\infty}e^{{-}2\sqrt{\omega}\zeta_{k}(t)}Sym\left[e^{{-}iv_{k}(\cdot)}\tau_{\zeta_{k}(t)}e^{i\alpha_{k}(t,\cdot)}F^{'}(\phi_{\omega}(\cdot)^{2})e^{{-}\sqrt{\omega}(\cdot)}\right](x)+O\left(e^{{-}4\sqrt{\omega}\zeta_{k}}\right)
    \end{aligned}
\end{multline}
such that the term $O\left(e^{{-}4\sqrt{\omega}\zeta_{k}}\right)$ also means a smooth function $f(t,x)$ satisfying $\norm{\frac{\partial^{l}}{\partial t^{l}}f(t,x)}_{H^{1}}\lesssim_{l} v^{l}e^{{-}4\sqrt{\omega}\zeta_{k}}$ for any $l\in\mathbb{N}$.  Analogously, using Lemma \ref{separation}, we also can verify that $f(t,x)$ can be estimated as a finite sum of functions $e^{{-}2d_{m}\sqrt{\omega}\zeta_{k}}r_{m}(\pm x-\zeta_{k}(t))$ plus a remainder $g(t,x)$ such that one  of the functions $r_{m}(\pm \cdot)$ is in $\mathcal{S}^{+}$ and $\norm{\frac{\partial^{l}}{\partial t^{l}}g(t,x)}_{H^{s}_{x}(\mathbb{R})}\lesssim_{l,s} v^{l+20k}e^{{-}4\sqrt{\omega}\zeta_{k}}$ for any $l\in\mathbb{N}$ and $s\geq 0.$\\
\par Consequently, using estimates \eqref{m} and \eqref{po2},it is not difficult to verify from \eqref{lambdauk} that
\begin{multline}\label{pophik}
\begin{aligned}
\Lambda(\varphi_{k})(t,x)=& Mod_{k}(t,x)
    \\&{-}2 a_{{+}\infty}e^{{-}2\sqrt{\omega}\zeta_{k}}Sym\left[e^{i\alpha_{k}(t,\cdot)}\tau_{\zeta_{k}(t)}F^{''}(\phi_{\omega}(\cdot)^{2})\phi_{\omega}(\cdot)^{2}e^{{-}\sqrt{\omega}(\cdot)}\right](x)\\ &{-}a_{{+}\infty}e^{{-}2\sqrt{\omega}\zeta_{k}(t)}Sym\left[e^{{-}iv_{k}(\cdot)}e^{i\alpha_{k}(t,\cdot)}\tau_{\zeta_{k}(t)}F^{'}(\phi_{\omega}(\cdot)^{2})e^{{-}\sqrt{\omega}(\cdot)}\right](x)\\
    &{-}\sum_{j\in K_{0}}f_{even,j}(t)Sym\left[e^{i\alpha_{k}(t,\cdot)}\tau_{\zeta_{k}(t)}\Pi^{\perp}p_{j}(\cdot)\right](x)\\&{-}\sum_{j\in M_{0}}f_{odd,j}(t)Sym\left[e^{i\alpha_{k}(t,\cdot)}\tau_{\zeta_{k}(t)}\Pi^{\perp}p_{j}(\cdot)\right](x)
    \\
    &{+}\sum_{n}f_{n}(v,t)f_{\omega,k}(t)^{a_{n}}e^{{-}2l_{n}\sqrt{\omega}\zeta_{k}}Sym(e^{i j_{n}v_{k}(t) (\cdot)}r_{n,k})(t,x)+O(e^{{-}20k\sqrt{\omega}\zeta_{k}})\\
    =&\sum_{j\in J_{k}} r_{j}(v,t)Sym\left(\left(\rho_{j}\right)_{k}\right)(t,x)+O(e^{{-}20k\sqrt{\omega}\zeta_{k}}),
\end{aligned}    
\end{multline}
such that the estimate \eqref{estre} holds, $J_{k}$ is a finite set depending only on $k$ and all the functions $\rho_{j}\left(\frac{\cdot}{\sqrt{\omega}}\right)$ are in $\mathcal{S}^{+}_{\infty},$ see Definition \ref{S}.
The expression $O(e^{{-}20k\sqrt{\omega}\zeta_{k}})$ in \eqref{pophik} means a smooth function $g(t,x)$ satisfying
\begin{equation*}
\norm{\partial^{l}_{t}g(t,x)}_{H^{s}_{x}}\lesssim_{s,l}v^{l}e^{{-}20k\sqrt{\omega}\zeta_{k}(t)},
\end{equation*}
for any $t\in\mathbb{R},$ if $v>0$ is small enough.
\par Furthermore, since if $\varphi_{k}$ satisfies the hypotheses of Theorem \ref{app lemma},  using Lemmas \ref{represent1} and \ref{oddeven}, we can restrict to the case where all the functions $r_{j}$ are real and 
\begin{align}\label{oddrho}
   r_{j} \text{ is even on $t,$ if $\rho_{j}$ is a real function,}\\ \nonumber 
   r_{j} \text{ is odd on $t,$ if $i\rho_{j}$ is a real function.}
\end{align}
 
 \textbf{Step 3.}(Construction of parameters $\zeta_{k_{0}+1},\,v_{k_{0}+1},\, \gamma_{k_{0}+1},\,f_{\omega,k_{0}+1}.$)
 Furthermore, we are going to verify that for
\begin{align*}
f_{\omega,k+1}(v,t)=& f_{\omega,k}(v,t)+\delta f_{\omega,k}(v,t),\,v_{k+1}(v,t)=v_{k}(v,t)+\delta v_{k}(v,t),\\
\zeta_{k+1}(v,t)=&\zeta_{k}(v,t)+\delta\zeta_{k}(v,t),\,\gamma_{k+1}(v,t)=\gamma_{k}(v,t)+\delta\gamma_{k}(v,t),
\end{align*}
the functions $\delta f_{\omega,k},\,\delta v_{k},\,\delta\zeta_{k},\,\delta\gamma_{k}$ shall satisfy for some $c(k)>0$ the following decay  
\begin{multline}\label{deltadecay}
   v \left\vert \frac{\partial^{l}}{\partial t^{l}}\delta f_{\omega,k}(v,t)\right\vert+v\left\vert \frac{\partial^{l}}{\partial t^{l}}\delta v_{k}(v,t)\right\vert+v^{2}\left\vert \frac{\partial^{l}}{\partial t^{l}}\delta\zeta_{k}(v,t)\right\vert+v^{2}\left\vert \frac{\partial^{l}}{\partial t^{l}}\delta\gamma_{k}(v,t)\right\vert\\{+} \left\vert \frac{\partial^{l+1}}{\partial t^{l+1}} \delta\gamma_{k}(v,t)-\frac{\partial^{l}}{\partial t^{l}}f_{\omega,k}(v,t))\right\vert\leq C(l) v^{2k+2+l}\left(\ln{\frac{1}{v}}+\vert t\vert v\right)^{c(k)}e^{{-}2\vert t\vert\sqrt{\omega} v},
\end{multline}
for every $l\in\mathbb{N},\,t\in\mathbb{R},$ $C(l)>0$ depending only on $l$ and $0<v\ll 1.$ 
From Step $1,$ we have considered $f_{\omega,1},\,\zeta_{\omega,1}=d(t)+f_{\zeta}(t),\,\gamma_{1}(t)$ and $v_{1}(t)$ to be the same as the ones defined at Remark \ref{ffff}. 

\par Moreover, considering $f_{\omega,0}(v,t)\equiv 0,\,\gamma_{0}(v,t)=\omega t,\, v_{0}(v,t)=\dot d(t) ,\,\zeta_{0}(v,t)=d(t),$ we deduce from Remark \ref{ffff} that the estimate \eqref{deltadecay} is true for $k=0.$ Therefore, we can assume that the existence of $v_{k},\, \gamma_{k},\, \zeta_{k},\,f_{\omega,k}$ satisfying \eqref{deltadecay} until $k=k_{0}-1\in\mathbb{N}.$

 \par From now on, we consider $f_{\omega,k_{0}+1},\,v_{k_{0}+1},\,\zeta_{k_{0}+1}$ and $\gamma_{k_{0}+1}$ to be functions such that the estimates \eqref{deltadecay} are true for $k=k_{0}$ for some $c_{k_{0}}>0,$ these functions will be chosen carefully later for the construction of $\varphi_{k_{0}+1}$ satisfying Theorem \ref{app lemma}. 
 \par Next, using Taylor's Expansion Theorem, we can verify for any $l\in\mathbb{N}$ and $v>0$ sufficiently small the following estimate
 \begin{multline}\label{taylorint}
   {-} a_{{+}\infty}e^{{-}2\sqrt{\omega}\zeta_{k_{0}+1}(t)}={-}a_{{+}\infty}e^{{-}2\sqrt{\omega}\zeta_{k_{0}}(t)}+2a_{{+}\infty}\sqrt{\omega}\delta \zeta_{k_{0}}(t)e^{{-}2\sqrt{\omega}\zeta_{k_{0}}(t)}\\{+}O\left( v^{4k_{0}}\left(\ln{\frac{1}{v}}+\vert t\vert v\right)^{2c(k_{0})}e^{{-}4\vert t\vert\sqrt{\omega} v}\right),
 \end{multline}
 and the term $O\left( v^{4k_{0}}\left(\ln{\left(\frac{1}{v}\right)}+\vert t\vert v\right)^{2c(k_{0})}e^{{-}4\vert t\vert\sqrt{\omega} v}\right)$ means a smooth function $f(t,x)$ satisfying
 \begin{equation*}
     \norm{\frac{\partial^{l}}{\partial t^{l}}f(t,x)}_{H^{s}_{x}(\mathbb{R})}\lesssim_{l,s} v^{4k_{0}+l}\left(\ln{\left(\frac{1}{v}\right)}+\vert t\vert v\right)^{2c(k_{0})}e^{{-}4\vert t\vert\sqrt{\omega} v},
 \end{equation*}
for any $l\in\mathbb{N},$ if $v>0$ is sufficiently small. Indeed, using the estimates \eqref{deltadecay}, identity $\zeta_{0}(t)=d(t)$ and estimate \eqref{decd}, we can also obtain from \eqref{taylorint} the existence of a new constant $c_{k_{0}}$ such that
\begin{multline*}
    {-} a_{{+}\infty}e^{{-}2\sqrt{\omega}\zeta_{k_{0}+1}(t)}={-}a_{{+}\infty}e^{{-}2\sqrt{\omega}\zeta_{k_{0}}(t)}+2a_{{+}\infty}\sqrt{\omega}\delta \zeta_{k_{0}}(t)e^{{-}2\sqrt{\omega}d(t)}\\{+}O\left( v^{2k_{0}+4}\left(\ln{\frac{1}{v}}+\vert t\vert v\right)^{2c(k_{0})}e^{{-}2\vert t\vert\sqrt{\omega} v}\right).
\end{multline*}
From now on, using $f_{\omega,k_{0}+1},\,\zeta_{k_{0+1}},\,v_{k_{0}+1},$ we consider from \eqref{trk} the following notation
\begin{equation*}
    \rho_{k_{0}+1}(t,x)\coloneqq e^{i\left(\gamma_{k_{0}+1}(t)+\frac{v_{k_{0}+1}(t)}{2}\left[x-\frac{\zeta_{k_{0}+1}(t)}{2}\right]\right)}\rho\left(x-\zeta_{k_{0}+1}(t)\right),
\end{equation*}
 for any real function $\rho$
 with domain $\mathbb{R}.$
 \par   Next, recalling the definition of $Sym$ on \eqref{symoo} and using the definition of $\varphi_{k}$ on \eqref{phik}, we consider the following function  
\begin{align}\nonumber 
    \varphi_{k_{0},0}(t,x)=&\left[e^{i\frac{v_{k_{0}+1}}{2}\left(x-\frac{\zeta_{k_{0}+1}}{2}\right)+i\gamma_{k_{0}+1}}\phi_{\omega}(x-\zeta_{k_{0}+1})-e^{i\frac{v_{k_{0}+1}}{2}\left({-}x-\frac{\zeta_{k_{0}+1}}{2}\right)+i\gamma_{k_{0}+1}}\phi_{\omega}({-}x-\zeta_{k_{0}+1})\right]\ \\ \label{phik00} 
    &{+}f_{\omega,k_{0}+1}(t)Sym\left[\tau_{\zeta_{k_{0}+1}}\left(e^{iv_{k_{0}+1}\frac{(\cdot)}{2}+i\gamma_{k_{0}+1}}\partial_{\omega}\phi_{\omega}(\cdot)\right)\right](x)
    \\  \nonumber
    &{+}\sum_{j\in J} ig_{odd,j}(t)Sym\left[(p_{j,\omega})_{k_{0}+1}(t,\cdot)\right](x)+\sum_{j\in I} g_{even,j}(t)Sym\left[(p_{j,\omega})_{k_{0}+1}(t,\cdot)\right](x)\\ 
    \nonumber
    &{+}\sum_{j\in J_{k_{0}}}i g_{odd,j}(t)Sym\left[(p_{j,\omega})_{k_{0}+1}(t,\cdot)\right](x) +\sum_{j\in I_{k_{0}}} g_{even,j}(t)Sym\left[(p_{j,\omega})_{k_{0}+1}(t,\cdot)\right](x)
\end{align}
Consequently, using the estimate \eqref{pophik} for $\Lambda(\varphi_{k_{0}}),$ estimates \eqref{deltadecay} and estimate  \eqref{taylorint},
we can verify for some $c(k_{0})>0$ that
\begin{multline}\label{finalestimate}
    \begin{aligned}
     \Lambda(\varphi_{k_{0},0})(t,x)= &i (\delta v_{k_{0}})^{(\prime)}(t)Sym\left(\partial_{v}\phi_{\omega,k_{0}+1}(t,\cdot)+\sum_{j}f_{j}(t)\partial_{v} p_{j,k_{0}+1}(t,\cdot)\right)(x)\\
     &{+}i\left[( \delta\zeta_{k_{0}})^{(\prime)}(t)-\delta v_{k_{0}}(t)\right]Sym\left(\partial_{\zeta}\phi_{\omega,k_{0}+1}(t,\cdot)+\sum_{j}f_{j}(t)\partial_{\zeta} p_{j,k_{0}+1}(t,\cdot)\right)(x)\\
     &{+}i (\delta\gamma_{k_{0}})^{(\prime)}(t)Sym\left(\partial_{\gamma}\phi_{\omega,k_{0}+1}(t,\cdot)+\sum_{j}f_{j}(t)\partial_{\gamma} p_{j,k_{0}+1}(t,\cdot)\right)(x)\\
    &{+}i ( \delta f_{\omega,k_{0}})^{(\prime)}(t)Sym\left(\partial_{\omega}\phi_{\omega,k_{0}+1}(t,\cdot)\right)(x)- \delta f_{\omega,k_{0}}(t)Sym\left(S_{\omega}\left(\partial_{\omega}\phi_{\omega}\right)_{k_{0}+1}(t,\cdot)\right)(x)\\
 &{+}4 a_{{+}\infty}\delta\zeta_{k_{0}}(t)\sqrt{\omega} e^{{-}2\sqrt{\omega}d(t)}Sym\left[e^{i\alpha_{k_{0}+1}(t,\cdot)}\tau_{\zeta_{k_{0}+1}(t)}F^{''}(\phi_{\omega}(\cdot)^{2})\phi_{\omega}(\cdot)^{2}e^{{-}\sqrt{\omega}(\cdot)}\right](x)\\  &{+}2a_{{+}\infty}\delta\zeta_{k_{0}}(t)\sqrt{\omega}e^{{-}2\sqrt{\omega}d(t)}Sym\left[e^{{-}iv_{k_{0}+1}(\cdot)}e^{i\alpha_{k_{0}+1}(t,\cdot)}\tau_{\zeta_{k_{0}+1}(t)}F^{'}(\phi_{\omega}(\cdot)^{2})e^{{-}\sqrt{\omega}(\cdot)}\right](x)
         \\&{+}\sum_{j} r_{j}(v,t)Sym\left(\left(\rho_{j}\right)_{k_{0}+1}\right)(t,x)+O\left(v^{2k_{0}+4}\left[\vert t \vert v+\ln{\frac{1}{v}}\right]^{c(k_{0})}e^{{-}2\sqrt{\omega}\vert t\vert v}\right),
    \end{aligned}
\end{multline}    
where the term above associated to $O$ means a smooth function $r(t,x)$ satisfying
\begin{equation*}
\norm{\frac{\partial^{l}r(t,x)}{\partial t^{l}}}_{H^{s}_{x}(\mathbb{R})}\lesssim_{s,l} v^{2k_{0}+4+l}\left[\vert t \vert v+\ln{\frac{1}{v}}\right]^{c(k_{0})}e^{{-}2\sqrt{\omega}\vert t\vert v},
\end{equation*}
for any $l\in\mathbb{N}$ and $s\geq 0.$
\par Furthermore, from Lemma \ref{represent1}, we can restrict to the case where the functions $\rho_{j}$ in estimate \eqref{finalestimate} satisfy $\left\langle \rho_{j},\,\rho_{l} \right\rangle=\delta^{j}_{l}.$ Consequently, using the estimate \eqref{pophik} with $\eqref{ukest}$ for $k=k_{0}$ and Lemma \ref{interactt}, we can deduce that 
\begin{equation}\label{estrj}
    \max_{j}\norm{\frac{\partial^{l}}{\partial t^{l}}r_{j}(t,v)}\lesssim_{l} v^{2k_{0}+2+l}\left(\vert t\vert v+\ln{\left(\frac{1}{v}\right)}\right)^{c(k_{0})}e^{{-}2\sqrt{\omega}\vert t\vert v}, 
\end{equation}
for all $l\in\mathbb{N},$ if $v>0$ is sufficiently small.
\par Moreover, if $\dot\delta \zeta_{k_{0}}(t)=\delta v_{k_{0}}(t),$ using estimates \eqref{decd}, \eqref{deltadecay}, \eqref{finalestimate} and \eqref{estrj}, we can verify the existence of a constant $c(k_{0})>0$ satisfying
\begin{equation}\label{decayphik0}
    \norm{\frac{\partial^{l}}{\partial t^{l}}\Lambda\left(\varphi_{k_{0},0}\right)(t,x)}_{H^{s}_{x}(\mathbb{R})}\lesssim_{s}v^{2k_{0}+2+l}\left[ \vert t\vert v+\ln{\left(\frac{1}{v}\right)}\right]^{c(k_{0})}e^{{-}2\sqrt{\omega}\vert t\vert v},
\end{equation}
for all $l\in\mathbb{N}.$
\\
\textbf{Step 4.}(Ordinary differential equations for $\delta\zeta_{k_{0}},\,\delta v_{k_{0}},\, \delta\gamma_{k_{0}},\,\delta f_{\omega,k_{0}}.$)\\
Similar to the reasoning used in Theorem $4.1$ from \cite{second}, using \eqref{finalestimate}, we are going to obtain a system of ordinary differential equations for the parameters $\delta\zeta_{k_{0}},\,\delta v_{k_{0}},\, \delta\gamma_{k_{0}},\,\delta f_{\omega,k_{0}}$ such that 
\begin{align}\label{estd1}
   \left\vert\frac{\partial^{l}}{\partial t^{l}}\left\langle \Lambda(\varphi_{k_{0},0})(t,x),\left(\phi^{'}\right)_{k_{0}+1}(t,x) \right\rangle\right\vert=&O\left(v^{2k_{0}+4+l}\left(\vert t\vert v+\ln{\left(\frac{1}{v}\right)}\right)^{c(k_{0})}e^{{-}2\sqrt{\omega}\vert t\vert v}\right),\\ \label{estd2}
   \left\vert\frac{\partial^{l}}{\partial t^{l}}\left\langle \Lambda(\varphi_{k_{0},0})(t,x),\partial_{\gamma_{k_{0}+1}}\phi_{k_{0}+1}(t,x) \right\rangle\right\vert=&O\left(v^{2k_{0}+4+l}\left(\vert t\vert v+\ln{\left(\frac{1}{v}\right)}\right)^{c(k)}e^{{-}2\sqrt{\omega}\vert t\vert v}\right),\\ \label{estd3}
   \left\vert\frac{\partial^{l}}{\partial t^{l}}\left\langle \Lambda(\varphi_{k_{0},0})(t,x),i(x-\zeta_{k_{0}+1}(t))\phi_{k_{0}+1}(t,x) \right\rangle\right\vert=&O\left(v^{2k_{0}+4+l}\left(\vert t\vert v+\ln{\left(\frac{1}{v}\right)}\right)^{c(k_{0})}e^{{-}2\sqrt{\omega}\vert t\vert v}\right),\\ \label{estd4} \left\vert\frac{\partial^{l}}{\partial t^{l}}\left\langle \Lambda(\varphi_{k_{0},0})(t,x),\partial_{\omega}\phi_{k_{0}+1}(t,x) \right\rangle\right\vert=&O\left(v^{2k_{0}+4+l}\left(\vert t\vert v+\ln{\left(\frac{1}{v}\right)}\right)^{c(k_{0})}e^{{-}2\sqrt{\omega}\vert t\vert v}\right),
\end{align}
for any $l\in\mathbb{N}.$ 
\par First, using Lemma \ref{interactt} and estimates \eqref{deltadecay}, we obtain from the estimate \eqref{finalestimate} that estimate \eqref{estd1} is equivalent to
\begin{align*}
    \frac{(\delta v_{k_{0}})^{(\prime)}(t)}{4}\norm{\phi_{\omega}}_{L^{2}}^{2}=&{-}2a_{{+}\infty}\omega \delta\zeta_{k_{0}}(t) \left[\int_{\mathbb{R}}F^{'}\left(\phi_{\omega}(x)^{2}\right)\phi_{\omega}(x)e^{{-}\sqrt{\omega}x}\,dx\right]e^{{-}2\sqrt{\omega}d(t)}\\ \nonumber
    &{-}\sum_{j}r_{j}(v,t)\left\langle \rho_{j}(x),\phi^{'}_{\omega}(x) \right\rangle+O\left(v^{2k_{0}+4}\left(\vert t\vert v+\ln{(\frac{1}{v})}\right)^{c(k)}e^{{-}2\sqrt{\omega}\vert t\vert v}\right),
\end{align*}
and using $C>0$ defined in \eqref{d edo}, this is equivalent to  
\begin{align}\label{EDO1}
   \frac{(\delta v_{k_{0}})^{(\prime)}(t)}{4}\norm{\phi_{\omega}}_{L^{2}}^{2}=&{-}\frac{C\sqrt{\omega}\norm{\phi_{\omega}}_{L^{2}}^{2}}{2} \delta\zeta_{k_{0}}(t)e^{{-}2\sqrt{\omega}d(t)}-\sum_{j} r_{j}(v,t)\left\langle\rho_{j},\phi^{'}_{\omega}\right\rangle\\ \nonumber &{+}O\left(v^{2k_{0}+4}\left(\vert t\vert v+\ln{(\frac{1}{v})}\right)^{c(k_{0})}e^{{-}2\sqrt{\omega}\vert t\vert v}\right).
   \end{align}
\par Next, we consider 
\begin{align}\label{p000}
    \delta\zeta_{k_{0}}(t)\coloneqq & \delta_{\zeta_{k_{0}},0}(t)+f_{\zeta_{k_{0}}}(t)-\frac{2}{\norm{\phi_{\omega}}_{L^{2}}^{2}}\sum_{j}\int_{{-}\infty}^{t}r_{j}(v,s)\langle p_{j}(x),ix\phi_{\omega}(x)\rangle\,ds\\ \nonumber
&=\delta_{\zeta_{k_{0}},0}(t)+f_{\zeta_{k_{0}}}(t)+R(t),
\end{align}
such that we are going to define $\delta_{\zeta_{k_{0}},0},\,$ and $f_{\zeta_{k_{0}}}(t)$ in the next paragraphs. 
We observe from \eqref{oddrho}, \eqref{estrj} and the inductive hypothesis on $k_{0}$ that there exists $\sigma>0$ satisfying
\begin{equation*}
    \left\vert \frac{d^{l}}{dt^{l}} R(t)\right\vert=O\left(v^{2k_{0}+1+l}\left(\vert t\vert v+\ln{\frac{1}{v}}\right)^{\sigma}e^{{-}2\sqrt{\omega}\vert t\vert v}\right) \text{, for any $l\in\mathbb{N}.$}
\end{equation*}
\par Next, we define
$\delta_{\zeta_{k_{0}},0}$ as the unique solution of
\begin{align}\label{ddeltav}
    \frac{(\delta\zeta_{k_{0},0})^{\prime\prime}(t)}{4}\norm{\phi_{\omega}}_{L^{2}}^{2}=&{-}\frac{C\sqrt{\omega}\norm{\phi_{\omega}}_{L^{2}}^{2}}{2}\delta\zeta_{k_{0},0}(t) e^{{-}2\sqrt{\omega}d(t)}-\frac{C\sqrt{\omega}\norm{\phi_{\omega}}_{L^{2}}^{2}}{2}R(t) e^{{-}2\sqrt{\omega}d(t)}\\ \nonumber &{-}\sum_{j}r_{j}(v,t)\left\langle \rho_{j}(x),\phi^{'}_{\omega}(x) \right\rangle,\\ \nonumber
    \lim_{t\to{+}\infty} (\delta\zeta_{k_{0},0})^{(\prime)}(t)=0,&\,\lim_{t\to{+}\infty}\delta\zeta_{k_{0},0}(t)=0,
\end{align}
and $f_{\zeta_{k_{0}}}$ an even function to be chosen later satisfying
\begin{equation}\label{fzz}
    \left\vert \frac{d^{l}}{dt^{l}}f_{\zeta_{0}}(t) \right\vert=O\left(v^{2k_{0}+2+l}\left(\vert t\vert v+\ln{\frac{1}{v}}\right)^{\sigma}e^{{-}2\sqrt{\omega}\vert t\vert v}\right),
\end{equation}
for some constant $\sigma>0$ and any $l\in\mathbb{N}.$ 
In particular $\delta \zeta_{k_{0},0},$  is an even function because of \eqref{oddrho}, and for any $l\in\mathbb{N}$
\begin{equation}\label{delta00z}
    \left\vert \frac{d^{l}}{dt^{l}} \delta\zeta_{k_{0},0}(t) \right\vert=O\left(v^{2k_{0}+l}\left(\vert t\vert v+\ln{\frac{1}{v}}\right)^{c(k_{0})}e^{{-}2\sqrt{\omega}\vert t\vert v}\right),
\end{equation}
because of the variation of the parameters method and estimate \eqref{estrj}.
\par Consequently, choosing $\delta v_{k_{0}}=\dot\delta \zeta_{k_{0},0},$ we deduce from the properties of $\delta\zeta_{k_{0},0}$ and estimate \eqref{EDO1} that \eqref{estd1} holds and that $\delta v_{k_{0}}$ is an odd function.  
\par Furthermore, using the identity \eqref{p000}, estimates \eqref{finalestimate}, \eqref{delta00z} and Lemma \ref{interactt}, we can verify that both estimates \eqref{estd3} and \eqref{fzz} would be true if $f_{\zeta_{k_{0}}}(t)$ is the unique solution of 
\begin{align}\label{fzetaODE}
    \frac{\dot f_{\zeta_{k_{0}}}(t)}{2}\norm{\phi_{\omega}}_{L^{2}}^{2}=2a_{{+}\infty}\delta \zeta_{k_{0},0}\sqrt{\omega}e^{{-}2\sqrt{\omega}d(t)}\dot d(t)\left\langle (x+d(t))F^{\prime }(\phi_{\omega}(x)^{2})e^{{-}\sqrt{\omega}x},x\phi_{\omega}(x) \right\rangle,
\end{align}
satisfying $\lim_{t\to{+}\infty}f_{\zeta_{k_{0}}(t)}=0,$
which is an even function satisfying \eqref{p000}. The equation 
\eqref{fzetaODE} comes from the asymptotic estimate of \eqref{estd3} using the decays of $\delta \zeta_{k_{0},0},f_{\zeta_{k_{0}}}.$ 
\par Moreover, using formula \eqref{trk} for $\phi_{\omega},$ we can verify that
\begin{equation*}
   \partial_{v}\phi_{\omega,k}(t,x)=\frac{i}{2}(x-\zeta_{k}(t))\phi_{\omega,k}(t,x)+i\frac{\zeta_{k}(t)}{4}\phi_{\omega,k}(t,x),
\end{equation*}
we also observe that $\phi_{\omega}$ and $\partial_{\omega}\phi_{\omega}$ are even functions, and $x\phi_{\omega}(x)$ is an odd function.
Consequently, using estimates \eqref{phik00} and \eqref{deltadecay}, we can verify that the estimate \eqref{estd4} is true only if
\begin{multline*}\label{edo3almost}
\left(\left(\delta\gamma_{k_{0}}\right)^{(\prime)}(t)+\frac{d(t)}{4}(\delta v_{k_{0}})^{(\prime)}(t)-\delta f_{\omega,k_{0}}(t)\right)\frac{d}{d\omega}\left[\int_{\mathbb{R}}\frac{\phi_{\omega}(x)^{2}}{2}\,dx\right]\\
\begin{aligned}
=& 2a_{{+}\infty}\delta\zeta_{k_{0}}(t)\sqrt{\omega}e^{{-}2\sqrt{\omega}d(t)}\left\langle F^{\prime }(\phi_{\omega}(x)^{2})e^{{-}\sqrt{\omega}x},\partial_{\omega}\phi_{\omega}(x) \right\rangle
\\&{+}4a_{{+}\infty}\delta\zeta_{k_{0}}(t)\sqrt{\omega}e^{{-}2\sqrt{\omega}d(t)}\left\langle F^{\prime \prime}(\phi_{\omega}(x)^{2})\phi_{\omega}(x)^{2}e^{{-}\sqrt{\omega}x},\partial_{\omega}\phi_{\omega}(x) \right\rangle
\\ &{+}\sum_{j}r_{j}(v,t)\left\langle \rho_{j},\partial_{\omega}\phi_{\omega} \right\rangle\\
&{+}O\left(v^{2k_{0}+4+l}\left(\vert t\vert v+\ln{\left(\frac{1}{v}\right)}\right)^{c(k_{0})}e^{{-}2\sqrt{\omega}\vert t\vert v}\right).
\end{aligned}
\end{multline*}
Consequently, since $\delta v_{k_{0}}=\dot\delta\zeta_{k_{0},0}$, we can verify from the ordinary equation satisfied by $\delta\zeta_{k_{0},0}$ that there exist real values $C_{2,\omega},\,C_{3,\omega},\,C_{4,\omega},\,C_{5,\omega}$ depending only on $\omega$ such that \eqref{estd4} is true when $\delta\gamma_{k_{0}}$ is the unique solution of 
\begin{multline*}
   (\delta\gamma_{k_{0}})^{(\prime)}(t)= \delta f_{\omega,k_{0}}(t)+ C_{2,\omega}\delta\zeta_{k_{0}}(t)e^{{-}2\sqrt{\omega}d(t)}+C_{3,\omega}\frac{d(t)}{4} \delta \zeta_{k_{0}}(t)e^{{-}2\sqrt{\omega}d(t)}\\ {+}C_{4,\omega}\sum_{j}r_{j}(v,t)\left\langle \rho_{j},\partial_{\omega}\phi_{\omega} \right\rangle
   +C_{5,\omega}\frac{d(t)}{4}\sum_{j}r_{j}(v,t)\left\langle \rho_{j},\phi^{'}_{\omega} \right\rangle,
\end{multline*}
\begin{equation}\label{EDO3}
    \delta\gamma_{k_{0}}(0)=0,
\end{equation}
and when the estimates \eqref{deltadecay} hold.
\par Next, it remains to find the appropriate ordinary differential estimate that will imply estimate \eqref{estd2}. We also observe that
$
    \partial_{\gamma}\phi_{\omega,k_{0}}(t,x)=i\gamma_{0}\phi_{\omega,k_{0}}(t,x), \,\left\langle  \phi^{'}_{\omega},\,\phi_{\omega}\right\rangle=0,
$
and the real part of $\partial_{v_{k_{0}}}\phi_{\omega,k_{0}}(t,x)$ is zero. Consequently, using Lemmas \ref{interactt}, and estimates \eqref{deltadecay}, we deduce that \eqref{estd2} is true if
\begin{align}\label{fwode}
    \left( \delta f_{\omega,k_{0}}\right)^{(\prime)}(t)\left\langle \partial_{\omega}\phi_{\omega},\phi_{\omega} \right\rangle=&{-}2a_{{+}\infty}\delta \zeta_{k_{0}}(t)\sqrt{\omega}e^{{-}2\sqrt{\omega}d(t)}\dot d(t)
    \left\langle (x+d(t))F^{\prime}(\phi_{\omega}(x)^{2})e^{{-}\sqrt{\omega}x},\phi_{\omega}(x) \right\rangle\\  \nonumber &{-}\sum_{j}r_{j}(v,t)\left\langle \rho_{j}(x),i\phi_{\omega}(x) \right\rangle,\\ \nonumber
     \lim_{t\to{+}\infty} \delta f_{\omega,k_{0}}(t)=0,
     %small change i zeta
\end{align}
because of the Fundamental Theorem of Calculus and the fact that $\delta\zeta_{k_{0}}$ defined at \eqref{ddeltav} satisfies \eqref{deltadecay}, the function $\delta\zeta_{k_{0}}$ is also even because of the equation above.
\par Consequently, using \eqref{EDO3} and the fact that $f_{\omega,k_{0}}$ satisfies \eqref{deltadecay}, we deduce that $\delta\gamma_{k_{0}}$ satisfies all decays in \eqref{deltadecay} and it is an odd function.
 In conclusion, all the functions $\delta\zeta_{k_{0}},\,\delta v_{k_{0}},\,\delta \gamma_{k_{0}}$ and $\delta f_{\omega,k_{0}}$ constructed in this Step satisfy \eqref{deltadecay}.
\\
\textbf{Step 5.}(Construction of $\varphi_{k+1}$ and conclusion.)
\par First, we recall that the function $\varphi_{k_{0},0}$ defined in \eqref{phik00} satisfies estimate \eqref{decayphik0} and the hypotheses of Lemma \ref{oddeven} for $k=k_{0}.$
\par Therefore, using Lemma \ref{oddeven}, we can find a natural number $n_{0}\geq 0$ and real functions $Q_{j},\, P_{j},\,f_{even,j},\,f_{odd,j}$ such that all functions $g_{even,j}$ are even, all functions $g_{odd,j}$ are odd, and 
\begin{align}\label{eqfinal}
    \Lambda(\varphi_{k_{0},0})(t,x)=&\sum_{j=1}^{M_{k_{0}}}if_{odd,j}(v,t)Sym\left(\left(Q_{j}\right)_{k_{0}}\right)(t,x)+\sum_{j=1}^{N_{k_{0}}}f_{even,j}(v,t)Sym\left(\left(P_{j}\right)_{k_{0}}\right)(t,x)\\ \nonumber {+}&O\left(v^{20 k_{0}}\left(\vert t\vert v+\ln{\frac{1}{v}}\right)^{n_{0}}e^{{-}2\sqrt{\omega}\vert t\vert v}\right),
\end{align}
all the functions $P_{j}\left(\frac{\cdot}{\sqrt{\omega}}\right),\,Q_{j}\left(\frac{\cdot}{\sqrt{\omega}}\right)$ are in $\mathcal{S}^{+}_{\infty},$
and the term $O(v^{20 k_{0}})$ in the equation means a smooth function $r(t,x)$ satisfying
\begin{equation*}
\norm{\frac{\partial^{l}r(t,x)}{\partial t^{l}}}_{H^{s}_{x}(\mathbb{R})}\lesssim_{l,s}v^{20 k_{0}+l}\left(\vert t\vert v+\ln{\frac{1}{v}}\right)^{n_{0}}e^{{-}2\sqrt{\omega}\vert t\vert v}.
\end{equation*}
Moreover, using the decay estimate \eqref{decayphik0}, we can restrict to the case where all the functions $f_{odd,j}$ and $f_{even,j}$ satisfy
\begin{equation}\label{cordecay}
    \left\vert \frac{\partial^{l} f_{even,j}(v,t)}{\partial t^{l}} \right\vert+\left\vert \frac{\partial^{l} f_{odd,j}(v,t)}{\partial t^{l}} \right\vert \lesssim_{l} v^{2k_{0}+2+l}\left[ \vert t\vert v+\ln{\left(\frac{1}{v}\right)}\right]^{c(k_{0})}e^{{-}2\sqrt{\omega}\vert t\vert v}.
\end{equation}
\par Furthermore, using estimates \eqref{estd1}, \eqref{estd2}, \eqref{estd3} and \eqref{estd4} in Step $4,$ we can deduce from \eqref{eqfinal} using Lemma \ref{interactt} that
\begin{multline}
\norm{\frac{\partial^{l}}{\partial t^{l}}\left[\sum_{j=1}^{M_{k}}f_{odd,j}(v,t)\Pi\left(i Q_{j}\right)(x)+\sum_{j=1}^{N_{k}}f_{even,j}(v,t)\Pi\left(P_{j}\right)(x)\right]}_{H^{s}_{x}(\mathbb{R})}\\ \lesssim_{s} v^{2k_{0}+4+l}\left[ \vert t\vert v+\ln{\left(\frac{1}{v}\right)}\right]^{c_{1}(k_{0})}e^{{-}2\sqrt{\omega}\vert t\vert v},
\end{multline}
for all $l\in\mathbb{N}$ and $c_{1}(k_{0})>0$ is a constant depending only on $k_{0}.$
\par Next, using Lemma \ref{inverhol}, we can consider the following functions of correction
\begin{align*}
    Cor_{1}(t,x)\coloneq &S^{{-}1}_{\omega}\left[\sum_{j=1}^{M_{k}}f_{odd,j}(v,t)\Pi^{\perp}\left(i Q_{j}\right)(x)+\sum_{j=1}^{N_{k}}f_{even,j}(v,t)\Pi^{\perp}\left(P_{j}\right)(x)\right],\\
    Cor_{2}(t,x)\coloneq & S^{{-}1}_{\omega}\left[i S^{{-}1}_{\omega}\left[\sum_{j=1}^{M_{k}}\partial_{t}f_{odd,j}(v,t)\Pi^{\perp}\left(i Q_{j}\right)(x)+\sum_{j=1}^{N_{k}}\partial_{t}f_{even,j}(v,t)\Pi^{\perp}\left(P_{j}\right)(x)\right]\right]
\end{align*}
whose main motivation is to remove the expressions in the sums on the left-hand side of\eqref{eqfinal}, this approach was made similarly in Step $1$ for $k=1.$
Based on this observation, we consider the following function
\begin{equation}
\label{thephik0+1}
   \varphi_{k_{0}+1}(t,x)\coloneqq  \varphi_{k_{0},0}(t,x)+Sym\left[\tau_{\zeta_{k_{0}+1}}e^{i\alpha_{k_{0}+1}(t,\cdot)}\left(Cor_{1}(t,\cdot)+Cor_{2}(t,\cdot)\right)\right](x),
\end{equation}
which is a function of the same form as the right-hand side of the equation \eqref{phik}.
\par In conclusion, similarly to the approach made in Step $1,$ we can verify using Lemmas \ref{interactt}, \ref{invertexpfor}, Remark \ref{oddinve} and Taylor's Expansion Theorem that the function \eqref{thephik0+1} satisfies Theorem \ref{app lemma} for $k=k_{0}+1.$
\end{proof}

\section{Energy Estimate}\label{s}
\subsection{Dynamics of the Modulation Parameters}
\par First, to simplify our ansatz, we recall the notation \eqref{trk}
\begin{equation*}
    \rho_{k}(t,x)\coloneqq e^{i\left(\gamma_{k}(t)+\frac{v_{k}(t)}{2}\left[x-\frac{\zeta_{k}(t)}{2}\right]\right)}\rho(x-\zeta_{k}(t)).
\end{equation*}
\par In notation of Theorem \ref{app lemma}, we recall the function $\alpha_{k}$ denoted in \eqref{alpha} which is given by
\begin{equation*}
    \alpha_{k}(t,x)=\alpha(t,x,(\zeta_{k},v_{k},\gamma_{k},f_{\omega,k}))\coloneqq \gamma_{k}(t)+\frac{v_{k}(t)}{2}\left(x-\frac{\zeta_{k}(t)}{2}\right).
\end{equation*}
\par From now on, for each $k\in\mathbb{N}_{\geq 2},$ we are going to construct by an inductive argument the modulation parameters $\sigma_{u}=(\zeta_{k}+p_{\zeta},v_{k}+p_{v},\gamma_{k}+p_{\gamma},\omega+p_{\omega})$ such that for \begin{align}\label{dphigamma}
    \varphi_{\gamma,k}(t,x)=&ie^{i\alpha_{k}(t, x,\sigma_{u}(t))}\phi_{\omega}(x-(\zeta_{k}+p_{\zeta})),\\ \label{dphizeta}
    \varphi_{\zeta,k}(t,x)=&e^{i\alpha_{k}(t, x,\sigma_{u}(t))}\phi^{'}_{\omega}(x-(\zeta_{k}+p_{\zeta})),\\ \label{dphiomega}
    \varphi_{\omega,k}(t,x)=&e^{i\alpha_{k}(t,x,\sigma_{u}(t))}\partial_{\omega}\phi_{\omega}(x-(\zeta_{k}+p_{\zeta})),\\ \label{dphiv}
    \varphi_{v,k}(t,x)= & \frac{i}{2}e^{i\alpha_{k}(t, x,\sigma_{u}(t))}(x- (\zeta_{k}+p_{\zeta}))\phi_{\omega}(x-(\zeta_{k}+p_{\zeta})),
\end{align}
the solution $u(t,x)$ of \eqref{NLS3} with initial condition at $t={+}\infty$ given by \eqref{initialcondition} satisfies
\begin{align*}
    \left\langle u(t,x)-P_{k}(t,x,\sigma_{u}),\varphi_{l,k}(t,x)\right\rangle=0,\, \norm{u(t,x)-P_{k}(t,x,\sigma_{u})}_{H^{1}_{x}(\mathbb{R})}<v^{2},
\end{align*}
for any $t\in[{-}M_{v},M_{v}]$ such that $M_{v}>0$ is a large value having size of order $O\left(\frac{\left(\ln{\frac{1}{v}}\right)^{\frac{4}{3}}}{v}\right).$

Moreover, using that $\phi_{\omega}$ is an even smooth function and the dot product defined in \eqref{dotproduct}, it is not difficult to verify that
\begin{equation}\label{orthoderiv}
    \left\langle \varphi_{j,k},\varphi_{l,k} \right\rangle=\delta^{j}_{l},
\end{equation}
for any $j\neq l\in\{\zeta,v,\gamma,\omega\}.$
\par Furthermore, using the Implicit Function Theorem for Banach Spaces and the identities \eqref{dphigamma}, \eqref{dphizeta}, \eqref{dphiomega}, \eqref{dphiv}, we can verify the following proposition.
\begin{lemma}[Modulation Lemma]\label{Modulation}
 Let $\varphi_{k}(t,x)$ be the same as in Theorem \ref{app lemma}. There are constants $K,\,c>0$ and a number $\delta_{k}\in(0,1)$ such that if $0<v<\delta_{k}$ and $u(t,x)$ is an odd function on $x$ in $C\left([{-}T,T],H^{1}_{x}(\mathbb{R})\right)$ for any $T\in \left[0,\frac{\left(\ln{\frac{1}{v}}\right)^{\frac{4}{3}}}{v}\right]$ satisfying 
 \begin{equation}\label{upop}
   \norm{u(t)-\varphi_{k}(t)}_{L^{\infty}\left([{-}T,T],H^{1}_{x}(\mathbb{R})\right)}<v^{2},  
 \end{equation}
then there exist continuous functions $p_{\zeta},\,p_{v},\,p_{\gamma},\,p_{\omega}$ such that
for
\begin{equation*}
    \sigma_{u}(t)=\sigma_{k}(t)+(p_{\zeta},p_{v},p_{\gamma},p_{\omega}),
\end{equation*}
the function $P_{k}(t,x,\sigma_{u}(t))$ satisfies for any $t\in[{-}T,T]$
\begin{align}\label{orthoid1}
    \left\langle u(t,x)-P_{k}(t,x,\sigma_{u}(t)),i\varphi_{\gamma,k}(t,\pm x)\right\rangle=&0,\\ \label{orthoid2}\left\langle u(t,x)-P_{k}(t,x,\sigma_{u}(t)),i\varphi_{\zeta,k}(t,\pm x) \right\rangle=&0,\\ \label{orthoid3}
\left\langle u(t,x)-P_{k}(t,x,\sigma_{u}(t)), i\varphi_{\omega,k}(t,\pm x)\right\rangle=&0,\\  \label{orthoid4}\left\langle u(t,x)-P_{k}(t,x,\sigma_{u}(t)),i\varphi_{v,k}(t,\pm x)\right\rangle=&0,
\end{align}
and 
\begin{align*}
\left\vert p_{\zeta}(t) \right\vert+\left\vert p_{v}(t) \right\vert+\left\vert p_{\omega}(t) \right\vert+\left\vert p_{\gamma}(t) \right\vert\leq &K\left(\ln{\frac{1}{v}}\right)^{c}\norm{u(t)-P_{k}(t,x,\sigma_{u})}_{H^{1}_{x}(\mathbb{R})},\\
\norm{u(t)-P_{k}(t,x,\sigma_{u}(t))}_{H^{1}_{x}(\mathbb{R})}\leq &K \left(\ln{\frac{1}{v}}\right)^{c}\norm{u(t)-\varphi_{k}(t,x)}_{H^{1}_{x}(\mathbb{R})}.
\end{align*}
\end{lemma}
\begin{proof}[Proof of Lemma \ref{Modulation}.]
The proof is completely analogous to the proof of Lemma $8.1$ from \cite{asympt0}, see also \cite{modstability}.    
\end{proof}
\begin{remark}
    The terms with $\left(\ln{\frac{1}{v}}\right)^{\frac{4}{3}}$ in the right-hand side of the inequalities of Lemma \ref{Modulation} are obtained using the identity
    \begin{equation*}
        \partial_{p_{v}}P_{k}(t,x,\sigma_{u}(t)) =\frac{1}{2}\varphi_{v,k}(t,x)+i\frac{\zeta_{k}(t)+p_{\zeta}(t)}{4}P_{k}(t,x,\sigma_{u}(t)),
    \end{equation*}
 with \eqref{remestimate} and from the fact that 
 \begin{equation*}
     \left\vert P_{k}(t,x,\sigma_{k}(t))- P_{k}(t,x,\sigma_{u}(t))\right\vert=DP_{k}(t,x,\sigma_{k}(t))(p_{\zeta},p_{v},p_{\gamma},p_{\omega})+O\left(\left(\ln{\frac{1}{v}}\right)^{\sigma}\max_{j\in\{\zeta,v,\gamma,\omega\}}\vert P_{j}(t,r) \vert^{2}\right).
 \end{equation*}
\end{remark}
 In the notation of Lemma \ref{Modulation}, we consider from now on the following representation for a solution $u(t,x)$ of \eqref{NLS3} 
\begin{equation}\label{uformula}
u(t,x)= P_{k}(t,x,\sigma_{u}(t))+r(t,x)e^{i(\gamma_{k}(t)+p_{\gamma}(t))},
\end{equation}
such that $\sigma_{u}(t)$ and the function $r$ satisfy all the properties of Lemma \ref{Modulation}. Since $u$ is a strong solution of \eqref{NLS3}, we can verify the following lemma.
 \par Furthermore, using the formula \eqref{uformula} for the solution $u(t,x)$ of the nonlinear Schrödinger equation and Taylor's Expansion Theorem, we can verify from Taylor's Expansion Theorem, Lemma \ref{interactt} and estimates \eqref{popopopo} of Theorem \ref{app lemma} that
\begin{multline}\label{lambdar}
\left[i\partial_{t}r(t,x)+\partial^{2}_{x}r(t,x)-\omega r(t,x)\right]e^{i(\gamma_{k}(t)+p_{\gamma}(t))}\\
\begin{aligned}
    =&{-}F^{'}\left(\phi_{\omega}(x-(\zeta_{k}+p_{\zeta}))^{2}+\phi_{\omega}(x+(\zeta_{k}+p_{\zeta}))^{2}\right)r(t,x)e^{i(\gamma_{k}(t)+p_{\gamma}(t))}\\
&{-}F^{''}\left(\phi_{\omega}(x-(\zeta_{k}+p_{\zeta}))^{2}\right)\phi_{\omega}(x-(\zeta_{k}+p_{\zeta}))^{2}r(t,x)e^{i(\gamma_{k}(t)+p_{\gamma}(t))}\\
&{-}F^{''}\left(\phi_{\omega}(x+(\zeta_{k}+p_{\zeta}))^{2}\right)\phi_{\omega}(x+(\zeta_{k}+p_{\zeta}))^{2}r(t,x)e^{i(\gamma_{k}(t)+p_{\gamma}(t))}\\
&{-}F^{''}\left(\phi_{\omega}(x-(\zeta_{k}+p_{\zeta}))^{2}\right)\phi_{\omega}(x-(\zeta_{k}+p_{\zeta}))^{2}\overline{r(t,x)}e^{2i\alpha(t,x,\sigma_{u})-i(\gamma_{k}(t)+p_{\gamma}(t))}\\
&{-}F^{''}\left(\phi_{\omega}(x+(\zeta_{k}+p_{\zeta}))^{2}\right)\phi_{\omega}(x+(\zeta_{k}+p_{\zeta}))^{2}\overline{r(t,x)}e^{2i\alpha(t,{-}x,\sigma_{u})-i(\gamma_{k}(t)+p_{\gamma}(t))}
\\
&{-}\Lambda\left(P_{k}\right)(t,x,\sigma_{u})+O\left(\norm{r(t)}_{H^{1}_{x}(\mathbb{R})}^{2}+\norm{r(t)}_{H^{1}_{x}(\mathbb{R})}v^{2}\left(\ln{\frac{1}{v}}\right)^{c}+\norm{r(t)}_{H^{1}_{x}(\mathbb{R})}\vert \dot p_{\gamma}(t) \vert \right),
\end{aligned}
\end{multline}
for some constant $c>0$ not depending on $k.$ Moreover, from estimates \eqref{m}, \eqref{po2} and Lemmas \ref{represent1}, \ref{Modulation}, and 
%estimate \eqref{pophik} from Step $2$ of the proof of Theorem \ref{app lemma}, 
Theorem \ref{app lemma}, we can verify using Taylor's Expansion Theorem that
\begin{multline}\label{lambdapk}
  \begin{aligned}
    \Lambda\left(P_{k}\right)(t,x,\sigma_{u})=& Mod_{k,0}(t,x)\\
    &{+}4 a_{{+}\infty}p_{\zeta}(t)\sqrt{\omega} e^{{-}2\sqrt{\omega}d(t)} Sym\left[e^{i\alpha(t,\cdot,\sigma_{u}(t))}\tau_{\zeta_{k}(t)+p_{\zeta}(t)}F^{''}(\phi_{\omega}(\cdot)^{2})\phi_{\omega}(\cdot)^{2}e^{{-}\sqrt{\omega}(\cdot)}\right](x)\\  &{+}2a_{{+}\infty}p_{\zeta}(t)\sqrt{\omega}e^{{-}2\sqrt{\omega}d(t)}Sym\Big[e^{{-}i(v_{k}+p_{v})(\cdot)}e^{i\alpha(t,\cdot,\sigma_{u})}\tau_{\zeta_{k}(t)+p_{\zeta}(t)}F^{'}(\phi_{\omega}(\cdot)^{2})e^{{-}\sqrt{\omega}(\cdot)}\Big](x)
       \\&{+}\sum_{j\in J_{k}} r_{j}(v,t)Sym\left(e^{\alpha(t,\cdot,\sigma_{u}(t))}\tau_{\zeta_{k}(t)+p_{\zeta}(t)}\rho_{j}(\cdot)\right)(x)\\&{+}O\left(e^{{-}20k\sqrt{\omega}\zeta_{k}}+\norm{r(t)}_{H^{1}}v^{2}\left(\ln{\frac{1}{v}}\right)^{c}\right),\end{aligned}
\end{multline}
such that all functions $\rho_{j}$ are in $\mathcal{S}^{+}_{\infty},$
 all functions $r_{j}$ satisfy
\begin{equation}\label{rjjjjj}
\left\vert\frac{\partial^{l}}{\partial t^{l}}r_{j}(v,t)\right\vert\lesssim_{l} v^{2k+2+l}\left(\vert t\vert v+\ln{\frac{1}{v}}\right)^{c(k)}e^{{-}2\sqrt{\omega}\vert t\vert v},
\end{equation}
for all $l\in\mathbb{N},$ and $Mod_{k,0}(t,x)$ is given by
\begin{multline}\label{Modestim}
   \begin{aligned} 
    Mod_{k,0}(t,x)=&i\dot p_{v}(t) \partial_{v}P_{k}(t,x,\sigma_{u})+i(\dot p_{\zeta}(t) -p_{v}(t))\partial_{\zeta}P_{k}(t,x,\sigma_{u})\\&{+}i\dot p_{\gamma}(t)\partial_{\gamma}P_{k}(t,x,\sigma_{u})+i\dot p_{\omega}(t)Sym\left(e^{i\alpha(t,\cdot,\sigma_{u})}\tau_{\zeta_{k}+p_{\zeta}}\partial_{\omega}\phi_{\omega}(t,\cdot)\right)(x)\\
    &{-}p_{\omega}(t)Sym\left(e^{i\alpha(t,\cdot,\sigma_{u})}\tau_{\zeta_{k}(t)+p_{\zeta}(t)}S_{\omega}\left(\partial_{\omega}\phi_{\omega}\right)(\cdot)\right).
\end{aligned}
\end{multline}

\par Moreover, to simplify our notation, we consider the following self-adjoint operator on the Hilbert Space $L^{2}(\mathbb{R},\mathbb{C})$ for the dot product \eqref{dotproduct}
\begin{multline}\label{linearr}
  \begin{aligned}  
    \mathcal{L}(\rho)(t,x)=&{-}\partial^{2}_{x}\rho(t,x)+\omega \rho(t,x)\\
    &{-}\left[F^{'}\left(\phi_{\omega}(x-(\zeta_{k}+p_{\zeta}))^{2}\right)+F^{'}\left(\phi_{\omega}(x+(\zeta_{k}+p_{\zeta}))^{2}\right)\right]\rho(t,x)\\
&{-}F^{''}\left(\phi_{\omega}(x-(\zeta_{k}+p_{\zeta}))^{2}\right)\phi_{\omega}(x-(\zeta_{k}+p_{\zeta}))^{2}\rho(t,x)\\
&{-}F^{''}\left(\phi_{\omega}(x+(\zeta_{k}+p_{\zeta}))^{2}\right)\phi_{\omega}(x+(\zeta_{k}+p_{\zeta}))^{2}\rho(t,x)\\
&{-}F^{''}\left(\phi_{\omega}(x-(\zeta_{k}+p_{\zeta}))^{2}\right)\phi_{\omega}(x-(\zeta_{k}+p_{\zeta}))^{2}\overline{\rho(t,x)}e^{2i\alpha(t,x,\sigma_{u})}\\
&{-}F^{''}\left(\phi_{\omega}(x+(\zeta_{k}+p_{\zeta}))^{2}\right)\phi_{\omega}(x+(\zeta_{k}+p_{\zeta}))^{2}\overline{\rho(t,x)}e^{2i\alpha(t,{-}x,\sigma_{u})}.
\end{aligned} 
\end{multline}
\par Consequently, we have the following estimate 
\begin{multline}\label{eqfinal1}
     ie^{i(\gamma_{k}(t)+p_{\gamma}(t))}\partial_{t}r(t,x)-\mathcal{L}\left(e^{i(\gamma_{k}+p_{\gamma})}r(t,x)\right)
    \\
    \begin{aligned}
    =&{-}Mod_{k,0}(t,x)\\&{-}4 a_{{+}\infty}p_{\zeta}(t)\sqrt{\omega} e^{{-}2\sqrt{\omega}d(t)}Sym\left[e^{i\alpha(t,\cdot,\sigma_{u}(t))}\tau_{\zeta_{k}(t)+p_{\zeta}(t)}F^{''}(\phi_{\omega}(\cdot)^{2})\phi_{\omega}(\cdot)^{2}e^{{-}\sqrt{\omega}(\cdot)}\right](x)\\  &{-}2a_{{+}\infty}p_{\zeta}(t)\sqrt{\omega}e^{{-}2\sqrt{\omega}d(t)}Sym\left[e^{{-}i(v_{k}+p_{v})(\cdot)}e^{i\alpha(t,\cdot,\sigma_{u})}\tau_{\zeta_{k}(t)+p_{\zeta}(t)}F^{'}(\phi_{\omega}(\cdot)^{2})e^{{-}\sqrt{\omega}(\cdot)}\right](x)
       \\&{+}O\left(\norm{r(t)}_{H^{1}_{x}(\mathbb{R})}^{2}+\norm{r(t)}_{H^{1}_{x}(\mathbb{R})}v^{2}\left(\ln{\frac{1}{v}}\right)^{c}+\norm{r(t)}_{H^{1}_{x}(\mathbb{R})}\vert \dot p_{\gamma}(t) \vert+v^{2k+2}\left(\vert t\vert v+\ln{\frac{1}{v}}\right)^{c(k)}e^{{-}2\sqrt{\omega}v\vert t\vert }\right).
\end{aligned}
\end{multline}

\par Therefore, from Taylor's Expansion Theorem and Lemma \ref{ProP}, we deduce using identities \eqref{somega1}, \eqref{Modestim} that $ ie^{i(\gamma_{k}(t)+p_{\gamma}(t))}\partial_{t}r(t,x)-\mathcal{L}\left(e^{i(\gamma_{k}(t)+p_{\gamma}(t))}r(t,x)\right)$ is equal to %put thr estimate of Pi_{1} of F^{'} e^{{-}\sqrt{\omega}} before what is written below ok
\begin{multline}\label{derivativenergy}
   \begin{aligned} 
    &{-}\left(i\dot p_{\gamma}(t)+\frac{\zeta_{k}(t)}{4}i\dot p_{v}(t)\right)Sym\left(\varphi_{\gamma,k}(t,\cdot)\right)(x)-i\dot p_{v}(t)Sym\left(\varphi_{v,k}(t,\cdot)\right)(x)\\&{+}i(\dot p_{\zeta}(t)-p_{v}(t))\left(\varphi_{\zeta,k}(t,\cdot)\right)(x)-i\dot p_{\omega}(t)Sym\left(\varphi_{\omega,k}(t,\cdot)\right)(x)+ip_{\omega}(t)Sym\left(\varphi_{\gamma,k}(t,\cdot)\right)(x)\\&{-}4 a_{{+}\infty}p_{\zeta}(t)\sqrt{\omega} e^{{-}2\sqrt{\omega}d(t)}Sym\left[e^{i\alpha(t,\cdot,\sigma_{u}(t))}\tau_{\zeta_{k}(t)+p_{\zeta}(t)}\left[F^{''}(\phi_{\omega}(\cdot)^{2})\phi_{\omega}(\cdot)^{2}e^{{-}\sqrt{\omega}(\cdot)}\right]\right](x)\\ &{-}2a_{{+}\infty}p_{\zeta}(t)\sqrt{\omega}e^{{-}2\sqrt{\omega}d(t)}Sym\left[e^{i\alpha(t,\cdot,\sigma_{u}(t))}\tau_{\zeta_{k}(t)+p_{\zeta}(t)}\left[F^{'}(\phi_{\omega}(\cdot)^{2})e^{{-}\sqrt{\omega}(\cdot)}\right]\right](x)\\
    \\&{+}2a_{{+}\infty}p_{\zeta}(t)\sqrt{\omega}e^{{-}2\sqrt{\omega}d(t)}Sym\left[e^{i\alpha(t,\cdot,\sigma_{u}(t))}\tau_{\zeta_{k}(t)+p_{\zeta}(t)}\left[(i(v_{k}(t)+p_{v}(t))(\cdot))F^{'}(\phi_{\omega}(\cdot)^{2})e^{{-}\sqrt{\omega}(\cdot)}\right]\right](x)\\
&{+}2a_{{+}\infty}p_{\zeta}(t)\sqrt{\omega}e^{{-}2\sqrt{\omega}d(t)}v_{k}(t)\zeta_{k}(t)Sym\left[e^{i\alpha(t,\cdot,\sigma_{u}(t))}\tau_{\zeta_{k}(t)+p_{\zeta}(t)}\left[iF^{'}(\phi_{\omega}(\cdot)^{2})e^{{-}\sqrt{\omega}(\cdot)}\right]\right](x)
\\&{+}O\left(\norm{r(t)}_{H^{1}_{x}(\mathbb{R})}^{2}+\norm{r(t)}_{H^{1}_{x}(\mathbb{R})}v^{2}\left(\ln{\frac{1}{v}}\right)^{c}+\norm{r(t)}_{H^{1}_{x}(\mathbb{R})}\vert \dot p_{\gamma}(t) \vert+v^{2k+2}\left(\vert t\vert v+\ln{\frac{1}{v}}\right)^{c(k)}e^{{-}2\sqrt{\omega}v\vert t\vert }\right).
\end{aligned}
\end{multline}
Next, we are going to use the expression \eqref{eqfinal1} and Lemma \ref{Modulation} to obtain high precision in the estimates of the derivatives of the modulation parameters $p_{\zeta},\,p_{v},\,p_{\gamma}$ and $p_{\omega}.$

\begin{lemma}\label{dynamicsmod}
 Let $u,\,v\in (0,1)$ and $p_{\zeta},\,p_{v},\,p_{\gamma},\,p_{\omega},\,T$ be the same as in the statement of Lemma \ref{Modulation}, and $r$ be the unique function satisfying \eqref{uformula}. There exist a positive constants $K,\,c_{0}$ and values $C_{1,\omega}>0,\,C_{2,\omega},\,C_{3,\omega}$ depending on $\omega>0$ such that if $\norm{r(t)}_{H^{1}}\leq v^{2}$ and $\vert t\vert \leq T$ then, for
 \begin{equation*}
    D_{k}(t)=\left(\norm{r(t)}_{H^{1}_{x}(\mathbb{R})}v^{2}\left(\ln{\frac{1}{v}}\right)^{c_{0}}+v^{2k+2}\left(\ln{\frac{1}{v}}\right)^{c(k)}\right)
 \end{equation*}
 and any $0<v<\delta_{k},$ we have
 \begin{align} \label{modd1}
     \left\vert\dot p_{v}(t)+C_{1,\omega} p_{\zeta}(t)e^{{-}2\sqrt{\omega}d(t)}\right\vert\leq& K D_{k}(t),\\
     \left\vert \dot p_{\zeta}(t)-p_{v}(t) \right\vert\leq & K D_{k}(t),\\
     \left\vert \dot p_{\gamma}(t)-p_{\omega}(t)-C_{2,\omega}p_{\zeta}(t)e^{{-}2\sqrt{\omega}d(t)}-C_{3,\omega} d(t)p_{\zeta}(t)e^{{-}\sqrt{2}d(t)} \right\vert\leq & KD_{k}(t),\\
     \left\vert \dot p_{\omega}(t) \right\vert\leq & K D_{k}(t).
 \end{align}
The constants $c_{0},\,K$ do not depend on $k,$ while $c(k)$ depends only on $k.$
\end{lemma}
\begin{remark}\label{formrrem}
Furthermore, since the expression $ ie^{i(\gamma_{k}+p_{\gamma})}\partial_{t}r(t,x)-\mathcal{L}\left(e^{i(\gamma_{k}+p_{\gamma})}r(t,x)\right)$ is equal to \eqref{derivativenergy},we can verify that the real numbers $C_{1,\omega}>0,\,C_{2,\omega},\,C_{3,\omega}$ denoted in the statement of Lemma \ref{dynamicsmod} satisfy
\begin{multline*}
   ie^{i(\gamma_{k}+p_{\gamma})}\partial_{t}r(t,x)-\mathcal{L}\left(e^{i(\gamma_{k}+p_{\gamma})}r(t,x)\right)\\
    \begin{aligned}
    =&{-}\left[\dot p_{v}(t)+C_{1,\omega}p_{\zeta}(t)e^{{-}2\sqrt{\omega}d(t)}\right]iSym\left(\varphi_{v,k}(t,\cdot)\right)(x)\\
    &{+}\left[\dot p_{\zeta}(t)-p_{v}(t)\right]iSym\left(\varphi_{\zeta,k}(t,\cdot)\right)(x)\\
    &{-}\left[\dot p_{\gamma}(t)-p_{\omega}(t)-C_{2,\omega}p_{\zeta}(t)e^{{-}2\sqrt{\omega}d(t)}-C_{3,\omega} d(t)p_{\zeta}(t)e^{{-}\sqrt{2}d(t)} \right]iSym\left(\varphi_{\gamma,k}(t,\cdot)\right)(x)\\
    &{-}\dot p_{\omega}(t)iSym\left(\varphi_{\omega,k}(t,\cdot)\right)(x)+O(D_{k}(t)).
    \end{aligned}
\end{multline*}
The expression inside $O$ means a function $\rho(t,x)$ satisfying $\norm{\rho(t,x)}_{H^{1}_{x}(\mathbb{R})}\leq K D_{k}(t),$ for some $K>0.$ 
This estimate will be useful in the energy estimate of the norm of the remainder of the approximate solution in the next subsection. 
\end{remark}
\begin{proof}[Proof of Lemma \ref{dynamicsmod}]
The proof that the parameters $p_{\zeta},\,p_{v},\,p_{\gamma},\,p_{\omega}$ are of class $C^{1}$ is similar to the proof of Theorem $11$ from \cite{first}, this follows from the time derivative of the equations \eqref{orthoid1}, \eqref{orthoid2}, \eqref{orthoid3}, \eqref{orthoid4} plus the fact that $u$ is a strong solution of \eqref{NLS3} satisfying \eqref{upop}.

\par Next, from Lemma \ref{invertexpfor} and Remark \ref{remestimate},
we can verify using identities \eqref{somega1}, \eqref{somega2} and Lemma \ref{interactt} the existence of a $c>0$ satisfying 
\begin{align}\label{id1}
    i\partial_{t}\varphi_{\gamma,k}(t,x)-\mathcal{L}\left(\varphi_{\gamma,k}(t,x)\right)+\omega \varphi_{\gamma,k}(t,x) =&O\left(v^{2}\left(\ln{\frac{1}{v}}\right)^{c}+\max_{ j}\vert \dot p_{ j}(t)\vert\right),\\
    i\partial_{t}\varphi_{\zeta,k}(t,x)-\mathcal{L}\left(\varphi_{\zeta,k}(t,x)\right)+\omega \varphi_{\zeta,k}(t,x) =&O\left(v^{2}\left(\ln{\frac{1}{v}}\right)^{c}+\max_{ j}\vert \dot p_{ j}(t)\vert\right),\\
    i\partial_{t}\varphi_{\omega,k}(t,x)-\mathcal{L}\left(\varphi_{\omega,k}(t,x)\right)+\omega \varphi_{\omega,k}(t,x) =&{-}i\varphi_{\gamma,k}(t,x)+O\left(v^{2}\left(\ln{\frac{1}{v}}\right)^{c}+\max_{ j}\vert \dot p_{ j}(t)\vert\right),\\
    i\partial_{t}\varphi_{v,k}(t,x)-\mathcal{L}\left(\varphi_{v,k}(t,x)\right)+\omega \varphi_{v,k}(t,x) =&i\varphi_{\zeta,k}(t,x)+O\left(v^{2}\left(\ln{\frac{1}{v}}\right)^{c}+\max_{ j}\vert \dot p_{ j}(t)\vert\right),
\end{align}
for all $v>0$ sufficiently small.
\par Consequently, using integration by parts, Remark \ref{remestimate}, and the modulation equations \eqref{orthoid1}, \eqref{orthoid2}, we obtain the following estimate
\begin{multline}\label{ultimate}
   \begin{aligned}
   \frac{d}{dt}\left\langle e^{i(\gamma_{k}+p_{\gamma})}r(t), i\varphi_{j,k}(t,x)\right\rangle=&\left\langle {-}ie^{i(\gamma_{k}+p_{\gamma})}\partial_{t}r(t)+\mathcal{L}\left(e^{i(\gamma_{k}+p_{\gamma})}r(t)\right), \varphi_{j,k}(t,x)\right\rangle \\&{+}O\left(\norm{r(t)}_{H^{1}_{x}(\mathbb{R})}v^{2}\left(\ln{\frac{1}{v}}\right)^{c}+\norm{r(t)}_{H^{1}_{x}(\mathbb{R})}\max_{ j}\vert \dot p_{ j}(t)\vert\right),
\end{aligned}
\end{multline}
for any $j\in\{\zeta,v,\gamma,\omega\}.$
\par Next, we recall the following estimate obtained from integration by parts 
\begin{multline}
\int_{\mathbb{R}} 4\sqrt{\omega}F^{''}(\phi_{\omega}(x)^{2})\phi_{\omega}(x)^{2}\phi^{'}_{\omega}(x)e^{{-}\sqrt{\omega}x}\,dx+2\sqrt{\omega}F^{'}(\phi_{\omega}(x)^{2})\phi^{'}_{\omega}(x)e^{{-}\sqrt{\omega}x}\,dx\\=2\omega\int_{\mathbb{R}} F^{'}(\phi_{\omega}(x)^{2})\phi_{\omega}(x)e^{{-}\sqrt{\omega} x}\,dx=C_{\omega}>0.
\end{multline}
Consequently, using the identity above, the estimate \eqref{derivativenergy}, and \eqref{ultimate} for $j=\zeta$ with Lemma \ref{interactt}, we deduce the existence of positive constants $C_{1,\omega}$ and $c$ satisfying
\begin{equation}\label{modode1}
    \dot p_{v}(t)+C_{1,\omega}p_{\zeta}(t)e^{{-}2\sqrt{\omega}d(t)}=O\left(v^{2}\left(\ln{\frac{1}{v}}\right)^{c}\max_{j}\vert \dot p_{j}(t) \vert +v^{4}\left(\ln{\frac{1}{v}}\right)^{c}\vert p_{\zeta}(t) \vert+D_{k}(t)\right).
\end{equation}
\par Similarly, we can verify using the estimate \eqref{derivativenergy} or using the estimate \ref{formrrem} of $ ie^{i(\gamma_{k}+p_{\gamma})}\partial_{t}r(t,x)-\mathcal{L}\left(e^{i(\gamma_{k}+p_{\gamma})}r(t,x)\right)$ from Remark \ref{formrrem}, and the identity \eqref{ultimate} for $j=v$ with Lemma \ref{interactt} that
\begin{equation}\label{modode2}
    \dot p_{\zeta}(t)-p_{v}(t)=O\left(v^{2}\left(\ln{\frac{1}{v}}\right)^{c}\max_{j}\vert \dot p_{j}(t) \vert +v^{3}\left(\ln{\frac{1}{v}}\right)^{c}\vert p_{\zeta}(t) \vert+D_{k}(t)\right),
\end{equation}
and the term of order $v^{3}\left(\ln{\frac{1}{v}}\right)^{c}\vert p_{\zeta}(t) \vert$ in the estimate above was obtained using the elementary inequality
\begin{equation}\label{eiest}
\Ree\int_{\mathbb{R}}e^{i(v_{k}+p_{v})x}ih(x)\,dx=O\left(v+\norm{r(t)}_{H^{1}_{x}(\mathbb{R})}\left(\ln{\frac{1}{v}}\right)^{c} \right) \text{, for any real-valued Schwartz function $h,$}
\end{equation}
which follows from Lemma \ref{Modulation} and the following estimate
\begin{align*}
\Ree\int_{\mathbb{R}}e^{i(v_{k}+p_{v})x}ih(x)\,dx=&\Ree\int_{\mathbb{R}}ih(x)\,dx-\Ree\int_{\mathbb{R}}(v_{k}+p_{v})xh(x)\,dx+O\left(\left[v_{k}^{2}+p_{v}^{2}\right]\norm{h(x)x^{2}}_{L^{1}_{x}(\mathbb{R})}\right)\\
=&{-}(v_{k}+p_{v})\Ree\int_{\mathbb{R}}xh(x)\,dx+O\left(\left[v_{k}^{2}+p_{v}^{2}\right]\norm{h(x)x^{2}}_{L^{1}_{x}(\mathbb{R})}\right),
\end{align*}
obtained from Taylor's Expansion Theorem.
\par Furthermore, using estimate \eqref{derivativenergy} of $ie^{i(\gamma_{k}+p_{\gamma})}\partial_{t}r(t,x)-\mathcal{L}\left(e^{i(\gamma_{k}+p_{\gamma})}r(t,x)\right),$ equation \eqref{ultimate} for $j=\omega$ and Lemma \ref{interactt}, we obtain the existence of a constant $c>0$ satisfying
\begin{multline}\label{modode3}
    \dot p_{\gamma}(t)-p_{\omega}(t)-C_{2,\omega}p_{\zeta}(t)e^{{-}2\sqrt{\omega}d(t)}-C_{3,\omega} d(t)p_{\zeta}(t)e^{{-}\sqrt{2}d(t)}\\
    =O\left(v^{2}\left(\ln{\frac{1}{v}}\right)^{c}\max_{j}\vert \dot p_{j}(t) \vert +v^{4}\left(\ln{\frac{1}{v}}\right)^{c}\vert p_{\zeta}(t) \vert+D_{k}(t)\right).
\end{multline}
\par Finally, by similar reasoning to the estimate \eqref{modode1}, using \eqref{ultimate} for $j=\gamma$ and estimate \eqref{derivativenergy} of $ ie^{i(\gamma_{k}+p_{\gamma})}\partial_{t}r(t,x)-\mathcal{L}\left(e^{i(\gamma_{k}+p_{\gamma})}r(t,x)\right),$ we can verify the existence of a constant $c>0$ satisfying
\begin{equation}\label{modode4}
    \dot p_{\omega}(t)=O\left(v^{2}\left(\ln{\frac{1}{v}}\right)^{c}\max_{j}\vert \dot p_{j}(t) \vert +v^{3}\left(\ln{\frac{1}{v}}\right)^{c}\vert p_{\zeta}(t) \vert+D_{k}(t)\right).
\end{equation}
\par In conclusion, the result of Lemma \ref{dynamicsmod} follow from estimates \eqref{modode1}, \eqref{modode2}, \eqref{modode3}, \eqref{modode4} and Lemma \ref{Modulation}.
\end{proof}

\subsection{Energy Estimate Functional}
\par The main result of this subsection is the following theorem.
\begin{theorem}\label{energyestimatetheor}
 There exist $C>1,$ for any $l\in\mathbb{N},\,k_{0}\in\mathbb{N}$ there exist $C_{l}>1,$ and for any $k\in\mathbb{N}_{\geq k_{0}},$ there exist $c(k),\,c(k,l)>0$ and $\delta_{k,l}\in(0,1)$ such that if $0<v<\delta_{k,l}$ and for $  \frac{30 k\ln{\frac{1}{v}}}{\sqrt{\omega}v}<T_{0,k}<\frac{\left(\ln{\frac{1}{v}}\right)^{\frac{5}{4}}}{v}$ 
 \begin{equation}\label{hypoo}
     \norm{(1+ x^{2l})^{\frac{1}{2}}\left[u(T_{0,k})-P_{k}(T_{0,k},x,\sigma_{k}(T_{0,k}))\right]}_{H^{2l}}<v^{10k},
 \end{equation}
then the unique solution of 
\begin{equation*}
     iu_{t}+u_{xx} +F^{'}(\vert u\vert^{2})u=0
\end{equation*}
satisfies
\begin{equation}\label{theenergyestimate}
    \norm{u(t)-P_{k}(t,x,\sigma_{k}(t))}_{H^{1}}\leq C v^{2k}\left(\ln{\frac{1}{v}}\right)^{c(k)+1}\exp\left(\frac{Cv\left\vert t-T_{0,k} \right\vert}{\ln{\frac{1}{v}}}\right),
\end{equation}
\begin{equation}\label{rrr}
    v^{2l}\norm{x^{l}\left[u(t)-P_{k}(t,x,\sigma_{k}(t))\right]}_{H^{1}}\leq C_{l} v^{2k}\left(\ln{\frac{1}{v}}\right)^{c(k,l)}\exp\left(\frac{C_{l}v\left\vert t-T_{0,k} \right\vert}{\ln{\frac{1}{v}}}\right),
\end{equation}
for all $t$ satisfying $\left\vert t -T_{0,k}\right\vert\leq \frac{\left(\ln{\frac{1}{v}}\right)^{\frac{4}{3}}}{v}.$
\end{theorem}

\begin{remark}
 The choosing of $\frac{4}{3}$ in the exponent of $\ln{\frac{1}{v}}$ for the time interval is arbitrary, the statement would still be true for any $\theta$ slightly larger than $1$ after selecting an appropriate small $\delta_{k}>0.$ Moreover, the constant $c(k)$ depends only on the function $\varphi_{k}$ and it is the same for all $v$ sufficiently small. 
\end{remark}
From now on, we denote the solution $u(t,x)$ of the Schrödinger equation \eqref{NLS3} satisfying the initial condition \eqref{initialcondition} by 
\begin{equation}\label{rmod}
u(t,x)=P_{k}\left(t,x,\sigma_{u}(t)\right)+e^{i\gamma_{k}(t)}r(t,x)
\end{equation}
for all $t $ satisfying $\left\vert t-T_{0,k} \right\vert\leq \frac{\left(\ln{\frac{1}{v}}\right)^{\frac{4}{3}}}{v}.$ such that $r(t)$ satisfies the orthogonality conditions from \eqref{orthoid1} through \eqref{orthoid4}.
\par The proof of Theorem \ref{energyestimatetheor} will follow from the study of a Lyapunov function obtained from a perturbation of the quadratic form

\begin{align}\label{Quadratic energy}
    L(t,r)=&\int_{\mathbb{R}}\left\vert \partial_{x}r(t,x)\right\vert^{2}+\omega \left\vert r(t,x)\right\vert^{2}\,dx\\ \nonumber
    &{-}\int_{\mathbb{R}}F^{'}\left(\phi_{\omega}(x-\zeta_{k}(t))^{2}\right)\left\vert  r(t,x)\right\vert^{2}+F^{'}\left(\phi_{\omega}(x+\zeta_{k}(t))^{2}\right)\left\vert  r(t,x)\right\vert^{2}\,dx\\ \nonumber
    &{-}\Ree\int_{\mathbb{R}} F^{''}\left(\phi_{\omega}(x-\zeta_{k})^{2}\right)\phi_{\omega}(x-\zeta_{k})^{2}e^{2i\alpha_{k}(t,x)}\overline{r(t,x)^{2}}e^{{-}2i\gamma_{k}(t)}\,dx\\ \nonumber
    &{-}\Ree\int_{\mathbb{R}}F^{''}\left(\phi_{\omega}(x+\zeta_{k})^{2}\right)\phi_{\omega}(x+\zeta_{k})^{2}e^{2i\alpha_{k}(t,{-}x)}\overline{r(t,x)^{2}}e^{{-}2i\gamma_{k}(t)}\,dx\\ \nonumber
    &{-}\int_{\mathbb{R}} F^{''}\left(\phi_{\omega}(x-\zeta_{k})^{2}\right)\phi_{\omega}(x-\zeta_{k})^{2}\vert r(t,x)\vert^{2}\,dx\\ \nonumber
    &{-}\int_{\mathbb{R}} F^{''}\left(\phi_{\omega}(x+\zeta_{k})^{2}\right)\phi_{\omega}(x+\zeta_{k})^{2}\vert r(t,x)\vert^{2}\,dx.
\end{align}
Since $r$ satisfies the orthogonality conditions of Lemma \ref{Modulation}, we can verify that the operator $L$ is positive on $r\in H^{1}_{x}(\mathbb{R},\mathbb{C}).$
\par Furthermore, to study the coercivity of $L$ and the growth of its derivative, we consider from now on a cut-off smooth function $\chi$ satisfying
\begin{equation}\label{chi}
    \chi(x)=\begin{cases}
        0,\text{ $x\geq \frac{6}{10},$}\\
        1, \text{ $x\leq \frac{1}{2},$}
    \end{cases}
    \,\text{ and } 0< \chi(x)< 1 \text{, if $\frac{1}{2}<x<\frac{6}{10},$} 
\end{equation}
and we define the following two functions
\begin{equation}\label{chi12}
    \chi_{1}(t,x)\coloneqq\chi\left(\frac{x+\zeta_{k}(t)}{2\zeta_{k}(t)}\right),\, \chi_{2}(t,x)\coloneqq 1-\chi_{1}(t,x),
\end{equation}
for all $(t,x)\in\mathbb{R}^{2}.$
\begin{lemma}[Coercivity Lemma]\label{coerc}
There exist constants $\delta,\,c>0$ such that if $0<v<\delta_{k}$ and $\norm{r}_{H^{1}_{x}(\mathbb{R})}\leq \delta$ and the function $r\in H^{1}_{x}(\mathbb{R},\mathbb{C})$ satisfies the identities \eqref{orthoid1}, \eqref{orthoid2}, \eqref{orthoid3}, \eqref{orthoid4}, then
\begin{equation}\label{coercstrong}
    L(t,r)\geq c\norm{r(t)}_{H^{1}}^{2}.
\end{equation}
    
\end{lemma}
\begin{proof}
\par First, we consider the functions
\begin{align*}
r_{1}(t,x)=r(t,x)\chi_{1}(t,x),\,\,r_{2}(t,x)=r(t,x)\chi_{2}(t,x).
\end{align*}
Moreover, using Remark \ref{remestimate}, we also observe for any $j\in\{1,2\}$ that
\begin{equation*}
   \left\vert \partial_{x}\chi_{j}(t,x) \right\vert=O\left(\frac{1}{\ln{\frac{1}{v}}}\right).
\end{equation*}
\par Next, since the soliton $\phi_{\omega}$ satisfies the inequality
\begin{equation*}
    \vert \phi_{\omega}(x) \vert=O\left( e^{{-}\sqrt{\omega}\vert x\vert}\right),
\end{equation*}
there is a constant $K>0$ such that
\begin{equation}\label{intchi}
    \left\vert\chi_{1}(t,x)\phi_{\omega}(x-\zeta_{k}(t))\right\vert+\left\vert \chi_{2}(t,x)\phi_{\omega}(x+\zeta_{k}(t))\right\vert\leq K e^{{-}\sqrt{\omega}\frac{4d(t)}{5}},
\end{equation}
because $\zeta_{k}(t)$ satisfies the Remark \ref{remestimate} when $v>0$ is small enough. Therefore, we can verify that the function $L(t,r)$ can be estimated as
\begin{equation}\label{sepL}
    L(t,r)=L_{1}(t)+L_{2}(t)+2\int_{\mathbb{R}}\chi_{1}(t,x)\chi_{2}(t,x)\left[\vert \partial_{x}r(t,x) \vert^{2}+\vert r(t,x) \vert^{2}\right]\,dx+O\left(\frac{1}{\ln{\frac{1}{v}}}\norm{r(t)}_{H^{1}}^{2}\right),
\end{equation}
such that 
\begin{multline*}
\begin{aligned}  L_{j}(t)=&\int_{\mathbb{R}}\left\vert \partial_{x}r_{j}(t,x)\right\vert^{2}+\omega \left\vert r_{j}(t,x)\right\vert^{2}\,dx\\ \nonumber
&{-}\int_{\mathbb{R}}F^{'}\left(\phi_{\omega}(x-({-}1)^{j}\zeta_{k}(t))^{2}\right)\left\vert  r_{j}(t,x)\right\vert^{2}\,dx\\ \nonumber
    &{-}\Ree\int_{\mathbb{R}} F^{''}\left(\phi_{\omega}(x-({-}1)^{j}\zeta_{k})^{2}\right)\phi_{\omega}(x-\zeta_{k})^{2}e^{2i\alpha_{k}(t,({-}1)^{j}x)}\overline{r_{j}(t,x)^{2}}e^{{-}2i\gamma_{k}(t)}\,dx\\ \nonumber
&{-}\int_{\mathbb{R}} F^{''}\left(\phi_{\omega}(x-({-}1)^{j}\zeta_{k})^{2}\right)\phi_{\omega}(x-({-}1)^{j}\zeta_{k})^{2}\vert r_{j}(t,x)\vert^{2}\,dx.
\end{aligned}
\end{multline*}
Moreover, since $\chi_{j}\geq 0$ for any $j\in\{1,2\},$ we have the following inequality
\begin{equation*}
   \int_{\mathbb{R}}\chi_{1}(t,x)\chi_{2}(t,x)\left[\vert \partial_{x}r(t,x) \vert^{2}+\vert r(t,x) \vert^{2}\right]\geq 0.
\end{equation*}
\par Furthermore, using the definition \eqref{alpha} of $\alpha,$ Remark \ref{remestimate} when $v>0$ is small enough, and estimate \eqref{intchi}, it is not difficult to verify for any $j\in\{1,2\}$ that the function
\begin{equation*}
    \hat{r}_{j}(t,x-({-}1)^{j}\zeta_{k}(t))\coloneqq e^{{-}i\alpha(t,({-}1)^{j}x)+i\gamma_{k}(t)}r_{j}(t,x)
\end{equation*}
satisfies $L_{j}(t)=\left\langle S_{\omega}(\hat{r}_{j}(t,x)),\hat{r}_{j}(t,x) \right\rangle+O(v\norm{r_{j}}_{H^{1}}^{2}),\,\norm{\hat{r}_{j}(t)}_{H^{1}}^{2}=\norm{r_{j}(t)}_{H^{1}}^{2}+O(v\norm{r_{j}}_{H^{1}}^{2}),$ and
\begin{align*}
\left\langle \hat{r}_{j}(t,x),x\phi_{\omega}(x)\right\rangle=&O\left(v^{\frac{3}{4}}\norm{r(t)}_{H^{1}}\right),\,
\left\langle \hat{r}_{j}(t,x),i\partial_{\omega}\phi_{\omega}(x)\right\rangle=O\left(v^{\frac{3}{4}}\norm{r(t)}_{H^{1}}\right).
\end{align*}

Therefore, we obtain from Remark \ref{coerremark} of Lemma \ref{coer} using Young Inequality the existence of constants $c,\,K>0$ such that if $v>0$ is sufficiently small, then
\begin{equation}\label{cooo}
    L_{j}(t)\geq c\norm{r_{j}(t)}_{H^{1}}^{2}-Kv^{\frac{3}{2}}\norm{r(t)}_{H^{1}}^{2}, 
\end{equation}
for any $j\in\{1,2\}.$
\par In conclusion, since $\chi_{1}+\chi_{2}=1,$ \eqref{sepL} and \eqref{cooo} imply \eqref{coercstrong}.
\end{proof}

\par Moreover, using the inequality \eqref{impoco} from Remark \ref{coerremark}, we deduce the existence of a constant $C>1$ such that
\begin{equation}\label{cooo2}
    L_{j}(t)\geq c\norm{r_{j}(t)}_{H^{1}}^{2}-C\norm{\Pi_{1}(r_{j}(t))}_{L^{2}}^{2}.
\end{equation}
Consequently, using \eqref{cooo2} in the place of \eqref{cooo} in the proof of Lemma \ref{coerc}, we obtain the following proposition.
    \begin{lemma}\label{coerc2}
There exist constants $\delta,\,c>0,C>1$ such that if $0<v<\delta_{k}$ and $\norm{r(t)}_{L^{\infty}[T_{1},T_{2}],H^{1}_{x}(\mathbb{R})}\leq \delta,$ then for any $t\in [T_{1},T_{2}]$  \begin{align*}
    L(t,r)\geq& c\norm{r(t)}_{H^{1}}^{2}-C\sum_{\pm}\left\langle r(t),i\varphi_{\gamma,k}(t,\pm x)\right\rangle^{2}-C\sum_{\pm}\left\langle r(t),i\varphi_{\zeta,k}(t,\pm x) \right\rangle^{2}\\ 
&{-}C\sum_{\pm}\left\langle r(t), i\varphi_{\omega,k}(t,\pm x)\right\rangle^{2}-C\sum_{\pm}\left\langle r(t),i\varphi_{v,k}(t,\pm x)\right\rangle^{2}.
\end{align*}
 
\end{lemma}

\par Next, following the reasoning in \cite{third}, we are going to use the estimates of the modulation parameters from Lemmas \ref{Modulation}, \ref{dynamicsmod} and equation \eqref{eqfinal1} to estimate $\left\vert\partial_{t}L(t,r)\right\vert$ during a large time interval. More precisely, we are going to verify that the function $L(t,r)$ satisfies the following proposition.
\begin{lemma}\label{energyestimate1}
Let $L(t,r)$ be the function defined at \eqref{Quadratic energy} and let $c(k)$ be the same as in the statement of Theorem \ref{app lemma}. There exists a constant $c>0$ and $\delta_{k}\in(0,1)$ such that if $0<v<\delta_{k}$ and $\norm{r(t)}_{H^{1}}\leq v^{2},$ then
\begin{multline*}
\begin{aligned}
\partial_{t}L(t,r)=&2\dot\zeta_{k}(t)\int_{\mathbb{R}}F^{''}\left(\phi_{\omega}(x-\zeta_{k}(t))^{2}\right)\phi^{'}_{\omega}(x-\zeta_{k})\phi_{\omega}(x-\zeta_{k})\vert r(t,x) \vert^{2}\,dx\\&{-}2\dot\zeta_{k}(t)\int_{\mathbb{R}}F^{''}\left(\phi_{\omega}(x+\zeta_{k}(t))^{2}\right)\phi^{'}_{\omega}(x+\zeta_{k})\phi_{\omega}(x+\zeta_{k})\vert r(t,x) \vert^{2}\\
    &{+}\dot \zeta_{k}(t)\Ree \int_{\mathbb{R}} \frac{\partial}{\partial x}\left[F^{''}\left(\phi_{\omega}(x-\zeta_{k})^{2}\right)\phi_{\omega}(x-\zeta_{k})^{2}\right]\overline{r(t,x)^{2}}\,dx\\
    &{-}\dot \zeta_{k}(t)\Ree \int_{\mathbb{R}} \frac{\partial}{\partial x}\left[F^{''}\left(\phi_{\omega}(x+\zeta_{k})^{2}\right)\phi_{\omega}(x+\zeta_{k})^{2}\right]\overline{r(t,x)^{2}}\,dx\\
    &{+}\dot \zeta_{k}(t)\int_{\mathbb{R}}\frac{\partial}{\partial x}\left[F^{''}\left(\phi_{\omega}(x-\zeta_{k})^{2}\right)\phi_{\omega}(x-\zeta_{k})^{2}\right]\vert r(t,x)\vert^{2}\,dx\\
    &{-}\dot \zeta_{k}(t)\int_{\mathbb{R}}\frac{\partial}{\partial x}\left[F^{''}\left(\phi_{\omega}(x+\zeta_{k})^{2}\right)\phi_{\omega}(x+\zeta_{k})^{2}\right]\vert r(t,x)\vert^{2}\,dx\\&{+}O\left(v^{2}\left(\ln{\frac{1}{v}}\right)^{c}\norm{r(t)}_{H^{1}}^{2}+\norm{r(t)}_{H^{1}}v^{2k+2}\left(\ln{\frac{1}{v}}\right)^{c(k)}+\norm{r(t)}_{H^{1}}^{3}\right),
 \end{aligned}
\end{multline*}
for for all $t$ satisfying $\left\vert t- \frac{30 k\ln{\frac{1}{v}}}{\sqrt{\omega}v}\right\vert\leq \frac{\left(\ln{\frac{1}{v}}\right)^{\frac{4}{3}}}{v}.$ 
\end{lemma}
\begin{remark}
    The constant $c>0$ does not depend on $k,$ it comes from the first approximate solution when $k=2.$
\end{remark}
\begin{proof}[Proof of Lemma \ref{energyestimate1}]
First, using the definition of $L(t,r)$ in \eqref{Quadratic energy} and estimates \eqref{popopopo}, we deduce from the derivative of $L(t,r)$ that proving Lemma \ref{energyestimate1} is equivalent to verify that the following expression 
\begin{align}\label{Quadratic energydet}
    &2\Ree\int_{\mathbb{R}}{-}\partial^{2}_{x}r(t)\partial_{t}\overline{r(t)}+\omega r(t)\partial_{t}\overline{r(t)}\,dx\\ \nonumber
    &{-}2\int_{\mathbb{R}}F^{'}\left(\phi_{\omega}(x-\zeta_{k}(t))^{2}\right)r(t)\partial_{t}\overline{r(t)}+F^{'}\left(\phi_{\omega}(x+\zeta_{k}(t))^{2}\right)r(t)\partial_{t}\overline{r(t)}\,dx\\ \nonumber
    &{-}2\Ree\int_{\mathbb{R}} F^{''}\left(\phi_{\omega}(x-\zeta_{k})^{2}\right)\phi_{\omega}(x-\zeta_{k})^{2}e^{2i\alpha_{k}(t,x)}\overline{r(t)}\partial_{t}\overline{r(t)}e^{{-}2i\gamma_{k}(t)}\,dx\\ \nonumber
    &{-}2\Ree\int_{\mathbb{R}}F^{''}\left(\phi_{\omega}(x+\zeta_{k})^{2}\right)\phi_{\omega}(x+\zeta_{k})^{2}e^{2i\alpha_{k}(t,{-}x)}\overline{r(t)}\partial_{t}\overline{r(t)}e^{{-}2i\gamma_{k}(t)}\,dx\\ \nonumber
    &{-}2\Ree\int_{\mathbb{R}} F^{''}\left(\phi_{\omega}(x-\zeta_{k})^{2}\right)\phi_{\omega}(x-\zeta_{k})^{2}r(t)\partial_{t}\overline{r(t)}\,dx\\ \nonumber
    &{-}2\Ree\int_{\mathbb{R}} F^{''}\left(\phi_{\omega}(x+\zeta_{k})^{2}\right)\phi_{\omega}(x+\zeta_{k})^{2}r(t)\partial_{t}\overline{r(t)}\,dx
\end{align}
has a modulus of order
\begin{equation*}
    O\left(v^{2}\left(\ln{\frac{1}{v}}\right)^{c}\norm{r(t)}_{H^{1}_{x}(\mathbb{R})}^{2}+\norm{r(t)}_{H^{1}_{x}(\mathbb{R})}v^{2k+2}\left(\ln{\frac{1}{v}}\right)^{c(k)}\right).
\end{equation*}
\par Moreover, using the operator $\mathcal{L}$ defined in \eqref{linearr}, we obtain using Lemmas \ref{interactt} and \ref{taylorlemma}  that the expression \eqref{Quadratic energydet} is equal for some constant $c>0$ 
\begin{equation}\label{simplifiedderivative}
2\left\langle \mathcal{L}\left(e^{i\gamma_{k}(t)}r(t)\right),e^{i\gamma_{k}(t)}\partial_{t}r(t) \right\rangle+O\left(v^{2}\left(\ln{\frac{1}{v}}\right)^{c}\norm{r(t)}_{H^{1}}^{2}\right),
\end{equation}
if $v>0$ is small enough. 
\par Furthermore, using Remark \ref{formrrem} and integration by parts, we can verify that
\begin{multline}\label{almostlastestimate}
    \left\langle \mathcal{L}\left(e^{i\gamma_{k}(t)}r(t)\right),e^{i\gamma_{k}(t)}\partial_{t}r(t) \right\rangle\\
    \begin{aligned}
    =&{-}\left\langle \mathcal{L}\left(e^{i\gamma_{k}(t)}r(t)\right),\left[\dot p_{v}(t)+C_{1,\omega}p_{\zeta}(t)e^{{-}2\sqrt{\omega}d(t)}\right]Sym\left(\varphi_{v,k}(t,\cdot)\right)(x) \right\rangle\\
    &{+}\left\langle \mathcal{L}\left(e^{i\gamma_{k}(t)}r(t)\right),\left[\dot p_{\zeta}(t)-p_{v}(t)\right]Sym\left(\varphi_{\zeta,k}(t,\cdot)\right)(x) \right\rangle\\
    &{-}\Bigg\langle \mathcal{L}\left(e^{i\gamma_{k}(t)}r(t)\right),\big[\dot p_{\gamma}(t)-p_{\omega}(t)\\&{-}C_{2,\omega}p_{\zeta}(t)e^{{-}2\sqrt{\omega}d(t)}-C_{3,\omega}d(t)p_{\zeta}(t)e^{{-}2\sqrt{\omega}d(t)}\big]Sym\left(\varphi_{\gamma,k}(t,\cdot)\right)(x) \Bigg\rangle\\
&{-}\left\langle \mathcal{L}\left(e^{i\gamma_{k}(t)}r(t)\right),\dot p_{\omega}(t)Sym\left(\varphi_{\omega,k}(t,\cdot)\right)(x) \right\rangle
\\&{+}O\left(\norm{r(t)}_{H^{1}}^{2}v^{2}\left(\ln{\frac{1}{v}}\right)^{c}+\norm{r(t)}_{H^{1}}^{3}\right)\\
&{+}O\left(\norm{r(t)}_{H^{1}}^{2}v^{2k+2}\left(\ln{\frac{1}{v}}\right)^{c(k)}\right)\\
\end{aligned}
\end{multline}
\par In conclusion, we obtain from the estimate \eqref{almostlastestimate} and Lemma \ref{dynamicsmod} that
the expression \eqref{Quadratic energydet} has a modulus of order
\begin{equation*}
O\left(\norm{r(t)}_{H^{1}}v^{2k+2}\left(\ln{\frac{1}{v}}\right)^{c(k)}+v^{2}\left(\ln{\frac{1}{v}}\right)^{c}\norm{r(t)}_{H^{1}_{x}(\mathbb{R})}^{2}\right),
\end{equation*}
which implies the result of Lemma \ref{energyestimate1}. 
\end{proof}
\par In addition, it is possible to add functions of correction to $L(t,r)$ and obtain a new Lyapunov function having a derivative with a size smaller than $\left\vert \partial_{t}L(t,r)\right\vert,$ this approach was also used in the paper \cite{third} for study the collision between two kinks for the $\phi^{6}$ model. These functions of corrections are the following localized momentum quantities
\begin{align}\label{localmomentum}
    P_{1}(t,r)\coloneqq \I \int_{\mathbb{R}}\chi_{1}(t,x)\overline{r(t,x)}\partial_{x}r(t,x)\,dx,\,\, P_{2}(t,r)\coloneqq \I \int_{\mathbb{R}}\chi_{2}(t,x)\overline{r(t,x)}\partial_{x}r(t,x)\,dx.
\end{align}
\par However, before constructing the last Lyapunov function, we shall estimate the derivatives of $P_{1}(t,r)$ and $P_{2}(t,r).$
\begin{lemma}\label{dpj}
 Let $j\in\{1,2\},$ the function $P_{j}(t,r)$ satisfies
\begin{multline}
  \begin{aligned}
\partial_{t}P_{j}(t,r)    =&\Ree \int_{\mathbb{R}}\frac{\partial}{\partial x}\left[F^{'}\left(\phi_{\omega}(x+({-}1)^{j+1}(\zeta_{k}+p_{\zeta}))^{2}\right)\right]\left\vert r(t,x) \right\vert^{2}\,dx\\
&{+}\Ree \int_{\mathbb{R}}\frac{\partial}{\partial x}\left[F^{''}\left(\phi_{\omega}(x+({-}1)^{j+1}(\zeta_{k}+p_{\zeta}))^{2}\right)\phi_{\omega}(x+({-}1)^{j+1}(\zeta_{k}+p_{\zeta}))^{2}\right]\vert r(t,x)\vert^{2}\,dx\\
&{+}\Ree \int_{\mathbb{R}}\frac{\partial}{\partial x}\left[F^{''}\left(\phi_{\omega}(x+({-}1)^{j+1}(\zeta_{k}+p_{\zeta}))^{2}\right)\phi_{\omega}(x+({-}1)^{j+1}(\zeta_{k}+p_{\zeta}))^{2}\right]r(t,x)^{2}\,dx\\
&{+}O\left(v^{2k+2}\left(\ln{\frac{1}{v}}\right)^{c(k)}\norm{r(t)}_{H^{1}}+\frac{1}{\ln{\frac{1}{v}}}\norm{r(t)}_{H^{1}}^{2}+\norm{r(t)}_{H^{1}}^{3}\right).
\end{aligned}
\end{multline}
\end{lemma}
\begin{proof}[Proof of Lemma \ref{dpj}]
It is enough to verify the statement of Lemma \ref{dpj} for $j=1,$ the proof for the case when $j=2$ is completely analogous. \par First, from the definition of $P_{1}(t,r),$ we can verify using integration by parts that
\begin{equation*}
    \frac{\partial P_{1}(t,r)}{\partial t}=2\I \int_{\mathbb{R}}\chi_{1}(t,x)\overline{\partial_{t}r(t,x)}\partial_{x}r(t,x)\,dx+O\left(\frac{v\norm{r(t)}_{H^{1}}^{2}}{\ln{\frac{1}{v}}}\right).
\end{equation*}
\par Furthermore, from Lemma \ref{dynamicsmod} and Remark \ref{formrrem}, we have that
\begin{multline}\label{mainequationr}
\begin{aligned}
\partial_{t}r(t,x)=&i\partial^{2}_{x}r(t,x)-i\omega r(t,x)\\
&{+}i\left[F^{'}\left(\phi_{\omega}(x-(\zeta_{k}+p_{\zeta}))^{2}\right)+F^{'}\left(\phi_{\omega}(x+(\zeta_{k}+p_{\zeta}))^{2}\right)\right]r(t,x)\\
&{+}iF^{''}\left(\phi_{\omega}(x-(\zeta_{k}+p_{\zeta}))^{2}\right)\phi_{\omega}(x-(\zeta_{k}+p_{\zeta}))^{2}r(t,x)\\
&{+}iF^{''}\left(\phi_{\omega}(x+(\zeta_{k}+p_{\zeta}))^{2}\right)\phi_{\omega}(x+(\zeta_{k}+p_{\zeta}))^{2}r(t,x)\\
&{+}iF^{''}\left(\phi_{\omega}(x-(\zeta_{k}+p_{\zeta}))^{2}\right)\phi_{\omega}(x-(\zeta_{k}+p_{\zeta}))^{2}\overline{r(t,x)}e^{2i\alpha_{k}(t,x,\sigma_{u})-2i\gamma_{k}(t)}\\
&{+}iF^{''}\left(\phi_{\omega}(x+(\zeta_{k}+p_{\zeta}))^{2}\right)\phi_{\omega}(x+(\zeta_{k}+p_{\zeta}))^{2}\overline{r(t,x)}e^{2i\alpha_{k}(t,{-}x,\sigma_{u})-2i\gamma_{k}(t)}\\
&{+}O\left(v^{2k+2}\left(\ln{\frac{1}{v}}\right)^{c(k)}+v^{2}\left(\ln{\frac{1}{v}}\right)^{c}\norm{r(t)}_{H^{1}}+\norm{r(t)}_{H^{1}}^{2}\right).
\end{aligned}
\end{multline}
\par Consequently, using the definition of $\alpha$ in \ref{alpha} and Remark \ref{induest}, we deduce
\begin{multline*}
2\I \int_{\mathbb{R}}\chi_{1}(t,x)\overline{\partial_{t}r(t,x)}\partial_{x}r(t,x)\,dx\\
\begin{aligned}
=&{-}2\Ree \int_{\mathbb{R}}\chi_{1}(t,x)\overline{\partial^{2}_{x}r(t,x)}\partial_{x}r(t,x)\,dx\\
&{+}2\omega\Ree \int_{\mathbb{R}}\chi_{1}(t,x)\overline{r(t,x)}\partial_{x}r(t,x)\,dx\\
&{-}2\Ree \int_{\mathbb{R}}\chi_{1}(t,x)\left[F^{'}\left(\phi_{\omega}(x-(\zeta_{k}+p_{\zeta}))^{2}\right)+F^{'}\left(\phi_{\omega}(x+(\zeta_{k}+p_{\zeta}))^{2}\right)\right]\overline{r(t,x)}\partial_{x}r(t,x)\,dx\\
&{-}2\Ree \int_{\mathbb{R}}\chi_{1}(t,x)F^{''}\left(\phi_{\omega}(x-(\zeta_{k}+p_{\zeta}))^{2}\right)\phi_{\omega}(x-(\zeta_{k}+p_{\zeta}))^{2}\overline{r(t,x)}\partial_{x}r(t,x)\,dx\\
&{-}2\Ree \int_{\mathbb{R}}\chi_{1}(t,x)F^{''}\left(\phi_{\omega}(x+(\zeta_{k}+p_{\zeta}))^{2}\right)\phi_{\omega}(x+(\zeta_{k}+p_{\zeta}))^{2}\overline{r(t,x)}\partial_{x}r(t,x)\,dx\\
&{-}2 \Ree \int_{\mathbb{R}}\chi_{1}(t,x)F^{''}\left(\phi_{\omega}(x-(\zeta_{k}+p_{\zeta}))^{2}\right)\phi_{\omega}(x-(\zeta_{k}+p_{\zeta})^{2}r(t,x)\partial_{x}r(t,x)\,dx\\
&{-}2 \Ree \int_{\mathbb{R}}\chi_{1}(t,x)F^{''}\left(\phi_{\omega}(x+(\zeta_{k}+p_{\zeta}))^{2}\right)\phi_{\omega}(x+(\zeta_{k}+p_{\zeta})^{2}r(t,x)\partial_{x}r(t,x)\,dx\\
&{+}O\left(\frac{1}{\ln{\frac{1}{v}}}\norm{r(t)}_{H^{1}}^{2}+v^{2k+2}\left(\ln{\frac{1}{v}}\right)^{c(k)}\norm{r(t)}_{H^{1}}+\norm{r(t)}_{H^{1}}^{3}\right),
\end{aligned}
\end{multline*}
from which, using integration by parts and the estimates \eqref{intchi}, we obtain that
\begin{multline*}
    2\I \int_{\mathbb{R}}\chi_{1}(t,x)\overline{\partial_{t}r(t,x)}\partial_{x}r(t,x)\,dx\\
    \begin{aligned}
    =&\Ree \int_{\mathbb{R}}\frac{\partial}{\partial x}\left[F^{'}\left(\phi_{\omega}(x+(\zeta_{k}+p_{\zeta}))^{2}\right)\right]\left\vert r(t,x) \right\vert^{2}\,dx\\
&{+}\Ree \int_{\mathbb{R}}\frac{\partial}{\partial x}\left[F^{''}\left(\phi_{\omega}(x+(\zeta_{k}+p_{\zeta}))^{2}\right)\phi_{\omega}(x+(\zeta_{k}+p_{\zeta}))^{2}\right]\vert r(t,x)\vert^{2}\,dx\\
&{+}\Ree \int_{\mathbb{R}}\frac{\partial}{\partial x}\left[F^{''}\left(\phi_{\omega}(x+(\zeta_{k}+p_{\zeta}))^{2}\right)\phi_{\omega}(x+(\zeta_{k}+p_{\zeta}))^{2}\right]r(t,x)^{2}\,dx\\
&{+}O\left(v^{2k+2}\left(\ln{\frac{1}{v}}\right)^{c(k)}\norm{r(t)}_{H^{1}}+\frac{1}{\ln{\frac{1}{v}}}\norm{r(t)}_{H^{1}}^{2}+\norm{r(t)}_{H^{1}}^{3}\right).
    \end{aligned}
\end{multline*}
\par In conclusion, we have that \eqref{dpj} is true when $v>0$ is small enough.
\end{proof}
Furthermore, using the decay estimates \eqref{decd}, we can deduce from Lemma \ref{dpj} the following proposition.
\begin{corollary}\label{codpj}
 Let $j\in\{1,2\}.$ The function $P_{j}(t,r)$ satisfies
\begin{multline}\label{d(t)dpj}
  \frac{d}{dt}\left[ \dot d(t)P_{j}(t,r)\right] \\  \begin{aligned}=&\dot d(t)\Ree \int_{\mathbb{R}}\frac{\partial}{\partial x}\left[F^{'}\left(\phi_{\omega}(x+({-}1)^{j+1}(\zeta_{k}+p_{\zeta}))^{2}\right)\right]\left\vert r(t,x) \right\vert^{2}\,dx\\
&{+}\dot d(t)\Ree \int_{\mathbb{R}}\frac{\partial}{\partial x}\left[F^{''}\left(\phi_{\omega}(x+({-}1)^{j+1}(\zeta_{k}+p_{\zeta}))^{2}\right)\phi_{\omega}(x+({-}1)^{j+1}(\zeta_{k}+p_{\zeta}))^{2}\right]\vert r(t,x)\vert^{2}\,dx\\
&{+}\dot d(t)\Ree \int_{\mathbb{R}}\frac{\partial}{\partial x}\left[F^{''}\left(\phi_{\omega}(x+({-}1)^{j+1}(\zeta_{k}+p_{\zeta}))^{2}\right)\phi_{\omega}(x+({-}1)^{j+1}(\zeta_{k}+p_{\zeta}))^{2}\right]r(t,x)^{2}\,dx\\
&{+}O\left(v^{2k+3}\left(\ln{\frac{1}{v}}\right)^{c(k)}\norm{r(t)}_{H^{1}}+\frac{v}{\ln{\frac{1}{v}}}\norm{r(t)}_{H^{1}}^{2}+v\norm{r(t)}_{H^{1}}^{3}\right).
\end{aligned}
\end{multline}
\end{corollary}
\par Next, we are going to consider the following Lyapunov Functional
\begin{equation}\label{Efunctional}
E(t,r)=L(t,r)-\dot d(t)P_{2}(t,r)+\dot d(t)P_{1}(t,r),
\end{equation}and demonstrate Theorem \ref{energyestimatetheor} from the derivative of $E(t,r)$ using Lemmas \ref{coerc}, \ref{energyestimate1}, and Corollary \ref{codpj}.
\begin{proof}[Proof of Theorem \ref{energyestimatetheor}.]
\textbf{Step 1.}(Proof of the estimate of $\norm{r(t)}_{H^{1}}.$)
\par First, from Lemma \ref{Modulation}, we have that if $\norm{u(t)-P_{k}(t,x,\sigma_{k}(t))}_{H^{1}}<v^{2},$ $t$ satisfies $\left\vert t -T_{0,k}\right\vert\leq \frac{\left(\ln{\frac{1}{v}}\right)^{\frac{4}{3}}}{v},$ and $v>0$ is small enough, then there exist $K>1$ and $c>0$ satisfying
\begin{equation*}
\norm{r(t)}_{H^{1}}\leq K\left(\ln{\frac{1}{v}}\right)^{c}\norm{u(t)-P_{k}(t,x,\sigma_{k}(t))}_{H^{1}},
\end{equation*}
for some $c>0,$ which is much smaller than $v^{\frac{3}{2}}.$
We also recall from Remark \ref{remestimate} the following estimate
\begin{equation*}
    \dot \zeta_{k}(t)=\dot d(t)+O\left(v^{2}\left(\ln{\frac{1}{v}}\right)^{c}\right),
\end{equation*}
if $v>0$ is small enough. Therefore, we deduce from Lemmas \ref{Modulation}, \ref{energyestimate1} and Corollary \ref{codpj} that
\begin{equation}\label{dE}
\left\vert\dot E(t,r)\right\vert=O\left(v^{2k+2}\left(\ln{\frac{1}{v}}\right)^{c(k)}\norm{r(t)}_{H^{1}}+\frac{v}{\ln{\frac{1}{v}}}\norm{r(t)}_{H^{1}}^{2}+\norm{r(t)}_{H^{1}}^{3}\right).
\end{equation}
\par Next, since $\vert\dot d(t)\vert=O(v),$ we can verify from Cauchy-Schwarz and identity \eqref{localmomentum} the existence of a constant $C>0$ satisfying following estimate for all $v>0$ small enough
\begin{equation*}
    \left\vert \dot d(t)P_{1}(t,r) \right\vert+\left\vert \dot d(t)P_{2}(t,r) \right\vert\leq C v\norm{r(t)}_{H^{1}}^{2}.
\end{equation*}
Therefore, using Lemma \ref{coerc} and the definition of $E(t,r),$ we can find a constant $C>0$ not depending on $v$ and $k$ such that if $v>0$ is small enough, then
\begin{equation}\label{coerE}
C\norm{r(t)}_{H^{1}}^{2}\leq E(t,r).
\end{equation}
Consequently, using the quantity
\begin{equation*}
m(t)=\max\left(\norm{r(t)}_{H^{1}},v^{2k+1}\left(\ln{\frac{1}{v}}\right)^{c(k)+1}\right),
\end{equation*}
we obtain for some positive constant $C_{1}$ that if  $\norm{r(t)}_{H^{1}}>v^{2k+1}\left(\ln{\frac{1}{v}}\right)^{c(k)+1}$ and $\norm{r(t)}_{H^{1}}<\frac{v}{\ln{\frac{1}{v}}},$ then
\begin{equation}\label{gron}
 \dot E(t,r)\leq C_{1}\frac{v}{\ln{\frac{1}{v}}}m(t)^{2},\,\, E(t,r)\geq \frac{C}{2}m(t)^{2}.
\end{equation}
Consequently, using Gronwall Lemma and the assumption \eqref{initialcondition}, we can verify while $\norm{r(t)}_{H^{1}}<v^{2}$ that
\begin{equation}\label{final}
\norm{r(t)}_{H^{1}}\leq m(t)\leq K v^{2k}\left(\ln{\frac{1}{v}}\right)^{c(k)}\exp\left(\frac{2C_{1}}{C}\frac{v}{\ln{\frac{1}{v}}}\left\vert t-T_{0,k} \right\vert\right),
\end{equation}
for any $T_{0,k}$ satisfying the assumptions of Theorem \ref{energyestimatetheor}.
\par Next, from the Mean Value Theorem, the definition of $P_{k}$ in Theorem \ref{app lemma}, we can verify the existence of a $c>0$ such that if $v>0$ is sufficiently small, then
\begin{equation*}
    \norm{ P_{k}(t,x,\sigma_{k}(t)) -P_{k}(t,x,\sigma_{u}(t))}_{H^{1}}\leq K\left(\ln{\frac{1}{v}}\right)^{c}\max_{j\in\{\zeta,v,\gamma,\omega\}}\vert p_{j}(t,r) \vert,
\end{equation*}
for some constant $K>1.$ Consequently, using Lemma \ref{Modulation} and \ref{final}, we obtain estimate \eqref{theenergyestimate} from Minkowski inequality.
\par Moreover, the partial differential equation \eqref{NLS3} is locally well-Posed in the space $H^{s}(\mathbb{R})$ for any $s\geq 0,$ and we are assuming $u(T_{0,k})\in H^{l}(\mathbb{R})$ from the hypotheses of Theorem \ref{energyestimatetheor}. Consequently, we deduce that $\partial^{j}_{x}r(t,x)\in H^{1}(\mathbb{R})$ if $0\leq j\leq l-1$ for all $t\in\mathbb{R}$ satisfying $\vert t-T_{0,k} \vert\leq \frac{\left(\ln{\frac{1}{v}}\right)^{\frac{4}{3}}}{v}.$ Since $u(T_{0,k})$ has sufficient regularity, we might use the same techniques of this step to estimate the higher Sobolev norms of $r(t),$ and also the weighted norms for $r(t)$ in the next steps.\\
 \textbf{Step 2.}(Estimate of $\norm{r(t)}_{H^{l}}$ for $l>1.$)
\par First, using the partial differential equation \eqref{mainequationr} and estimate \eqref{000}, we can verify for any $m\in \mathbb{N}$ satisfying $0\leq m\leq l-1,$ that the function $r_{m}(t,x)\coloneqq \partial^{m}_{x}r(t,x)$ satisfies the following partial differential equation
\begin{multline}\label{mainequationrj}
\begin{aligned}
\partial_{t}r_{m}(t,x)=&i\partial^{2}_{x}r_{m}(t,x)-i\omega r_{m}(t,x)\\
&{+}\left[iF^{'}\left(\phi_{\omega}(x-(\zeta_{k}+p_{\zeta}))^{2}\right)+F\left(\phi_{\omega}(x+(\zeta_{k}+p_{\zeta}))^{2}\right)\right]r_{m}(t,x)\\
&{+}iF^{''}\left(\phi_{\omega}(x-(\zeta_{k}+p_{\zeta}))^{2}\right)\phi_{\omega}(x-(\zeta_{k}+p_{\zeta}))^{2}r_{m}(t,x)\\
&{+}iF^{''}\left(\phi_{\omega}(x+(\zeta_{k}+p_{\zeta}))^{2}\right)\phi_{\omega}(x+(\zeta_{k}+p_{\zeta}))^{2}r_{m}(t,x)\\
&{+}iF^{''}\left(\phi_{\omega}(x-(\zeta_{k}+p_{\zeta}))^{2}\right)\phi_{\omega}(x-(\zeta_{k}+p_{\zeta}))^{2}\overline{r_{m}(t,x)}e^{2i\alpha(t,x,\sigma_{u}(t))-2i\gamma_{k}(t)}\\
&{+}iF^{''}\left(\phi_{\omega}(x+(\zeta_{k}+p_{\zeta}))^{2}\right)\phi_{\omega}(x+(\zeta_{k}+p_{\zeta}))^{2}\overline{r_{m}(t,x)}e^{2i\alpha(t,{-}x,\sigma_{u}(t))-2i\gamma_{k}(t)}\\
&{+}O\left(\norm{r(t)}_{H^{m}}+v^{2}\left(\ln{\frac{1}{v}}\right)^{c}\norm{r(t)}_{H^{m+1}}+\norm{r(t)}_{H^{m+1}}^{2d_{0}-1}+v^{2k+2}\left(\ln{\frac{1}{v}}\right)^{c(k)}\right),
\end{aligned}
\end{multline}
such that $d_{0}\in\mathbb{N}_{\geq 1}$ is the degree of the polynomial $F$ and
the expression inside $O$ means a function $G(t,x)$ such that there exists a constant $C>1$ satisfying
\begin{equation*}
   \norm{G(t,x)}_{H^{1}}\leq C \left[\norm{r(t)}_{H^{m}}+v^{2}\left(\ln{\frac{1}{v}}\right)^{c}\norm{r(t)}_{H^{m+1}}+\norm{r(t)}_{H^{m+1}}^{2d_{0}-1}\right],
\end{equation*}
for all $\vert t-T_{0,k}\vert\leq \frac{\left(\ln{\frac{1}{v}}\right)^{\frac{4}{3}}}{v},$ see the notation section in the introduction.
\par Moreover, using integration by parts and estimate \eqref{final}, we can verify the existence of constants $K,\,C>1$ and parameter a $c(k)>0$ satisfying
\begin{equation*}
    \left\vert \left\langle r_{m}(t),\varphi_{k,\beta}(t,\pm x) \right\rangle \right\vert\leq K v^{2k}\left(\ln{\frac{1}{v}}\right)^{c(k)}\exp\left(\frac{C v}{\ln{\frac{1}{v}}}\vert t-T_{0,k}\vert\right),
\end{equation*}
for any $\beta\in\{\zeta,v,\gamma,\omega\},$
while $\vert t-T_{0,k}\vert\leq \frac{\left(\ln{\frac{1}{v}}\right)^{\frac{4}{3}}}{v}.$  So, Lemma \ref{coerc2} implies the existence of $c>0$ and $K>1$ satisfying
\begin{equation}\label{pooo}
    L(r_{m},t)\geq c\norm{r_{m}(t)}_{H^{1}}^{2}-Kv^{4k}\left(\ln{\frac{1}{v}}\right)^{2c(k)}\exp\left(\frac{2C v}{\ln{\frac{1}{v}}}\vert t-T_{0,k}\vert\right),
\end{equation}
while $\vert t-T_{0,k}\vert\leq \frac{\left(\ln{\frac{1}{v}}\right)^{\frac{4}{3}}}{v}.$ 
\par Furthermore, using the partial differential equation \eqref{mainequationrj}, we can verify similarly to the Step $1$ for constants $K>1,\,c>0$ that the function 
$E(t,r_{m})$ satisfies 
\begin{align}\label{ej}
E(t,r_{m})+Kv^{4k}\left(\ln{\frac{1}{v}}\right)^{2c(k)}\exp\left(\frac{2C v}{\ln{\frac{1}{v}}}\vert t-T_{0,k}\vert\right)\geq c\norm{r_{m}(t)}_{H^{1}}^{2},\\ \label{dej}
\dot E(t,r_{m})\leq  K\left[\frac{v}{\ln{\frac{1}{v}}}\norm{r_{m}(t)}_{H^{1}}^{2}+\norm{r_{m}(t)}_{H^{1}}\norm{r(t)}_{H^{m}}+v^{2k+1}\norm{r_{m}(t)}_{H^{1}}\right],
\end{align}
if $v>0$ is sufficiently small, $\norm{r(t)}_{H^{m}}<1$ and $\vert t-T_{0,k}\vert\leq \frac{\left(\ln{\frac{1}{v}}\right)^{\frac{4}{3}}}{v}.$ 
\par From the estimates above and the hypothesis \eqref{hypoo}, we can verify inductively for any natural number $l_{1}\leq l$ that there exist numbers $C,\,C(l_{1})>1,\,c(l_{1})>0$ satisfying
\begin{equation}\label{jl}
\norm{r(t)}_{H^{l_{1}}}\leq C(l_{1})v^{2k+1-l_{1}}\left(\ln{\frac{1}{v}}\right)^{c(l_{1})}\exp\left(\frac{2C v}{\ln{\frac{1}{v}}}\vert t-T_{0,k}\vert\right),
\end{equation}
for any $\vert t-T_{0,k}\vert\leq \frac{\left(\ln{\frac{1}{v}}\right)^{\frac{4}{3}}}{v},$ if $v>0$ is small enough.
Indeed, estimate \eqref{jl} was verified for $l_{1}=1$ in the first step.
\par Consequently, if \eqref{jl} is true for $1\leq l_{1}\leq m<l,$ then we can verify from the elementary estimate
\begin{equation*}
\norm{r_{m}(t)}_{H^{1}}\norm{r(t)}_{H^{m}}\leq \frac{v}{\ln{\frac{1}{v}}}\max\left(\norm{r_{m}(t)}_{H^{1}},C(m)v^{2k-m}\left(\ln{\frac{1}{v}}\right)^{c(m)+1}\exp\left(\frac{2C\vert t-T_{0,k} \vert}{\ln{\frac{1}{v}}}\right)\right)^{2}
\end{equation*}
and estimates \eqref{ej}, \eqref{dej} using Gronwall Lemma that there exist $C>1,\,c(j)>0$ and $C_{1}(m)>1$ depending on $j$ such that
\begin{equation*}
c\norm{\partial^{m}_{x}r(t)}_{H^{1}}\leq C_{1}(m)v^{2k-m}\left(\ln{\frac{1}{v}}\right)^{c(j)}\exp\left(\frac{2C v}{\ln{\frac{1}{v}}}\vert t-T_{0,k}\vert\right),
\end{equation*}
if $v>0$ is sufficiently small and $\vert t-T_{0,k}\vert\leq \frac{\left(\ln{\frac{1}{v}}\right)^{\frac{4}{3}}}{v}.$ Therefore, since $\norm{r(t)}_{H^{j+1}}$ is equivalent to $\norm{r_{m}(t)}_{H^{1}}+\norm{r(t)}_{H^{m}},$ we conclude that there exists $C_{2}(m)>1$ depending only on $m$ satisfying
\begin{equation*}
\norm{r(t)}_{H^{m+1}}\leq C_{2}(m)v^{2k-m}\left(\ln{\frac{1}{v}}\right)^{c(l)}\exp\left(\frac{2C v}{\ln{\frac{1}{v}}}\vert t-T_{0,k}\vert\right),
\end{equation*}
for any $\vert t-T_{0,k}\vert\leq \frac{\left(\ln{\frac{1}{v}}\right)^{\frac{4}{3}}}{v},$ if $v>0$ is small enough, and so \eqref{jl} is true for any natural number $l_{1}\leq l$ when \eqref{hypoo} is true.
\\
\textbf{Step 3.}(Proof of the Weighted Estimates in $H^{1}.$)
Let $r_{j,m}(t,x)\coloneqq \partial^{m}_{x}\left[x^{j}r(t,x)\right]$  for any $j\in\mathbb{N}.$ 
\par Furthermore, similarly to the proof of estimate \eqref{lambdar}, using Taylor's Expansion Theorem, identity
\begin{equation}\label{xlid}
    x^{j}\Lambda\left(P_{k}(t,x,\sigma_{u}(t))+e^{i\gamma_{k}(t)}r(t,x)\right)=0,
\end{equation} 
and the estimate $\norm{f}_{L^{\infty}}\leq K\norm{f}_{H^{1}}$ in any function $f\in H^{1}(\mathbb{R})$ for some $K>1,$ we can verify from \eqref{000} and the estimate \eqref{final} on $\norm{r(t)}_{H^{1}}$ that $r_{j,0}$ satisfies for some constant $c>0,$ and a parameter $c(k,j)>0$ the following partial differential estimate
\begin{multline}\label{mainequationr2}
\begin{aligned}
\partial_{t}r_{j,0}(t,x)=&i\partial^{2}_{x}r_{j,0}(t,x)-i\omega r_{j,0}(t,x)-j(j-1) i r_{j-2,0}(t,x)+2j i\partial_{x}r_{j-1,0}(t,x)\\
&{+}i\left[F^{'}\left(\phi_{\omega}(x-(\zeta_{k}+p_{\zeta}))^{2}\right)+F^{'}\left(\phi_{\omega}(x+(\zeta_{k}+p_{\zeta}))^{2}\right)\right]r_{j,0}(t,x)\\
&{+}iF^{''}\left(\phi_{\omega}(x-(\zeta_{k}+p_{\zeta}))^{2}\right)\phi_{\omega}(x-(\zeta_{k}+p_{\zeta}))^{2}r_{j,0}(t,x)\\
&{+}iF^{''}\left(\phi_{\omega}(x+(\zeta_{k}+p_{\zeta}))^{2}\right)\phi_{\omega}(x+(\zeta_{k}+p_{\zeta}))^{2}r_{j,0}(t,x)\\
&{+}iF^{''}\left(\phi_{\omega}(x-(\zeta_{k}+p_{\zeta}))^{2}\right)\phi_{\omega}(x-(\zeta_{k}+p_{\zeta}))^{2}\overline{r_{j,0}(t,x)}e^{2i\alpha(t,x,\sigma_{u})-2i\gamma_{k}(t)}\\
&{+}iF^{''}\left(\phi_{\omega}(x+(\zeta_{k}+p_{\zeta}))^{2}\right)\phi_{\omega}(x+(\zeta_{k}+p_{\zeta}))^{2}\overline{r_{j,0}(t,x)}e^{2i\alpha(t,{-}x,\sigma_{u})-2i\gamma_{k}(t)}\\
&{+}O\left(v^{2k}\left(\ln{\frac{1}{v}}\right)^{c(k,j)}+v^{2}\left(\ln{\frac{1}{v}}\right)^{c}\norm{r_{j,0}(t)}_{H^{1}}+\norm{r(t)}_{H^{1}}\norm{r_{j,0}(t)}_{H^{1}}\right),
\end{aligned}
\end{multline}
for all natural number $j\geq 1$ and all $t$ satisfying $\left\vert t -T_{0,k}\right\vert\leq \frac{\left(\ln{\frac{1}{v}}\right)^{\frac{4}{3}}}{v}.$ The expression associated to $O$ means a function $G(t,x)$ satisfying for a $C>1$
\begin{equation*}
    \norm{G(t,x)}_{H^{1}}\leq C\left[v^{2k}\left(\ln{\frac{1}{v}}\right)^{c(k,j)}+v^{2}\left(\ln{\frac{1}{v}}\right)^{c}\norm{r_{j,0}(t)}_{H^{1}}+\norm{r(t)}_{H^{1}}\norm{r_{j,0}(t)}_{H^{1}}\right],
\end{equation*}
if $\left\vert t -T_{0,k}\right\vert\leq \frac{\left(\ln{\frac{1}{v}}\right)^{\frac{4}{3}}}{v}$ and $v>0$ is small enough.
Indeed, estimate \eqref{mainequationr2} is a elementary consequence of the product of $x^{j}$ with estimate \eqref{mainequationr}
\par Moreover, from estimate \eqref{final} and identities \eqref{dphizeta}, \eqref{dphigamma}, \eqref{dphiomega}, \eqref{dphiv}, we can verify from Cauchy-Schwarz inequality that
\begin{equation}\label{orthpopopopo}
   \max_{\beta\in\{\zeta,v,\gamma,\omega\}} \left\vert \left\langle r_{j,0}(t), \varphi_{\beta,k}(t,\pm x)\right\rangle\right\vert\leq Kv^{2k}\left(\ln{\frac{1}{v}}\right)^{c_{0}(k,j)}\exp\left(\frac{Cv}{\ln{\frac{1}{v}}}\left\vert t-T_{0,k}\right\vert\right),
\end{equation}
for some positive constants $C,\,K>1,$ and a number $c_{0}(k,j)>0$ depending only on $j$ and $k.$ 
\par Next, considering the functions
\begin{align}\label{Quadratic energy2}
    L(t,r_{j,0})=&\int_{\mathbb{R}}\left\vert \partial_{x}r_{j,0}(t,x)\right\vert^{2}+\omega \left\vert r_{j,0}(t,x)\right\vert^{2}\,dx\\ \nonumber
    &{-}\int_{\mathbb{R}}F^{'}\left(\phi_{\omega}(x-\zeta_{k}(t))^{2}\right)\left\vert  r_{j,0}(t,x)\right\vert^{2}+F^{'}\left(\phi_{\omega}(x+\zeta_{k}(t))^{2}\right)\left\vert  r_{j,0}(t,x)\right\vert^{2}\,dx\\ \nonumber
    &{-}\Ree\int_{\mathbb{R}} F^{''}\left(\phi_{\omega}(x-\zeta_{k})^{2}\right)\phi_{\omega}(x-\zeta_{k})^{2}e^{2i\alpha_{k}(t,x)}\overline{r_{j,0}(t,x)^{2}}e^{{-}2i\gamma_{k}(t)}\,dx\\ \nonumber
    &{-}\Ree\int_{\mathbb{R}}F^{''}\left(\phi_{\omega}(x+\zeta_{k})^{2}\right)\phi_{\omega}(x+\zeta_{k})^{2}e^{2i\alpha_{k}(t,{-}x)}\overline{r_{j,0}(t,x)^{2}}e^{{-}2i\gamma_{k}(t)}\,dx\\ \nonumber
    &{-}\int_{\mathbb{R}} F^{''}\left(\phi_{\omega}(x-\zeta_{k})^{2}\right)\phi_{\omega}(x-\zeta_{k})^{2}\vert r_{j,0}(t,x)\vert^{2}\,dx\\ \nonumber
    &{-}\int_{\mathbb{R}} F^{''}\left(\phi_{\omega}(x+\zeta_{k})^{2}\right)\phi_{\omega}(x+\zeta_{k})^{2}\vert r_{j,0}(t,x)\vert^{2}\,dx,
\end{align}
and
\begin{align}\label{localmomentum2}
    P_{1}(t,r_{j,0})\coloneqq \I \int_{\mathbb{R}}\chi_{1}(t,x)\overline{r_{j,0}(t,x)}\partial_{x}r_{j,0}(t,x)\,dx,\,\, P_{2}(t,r_{j,0})\coloneqq \I \int_{\mathbb{R}}\chi_{2}(t,x)\overline{r_{j,0}(t,x)}\partial_{x}r_{j,0}(t,x)\,dx,
\end{align}
we can verify similarly to the proof of estimate \eqref{dE} that the Lyapunov function
\begin{equation*}
    E(t,r_{j,0})=L(t,r_{j,0})-\dot d(t)P_{2}(t,r_{j,0})+\dot d(t)P_{1}(t,r_{j,0})
\end{equation*}
satisfies
\begin{align}\label{energt00}
    \left\vert \dot E(t,r_{j,0})  \right \vert =O\Bigg(&v^{2k}\left(\ln{\frac{1}{v}}\right)^{c(k,j)}\norm{r_{j,0}(t)}_{H^{1}}+\frac{v}{\ln{\frac{1}{v}}}\norm{r_{j,0}(t)}_{H^{1}}^{2}+\norm{r_{j,0}(t)}_{H^{1}}\norm{r(t)}_{H^{1}}^{2}\\ \nonumber
    &{+}j(j-1)\norm{r_{j,0}(t)}_{H^{1}}\norm{r_{j-2,0}(t)}_{H^{1}}+j\norm{r_{j,0}(t)}_{H^{1}} \norm{r_{j-1,0}(t)}_{H^{2}}\Bigg),
\end{align}
if $\left\vert t -T_{0,k}\right\vert\leq \frac{\left(\ln{\frac{1}{v}}\right)^{\frac{4}{3}}}{v},\,0\leq j\leq l$ and $v>0$ is small enough.
\par Moreover, using the estimates \eqref{orthpopopopo}, we can verify similarly to the proof of Lemma \ref{coerc} from Lemma \ref{coer} that there exists $K>1$ satisfying
\begin{equation}\label{coerct00}
E(t,r_{j,0})+Kv^{4k}\left(\ln{\frac{1}{v}}\right)^{2c_{0}(k,j)}\exp\left(\frac{2Cv}{\ln{\frac{1}{v}}}\left\vert t-T_{0,k}\right\vert\right)\geq \norm{r_{j,0}(t)}_{H^{1}}^{2}.
\end{equation}

\par Furthermore, motivated by the estimates \eqref{energt00} and \eqref{coerct00}, we are going to verify in the next step by induction on $j$ for any $j\in\mathbb{N}_{\leq l+1},\,m\in\mathbb{N}_{\leq l+1-j}$ satisfying that there exist numbers $C>1,\,C_{j+m}>1,\,c_{j+m}>0$ satisfying
\begin{equation}\label{weigind}
    \norm{r_{j,0}(t)}_{H^{m}}\leq C_{j+m}v^{2k-(2j+m)+1}\left(\ln{\frac{1}{v}}\right)^{c_{j+m}}\exp\left(C\frac{v\vert t-T_{0,k}\vert}{\ln{\frac{1}{v}}}\right),
\end{equation}
for all $\left\vert t -T_{0,k}\right\vert\leq \frac{\left(\ln{\frac{1}{v}}\right)^{\frac{4}{3}}}{v},$ if $v>0$ is sufficiently small.
\par Indeed if for a natural number $j_{0}<l$ the \eqref{weigind} is true for any $0\leq j\leq j_{0},$ then, using estimate \eqref{energt00} for $j=j_{0}+1,$ we obtain that 
\begin{equation*}
     \left\vert \dot E(t,r_{j_{0}+1,0})  \right \vert =O\left(\frac{v}{\ln{\frac{1}{v}}}\norm{r_{j_{0}+1},0(t)}_{H^{1}}^{2}+v^{2k-2j_{0}-1}\left(\ln{\frac{1}{v}}\right)^{c_{j_{0}+2}}\exp\left(C\frac{v\vert t-T_{0,k}\vert}{\ln{\frac{1}{v}}}\right)\norm{r_{j_{0}+1,0}(t)}_{H^{1}}\right),
\end{equation*}
from which we would obtain using Gronwall Lemma similarly to the approach in Step $1$ that there exist $K,\, C_{2}>1$ such that
\begin{equation}\label{j0+1}
   \norm{ x^{j_{0}+1} r(t,x)}_{H^{1}}=\norm{r_{j,0}(t)}_{H^{1}}\leq Kv^{2k-2j_{0}-2}\left(\ln{\frac{1}{v}}\right)^{c_{j_{0}+2}+1}\exp\left(C_{2}\frac{v\vert t-T_{0,k}\vert}{\ln{\frac{1}{v}}}\right),
\end{equation}
while $\left\vert t -T_{0,k}\right\vert\leq \frac{\left(\ln{\frac{1}{v}}\right)^{\frac{4}{3}}}{v}.$  \\
\textbf{Step 4.}(Proof of \eqref{weigind}.) The estimate \eqref{weigind} for $j=0$ was already proved in Step $2,$ so we can assume that \eqref{weigind} is true when $0\leq j\leq j_{0}< l$ for some natural number $j_{0}\geq 0.$ Indeed if it is true for any  $0\leq j\leq l,$ there is nothing more to prove. 
\par Moreover, from the last argument in Step $3,$ we verified that if \eqref{weigind} is true for all natural number $j$ satisfying $0\leq j\leq j_{0},$ then \eqref{j0+1} is true for all $\left\vert t -T_{0,k}\right\vert\leq \frac{\left(\ln{\frac{1}{v}}\right)^{\frac{4}{3}}}{v}.$ Therefore, there is a maximum $m_{0}\geq 1$ such that 
\begin{equation}\label{otimo}
     \norm{ x^{j_{0}+1} r(t,x)}_{H^{m}}=\norm{r_{j_{0}+1,0}(t)}_{H^{m}}\leq Kv^{2k-2j_{0}-2-m+1}\left(\ln{\frac{1}{v}}\right)^{c_{j_{0}+2}+1}\exp\left(C_{2}\frac{v\vert t-T_{0,k}\vert}{\ln{\frac{1}{v}}}\right),
\end{equation}
is true while $\left\vert t -T_{0,k}\right\vert\leq \frac{\left(\ln{\frac{1}{v}}\right)^{\frac{4}{3}}}{v}$ for  all natural $m$ satisfying $1\leq m\leq m_{0}.$ If $m_{0}\geq l-j_{0}+1,$ then \eqref{weightednorm} would be true for any $0\leq j\leq j_{0}+1.$ Consequently, to prove \eqref{weightednorm} for any $j=j_{0}+1,$ it is enough to verify that if \eqref{otimo} is true for any $0<m\leq m_{0} <l-j_{0}+1,$ then it is true for $m_{0}+1.$ 
\par Next, we recall the function $r_{j,m}(t,x)=\partial^{m}_{x}\left[x^{j}r(t,x)\right].$ Since $F$ is a polynomial satisfying \eqref{H1}, using Taylor's Expansion Theorem, \eqref{000}, and the product rule of derivative, we can verify from the derivative of \eqref{mainequationr} on $x$ that
\begin{multline}\label{mainequationrm}
\begin{aligned}
\partial_{t}r_{j_{0}+1,m_{0}}(t,x)=&i\partial^{2}_{x}r_{j_{0}+1,m_{0}}(t,x)-i\omega r_{j_{0}+1,m_{0}}(t,x)\\&{-}j_{0}(j_{0}+1) i r_{j_{0}-1,m_{0}}(t,x)+2(j_{0}+1) i\partial_{x}r_{j_{0},m_{0}}(t,x)
\\&{+}i\left[F^{'}\left(\phi_{\omega}(x-(\zeta_{k}+p_{\zeta}))^{2}\right)+F^{'}\left(\phi_{\omega}(x+(\zeta_{k}+p_{\zeta}))^{2}\right)\right]r_{j_{0}+1,m_{0}}(t,x)\\
&{+}iF^{''}\left(\phi_{\omega}(x-(\zeta_{k}+p_{\zeta}))^{2}\right)\phi_{\omega}(x-(\zeta_{k}+p_{\zeta}))^{2}r_{j_{0}+1,m_{0}}(t,x)\\
&{+}iF^{''}\left(\phi_{\omega}(x+(\zeta_{k}+p_{\zeta}))^{2}\right)\phi_{\omega}(x+(\zeta_{k}+p_{\zeta}))^{2}r_{j_{0}+1,m_{0}}(t,x)\\
&{+}iF^{''}\left(\phi_{\omega}(x-(\zeta_{k}+p_{\zeta}))^{2}\right)\phi_{\omega}(x-(\zeta_{k}+p_{\zeta}))^{2}\overline{r_{j_{0}+1,m_{0}}(t,x)}e^{2i\alpha(t,x,\sigma_{u})-2i\gamma_{k}(t)}\\
&{+}iF^{''}\left(\phi_{\omega}(x+(\zeta_{k}+p_{\zeta}))^{2}\right)\phi_{\omega}(x+(\zeta_{k}+p_{\zeta}))^{2}\overline{r_{j_{0}+1,m_{0}}(t,x)}e^{2i\alpha(t,{-}x,\sigma_{u})-2i\gamma_{k}(t)}\\
&{+}O\left(v^{2k}\left(\ln{\frac{1}{v}}\right)^{c(k,j_{0},m_{0})}+v^{2}\left(\ln{\frac{1}{v}}\right)^{c}\norm{r_{j_{0}+1,0}(t)}_{H^{m_{0}+1}}\right)\\&{+}O\left(\max_{d_{1}+d_{2}=m_{0}+2,\,\min(d_{1},d_{2})\geq 1}\norm{r(t)}_{H^{d_{1}}}\norm{r_{j_{0}+1,0}(t)}_{H^{d_{2}}}+\norm{r_{j_{0}+1,0}(t)}_{H^{m_{0}}}\right),
\end{aligned}
\end{multline}
where $c(k,j_{0},m_{0})>0$ is a number depending on $k,\,j_{0},m_{0},$ while $\left\vert t -T_{0,k}\right\vert\leq \frac{\left(\ln{\frac{1}{v}}\right)^{\frac{4}{3}}}{v}$ and $v>0$ is small enough.
The expression inside $O$ in \eqref{mainequationrm} means a function $G(t,x)$ satisfying for some constant $C>1$
\begin{align*}
\norm{G(t,x)}_{H^{1}}\leq &C\left[v^{2k}\left(\ln{\frac{1}{v}}\right)^{c(k,j_{0},m_{0})}+v^{2}\left(\ln{\frac{1}{v}}\right)^{c}\norm{r_{j_{0}+1,0}(t)}_{H^{m_{0}+1}}\right]\\&{+}C\left[\max_{d_{1}+d_{2}=m_{0}+2,\,\min(d_{1},d_{2})\geq 1}\norm{r(t)}_{H^{d_{1}}}\norm{r_{j,0}(t)}_{H^{d_{2}}}+\norm{r_{j_{0}+1,0}(t)}_{H^{m_{0}}}\right],
\end{align*}
while $\left\vert t -T_{0,k}\right\vert\leq \frac{\left(\ln{\frac{1}{v}}\right)^{\frac{4}{3}}}{v}.$
\par Therefore, if \eqref{otimo} is true for $m=m_{0},$ we deduce that there exists $C>1$ satisfying
\begin{align*}
\norm{G(t,x)}_{H^{1}}\leq &C\left[v^{2k}\left(\ln{\frac{1}{v}}\right)^{c(k,j_{0},m_{0})}+v^{2}\left(\ln{\frac{1}{v}}\right)^{c}\norm{r_{j_{0}+1,0}(t)}_{H^{m_{0}+1}}\right]\\&{+}C\left[\max_{d_{1}+d_{2}=m_{0}+2,\,\min(d_{1},d_{2})\geq 1}\norm{r(t)}_{H^{d_{1}}}\norm{r_{j,0}}_{H^{d_{2}}}\right]\\&{+}C\left[v^{2k-2j_{0}-1-m_{0}}\left(\ln{\frac{1}{v}}\right)^{c_{j_{0}+2}+1}\exp\left(C_{2}\frac{v\vert t-T_{0,k}\vert}{\ln{\frac{1}{v}}}\right)\right],
\end{align*}
\par Consequently, similarly to the proof of estimates of \eqref{energt00}, \eqref{coerct00}, if $0\leq m\leq l-j_{0},$ we can verify using estimates \eqref{jl}, \eqref{weigind},and the partial differential equation \eqref{mainequationrm} that if $k\geq k_{0}$ is large enough, then
\begin{align}\label{energt002m}
    \left\vert \dot E(t,r_{j_{0}+1,m_{0}})  \right \vert =O\Bigg(&v^{2k-2j_{0}-1-m_{0}}\left(\ln{\frac{1}{v}}\right)^{c_{j_{0}+2}+1}\exp\left(C_{2}\frac{v\vert t-T_{0,k}\vert}{\ln{\frac{1}{v}}}\right)\norm{r_{j_{0}+1,m_{0}}(t)}_{H^{1}}\\ \nonumber &{+}\frac{v}{\ln{\frac{1}{v}}}\norm{r_{j_{0}+1,m_{0}}(t)}_{H^{1}}^{2}+j_{0}(j_{0}+1)\norm{r_{j_{0}-1,m_{0}}(t)}_{H^{1}}\norm{r_{j_{0}+1,m_{0}}(t)}_{H^{1}}\\&{+}(j_{0}+1)\norm{r_{j_{0},m_{0}}(t)}_{H^{2}} \norm{r_{j_{0}+1,m_{0}}(t)}_{H^{1}}\Bigg),
\end{align}
and
\begin{equation}\label{coerct00m}
E(t,r_{j_{0}+1,m_{0}})+Kv^{4k}\left(\ln{\frac{1}{v}}\right)^{2c(k,j_{0}+1)}\exp\left(\frac{2Cv}{\ln{\frac{1}{v}}}\left\vert t-T_{0,k}\right\vert\right)\geq \norm{r_{j_{0}+1,m_{0}}(t)}_{H^{1}}^{2},
\end{equation}
for some number $c(k,j_{0}+1)>0,$ and all $\left\vert t -T_{0,k}\right\vert\leq \frac{\left(\ln{\frac{1}{v}}\right)^{\frac{4}{3}}}{v},$ if $v>0$ is small enough.
\par Furthermore, from the estimates \eqref{weigind} for $j\leq j_{0},$ we deduce from \eqref{weigind} that
\begin{equation}\label{e00better}
\left\vert \dot E(t,r_{j_{0}+1,m_{0}})  \right \vert =O\left(\frac{v}{\ln{\frac{1}{v}}}\norm{r_{j_{0}+1,m_{0}}(t)}_{H^{1}}^{2}+\norm{r_{j_{0}+1,m}(t)}_{H^{1}}v^{2k-2j_{0}-m_{0}-1}\left(\ln{\frac{1}{v}}\right)^{c_{j,m}}\exp\left(C\frac{v\vert t-T_{0,k} \vert}{\ln{\frac{1}{v}}}\right)\right),   
\end{equation}
for some constant $c_{j_{0},m}$ while $\left\vert t -T_{0,k}\right\vert\leq \frac{\left(\ln{\frac{1}{v}}\right)^{\frac{4}{3}}}{v}.$
\par In conclusion, using the assumption \eqref{hypoo}, and estimates \eqref{coerct00m}, \eqref{e00better}, we obtain using Gronwall Lemma in the same way we used in Step $1$ that there exist constants $K_{m_{0}}>1,\,C_{2}>1,\,c_{j_{0},m_{0}}>0$ satisfying
\begin{equation*}
    \norm{\partial^{m}_{x}\left[r(t) x^{j_{0}+1}\right]}_{H^{1}}\leq K_{m_{0}}v^{2k-2j_{0}-m_{0}-2}\left(\ln{\frac{1}{v}}\right)^{c_{j_{0},m_{0}}+1}\exp\left(C_{2}\frac{v\vert t-T_{0,k} \vert}{\ln{\frac{1}{v}}}\right),
\end{equation*}
 for any natural number $m\leq m_{0},$ if $\left\vert t -T_{0,k}\right\vert\leq \frac{\left(\ln{\frac{1}{v}}\right)^{\frac{4}{3}}}{v}$ and $v>0$ is small enough. The estimate above implies that \eqref{weigind} is true for $j=j_{0}+1,$ so it is true for all $j\in\mathbb{N}_{\leq l}$ when $v>0$ is sufficiently small, which implies \eqref{rrr} when $j=l$ and $m=1.$
\end{proof}
\section{Orbital Stability of two Solitary Waves}\label{orbsec}
The main objective of this section is to prove the following result.
\begin{theorem}\label{orbitstab}
There exists $C>1$ and $\delta\in(0,1)$ such that if $0<v_{0}<\delta,$ and $\zeta_{0}>\frac{16}{\sqrt{\omega}}\ln{\frac{1}{v}},$ and $r_{j,1}\in H^{1}(\mathbb{R})$ is an odd function satisfying $\norm{r_{j,1}}_{H^{1}}<v_{0}^{5},$ then the solution $\psi(t,x)$ of the following Initial Value Problem 
\begin{equation*}
    \begin{cases}
    i\psi_{t}+\psi_{xx}+F^{'}\left(\vert \psi\vert^{2}\right)\psi=0,\\
    \psi(0,x)=e^{i\gamma_{0}}\left(e^{{-}iv_{0}\frac{x}{2}}\phi_{\omega}(x-\zeta_{0})-e^{iv_{0}\frac{x}{2}}\phi_{\omega}(x+\zeta_{0})\right)+r_{j,1}(x)
    \end{cases}
\end{equation*}
satisfies
\begin{equation}
    \psi(t,x)=e^{i\left(\gamma(t)+\omega t-\frac{v_{0}\zeta(t)}{4}\right)}\left(e^{{-}i\frac{v_{0}}{2}(x-\frac{\zeta(t)}{2})}\phi_{\omega}(x-\zeta(t))-e^{i\frac{v_{0}}{2}(x+\frac{\zeta(t)}{2})}\phi_{\omega}(x+\zeta(t))\right)+r(t,x),
\end{equation}
such that
\begin{align}\label{estdynm}
\left\vert \dot \gamma(t)-\omega+\frac{v_{0}^{2}}{4} \right\vert+\vert \dot \zeta(t)+v_{0} \vert \leq  & C\left[e^{{-}\sqrt{\omega}\zeta_{0}}+\norm{r_{j,1}}_{H^{1}}\right],\\ \nonumber
\norm{r(t)}_{H^{1}}\leq & C\left[ \norm{r_{j,1}}_{H^{1}}+e^{{-}\frac{\sqrt{\omega}}{2}\zeta_{0}}\right],
\end{align}
for any $t\leq 0.$
\end{theorem}
\begin{remark}
    The proof of Theorem \ref{orbitstab} is completely similar to the demonstration of the Theorem $1.3$ from the paper \cite{third}. More precisely, we will use the monotonicity of the energy, mass, and momentum of the solution on the half-line as it was done in \cite{third}. 
\end{remark}
The Theorem \ref{orbitstab} will allow us to analyze the dynamics of the two colliding solitons when the time $t$ approaches ${-}\infty,$ and, consequently, we will be able to conclude the proof of Theorem \ref{main}.    
\begin{lemma}\label{linearpart}
 Let 
 \begin{equation*}
 \phi_{\omega,\pm}(x)=e^{i\omega t}e^{i\frac{\pm vx}{2}}\phi_{\omega}(x).
 \end{equation*}
For any function $r\in H^{1}(\mathbb{R},\mathbb{C}),$ we have that
\begin{equation*}
    \left\langle H^{'}(\phi_{\omega,\pm})+(\omega+\frac{v^{2}}{4}) \frac{Q^{'}(\phi_{\omega,\pm})}{2}\mp \frac{v}{2}M^{'}(\phi_{\omega,\pm}),r\right\rangle=0.
\end{equation*}
\end{lemma}
\begin{proof}[Proof of Lemma \ref{linearpart}.]
\par It is enough to prove Lemma \ref{linearpart} for $\phi_{\omega,+},$ the proof for $\phi_{\omega,-}$ is analogous. First, from the definition of $\phi_{\omega,+},$ we have the following identity
\begin{multline}\label{dx2linear}
  \begin{aligned}
    {-}\int_{\mathbb{R}}\partial^{2}_{x}\phi_{\omega,+}(x)\overline{r(x)}\,dx=&{-}\int_{\mathbb{R}}e^{i(\omega t+\frac{vx}{2})}\phi^{''}_{\omega}(x)\overline{r(x)}-iv\int_{\mathbb{R}}e^{i(\omega t+\frac{vx}{2})}\phi^{'}_{\omega}(x)\overline{r(x)}\,dx\\&{+}\frac{v^{2}}{4}\int_{\mathbb{R}}\phi_{\omega,+}(x)\overline{r(x)}\,dx.
\end{aligned}
\end{multline}
\par Moreover, using the dot product \eqref{dotproduct}, from the definitions of \eqref{H}, \eqref{M} and \eqref{P}, we can verify using integration by parts for any $r\in H^{1}(\mathbb{R},\mathbb{C})$ that
\begin{align*}
    \left\langle H^{'}(\phi_{\omega,+}),r(x) \right\rangle=&\left\langle {-}\partial^{2}_{x}\phi_{\omega,+}(x)-F^{'}(\phi_{\omega}(x)^{2})\phi_{\omega,+},r(x)\right\rangle,\\
    \left\langle Q^{'}(\phi_{\omega,+}),r\right\rangle=&2\left\langle \phi_{\omega}(x)e^{i(\omega t+\frac{vx}{2})},r(x) \right\rangle,\\
    \left\langle M^{'}(\phi_{\omega,+}),r\right\rangle=&\left\langle {-}2i\phi^{'}_{\omega}(x)e^{i(\omega t+\frac{vx}{2})}+v\phi_{\omega}(x)e^{i(\omega t+\frac{vx}{2})},r(x) \right\rangle.
\end{align*}
 Therefore, using
\begin{equation*}
    {-}\phi^{''}_{\omega}(x)+\omega \phi_{\omega}(x)-F^{'}(\phi_{\omega}(x)^{2})\phi_{\omega}(x)=0,
\end{equation*}
and the identities above, we obtain the result of Lemma \ref{linearpart}.
\end{proof}
\par Furthermore, before the start of the proof of Theorem \ref{orbitstab}, we also need to analyze the properties of the following localized quantities:
\begin{align}\label{Half-Mass}\tag{Half-Mass}
Q^{+}(\psi)=&\int_{0}^{{+}\infty}\vert \psi(x) \vert^{2}\,dx,\\ \label{Half-Energy}\tag{Half-Energy}
H^{+}(\psi)=&\int_{0}^{{+}\infty} \frac{\left\vert \psi_{x}(x)\right\vert^{2}}{2}-\frac{F\left(\vert \psi(x) \vert^{2}\right)}{2}\,dx,\\ \label{Half-Momentum}\tag{Half-Momentum}
M^{+}(\psi)=&\I\int_{0}^{{+}\infty}\psi_{x}\overline{\psi}\,dx.
\end{align}

First, using the fact that we are considering odd solutions of the partial differential equation \eqref{NLS3}, we can verify that $H^{+}(\psi(t))=\frac{1}{2}H(\psi)$ and $Q^{+}(\psi(t))=\frac{1}{2}Q(\psi),$ therefore they are constant functions on $t.$ Moreover, the oddness of the function $\psi$ also implies that $M^{+}(\psi(t))$ is a non-decreasing function on $t.$ Consequently, we have the following proposition:
\begin{lemma}\label{almostcons}
The functions $M^{+},\,H^{+},\,Q^{+}$ satisfy for every $t\in\mathbb{R}:$
\begin{align*}
    Q^{+}(\psi(t))=&Q^{+}(\psi(0)),\,H^{+}(\psi(t))= H^{+}(\psi(0)),\\
    \frac{d}{d t}M^{+}(\psi(t))\geq & 0.
\end{align*}
\end{lemma}

\begin{proof}
    The conservation of $Q^{+}(\psi(t))$ and $H^{+}(\psi(t))$ was already explained in the last paragraph. Next, using the time derivative of $M^{+}(\psi(t)),$ identity $\psi(t,0)=0$ and integration by parts, we deduce that if $\psi(t)$ is an odd solution of \eqref{NLS3} belonging to $C(\mathbb{R},H^{2}(\mathbb{R})),$ then 

    \begin{equation*}
        \frac{d}{dt}M^{+}(\psi(t))=\frac{\vert \psi_{x}(t,0) \vert^{2}}{2}\geq 0.
    \end{equation*}
    Therefore, $\frac{d M^{+}(\psi(t))}{dt}\geq 0$ follows from the density of $H^{2}_{x}(\mathbb{R})$ in $H^{1}_{x}(\mathbb{R})$ and the local well-posedness of the partial differential equation \eqref{NLS3} in $H^{s}_{x}(\mathbb{R})$ for $s\geq 1.$
\end{proof}
\par The next important results we need to prove Theorem \ref{orbitstab} are the estimates of the functions $M^{+},\, H^{+},\, Q^{+}$ around a small perturbation of two solitons separated with a large distance and each one of them moving with low speed. \par Next, to simplify our notation, we consider the following bilinear form
\begin{equation*}
    \left\langle f,g\right\rangle_{L^{2}(0,{+}\infty)}\coloneqq\Ree \int_{0}^{{+}\infty} f(x)\overline{g(x)}\,dx.
\end{equation*}

\begin{lemma}\label{asymptlemma} Let 
\begin{align*}
\psi(x)=&e^{i\gamma}\left(e^{{-}i\frac{vx}{2}}\phi_{\omega}(x-\zeta)-e^{i\frac{vx}{2}}\phi_{\omega}(x+\zeta)\right)+r(x)\\
=&Sym(\phi_{\omega,\sigma})(x)+r(x),
\end{align*}such that  $r\in H^{1}_{x}(\mathbb{R})$ is an odd function, and $\omega>0$ satisfies
\begin{equation*}
   Q(\phi_{\omega})= Q^{+}(\psi)=\frac{Q(\psi)}{2}.
\end{equation*}
Then, there exists $\delta\in(0,1)$ such that if $0<v<\delta$ and $\zeta>\delta^{{-}1},$ then
\begin{align}\label{m1}
   \left\langle r,\phi_{\omega,\sigma} \right\rangle_{L^{2}(0,{+}\infty)}=&O(\zeta e^{{-}2\sqrt{\omega}\zeta}+\norm{r}_{L^{2}}^{2}),\\ \label{e1}
    H^{+}(\psi)=&H(\phi_{\omega,\sigma})+\left \langle H^{'}(\phi_{\omega,\sigma}),r\right\rangle_{L^{2}(0,{+}\infty)}+\frac{1}{2}\left\langle D^{(2)}H(\phi_{\omega,\sigma})r,r\right\rangle_{L^{2}(0,{+}\infty)}\\ \nonumber &{+}O\left(\norm{r}_{H^{1}}^{3}+e^{{-}\sqrt{\omega}\zeta}\norm{r}_{H^{1}}+\zeta e^{{-}2\sqrt{\omega}\zeta}\right),\\ \label{p1}
    M^{+}(\psi)=&\frac{v}{2}\norm{\phi_{\omega}}_{L^{2}}^{2}+\left\langle M^{'}(\phi_{\omega,\sigma}),r(x)\right\rangle_{L^{2}(0,{+}\infty)}+M^{+}(r)+O\left( e^{{-}\sqrt{\omega}\zeta}\norm{r}_{H^{1}}+e^{{-}2\sqrt{\omega}\zeta}\zeta\right).
\end{align}
\end{lemma}
\begin{proof}
    \par First, from the estimates
    \begin{equation}\label{gedec}
        \left\vert\frac{d^{l}}{dx^{l}}\phi_{\omega}(x)\right\vert=O(e^{{-}\sqrt{\omega}\vert x\vert}),
    \end{equation}
   we deduce that if $x\geq 0,$ then 
   \begin{equation}\label{de+}
       \left\vert \frac{\partial^{l}}{\partial x^{l}}\phi_{\omega}(x+\zeta) \right\vert=O\left(e^{{-}\sqrt{\omega}\zeta}\right).
   \end{equation}
 \par Next, since $r$ and $\psi$ are odd functions, we can verify from the definition of \eqref{Half-Mass} the following identity
 \begin{equation}\label{taym}
     Q^{+}(\psi)=Q^{+}\left(Sym\left(\phi_{\omega,\sigma}\right)\right)+2\left\langle Sym\left(\phi_{\omega,\sigma}\right),r \right\rangle_{L^{2}(0,{+}\infty)}+\frac{1}{2}\norm{r}_{L^{2}}^{2}.
 \end{equation}
\par Moreover, using Lemma \ref{interactt} and estimate \eqref{gedec}, we have that
\begin{equation*}
    Q^{+}\left(Sym\left(\phi_{\omega,\sigma}\right)\right)=Q(\phi_{\omega})+O\left(\zeta e^{{-}2\sqrt{\omega}\zeta}\right).
\end{equation*}
Consequently, since we are assuming that $Q(\phi_{\omega})=Q(u),$ the estimate \eqref{m1} proceeds from the identity above with \eqref{taym}.
\par Next, from \eqref{de+}, we can also verify from the Fundamental Theorem of Calculus that
\begin{equation}\label{est}
    \int_{0}^{{+}\infty}\left\vert \phi_{\omega}(x+\zeta) \right\vert^{2} +\left\vert \frac{d}{dx}\phi_{\omega}(x+\zeta) \right\vert^{2}\,dx=O(e^{{-}2\sqrt{\omega}\zeta}).
\end{equation}
There, we can verify using Cauchy-Schwarz inequality that
\begin{equation}\label{est2}
    \left\vert \int_{0}^{{+}\infty}\left\vert r\right\vert^{l}\phi_{\omega}(x+\zeta)^{m}\right\vert=O(\norm{r}_{H^{1}}e^{{-}\sqrt{\omega}\zeta}),
\end{equation}
for all $l,$ and $m$ in $\mathbb{N}_{\geq 1}.$ Consequently, using \eqref{est} and \eqref{est2}, we obtain the estimate \eqref{e1} from Taylor's Expansion Theorem applied on $H^{+}(u)$ around $\eta_{\sigma}.$  
\par Finally, the estimates \eqref{p1} and \eqref{e1} follow from the Taylor's Expansion Theorem applied to $M^{+}(u)$ and $H^{+}(u)$ around $\phi_{\omega,\sigma}$ and from we the estimates \eqref{est2}, \eqref{est},  \eqref{de+}.
\end{proof}
\par From now on, we are going to start the proof of the main result of the section.
\begin{proof}[Proof of Theorem \ref{orbitstab}]
We can rewrite $u(0,x)$ as 
\begin{equation*}
    \psi(0,x)=e^{i\gamma_{0}}\left(e^{i\frac{{-}v_{0}x}{2}}\phi_{\omega}\left(x-\zeta_{0}\right)-e^{i\frac{v_{0}x}{2}}\phi_{\omega}\left(x+\zeta_{0}\right)\right)+r(x),
\end{equation*}
such that $Q(\psi(0))=2 Q(\phi_{\omega}),$ and $r$ is an odd function in $H^{1}(\mathbb{R})$ satisfying $\norm{r}_{H^{1}}\lesssim\norm{r_{j,1}}_{H^{1}}.$
\par Furthermore, since the partial differential equation \eqref{NLS3} is locally well-posed in $H^{1}(\mathbb{R}),$ see Chapter $3$ of the book \cite{dsipersivebook}, we have that that there exists $T<0$ such that
\begin{equation*}
    \psi(t,x)=e^{i\gamma_{0}}\left(e^{i\frac{{-}v_{0}x}{2}}\phi_{\omega}\left(x-\zeta_{0}\right)-e^{i\frac{v_{0}x}{2}}\phi_{\omega}\left(x+\zeta_{0}\right)\right)+r(t,x),
\end{equation*}
and $\norm{r(t)}_{H^{1}}<v^{2},$ if $t\in[T,0].$ Therefore, the following set
\begin{equation*}
    B=\left\{t<0\vert\,\inf_{y>\frac{8}{\sqrt{\omega}}\ln{\frac{1}{v_{0}}},\gamma\in\mathbb{R}}\norm{\psi(t)-Sym\left(e^{i(\gamma-\frac{v_{0}(\cdot)}{2})}\phi_{\omega}(\cdot-y)\right)(x)}_{H^{1}(\mathbb{R})}<v_{0}^{2}\right\}
\end{equation*}
is not empty.
\par Furthermore, using the Implicit Function Theorem for Banach spaces, see Section $3.1$ from \cite{Martel}, if an interval $(t,0)\subset B,$ we can verify the existence of real functions of class $C^{1}$ $\zeta(t)>\frac{8}{\sqrt{\omega}}\ln{\frac{1}{v_{0}}},\,\gamma(t)$ such that
\begin{equation*}
    \psi(s,x)=e^{i\gamma(s)}\left(e^{i\frac{{-}v_{0}x}{2}}\phi_{\omega}(x-\zeta(s))-e^{i\frac{v_{0}x}{2}}\phi_{\omega}(x+\zeta(s))\right)+r_{1}(s,x),
\end{equation*}
and $r_{1}$ is a odd function in $H^{1}_{x}(\mathbb{R})$ satisfying
\begin{align}\label{mod1iden}
    \left\langle e^{i(\gamma(s)\mp\frac{v_{0}x}{2})}(x\mp \zeta(s))\phi_{\omega}(x\mp \zeta(s)),r_{1}(s)\right\rangle=&0,\\ \label{mod2iden}
    \left\langle ie^{i(\gamma(s)\mp\frac{v_{0}x}{2})}\partial_{\omega}\phi_{\omega}(x\mp \zeta(s)),r_{1}(s)\right\rangle=&0,
\end{align}
and 
\begin{equation*}
    \norm{r_{1}(s)}_{H^{1}(\mathbb{R})}\leq C\inf_{y\geq \frac{7}{\sqrt{\omega}}\ln{\frac{1}{v}}}\norm{\psi(t)-Sym\left(e^{i(\gamma(s)-\frac{v_{0}(\cdot)}{2})}\phi_{\omega}(\cdot-y)\right)}_{H^{1}(\mathbb{R})},
\end{equation*}
for all $s\in [t,0),$ where $C>1$ is a constant.
\par From now on, we are going to prove that $B=\mathbb{R}_{<0}.$ Moreover, we consider for any $t>0$ the following function
\begin{equation*}
\mathcal{M}(\psi)(t)\coloneq H^{+}(\psi)({-}t)+\left(\omega+\frac{v_{0}^{2}}{4}\right)\frac{Q^{+}\psi({-}t)}{2}-\frac{v_{0}}{2}M^{+}(\psi)({-}t),
\end{equation*}
\par First, Lemma \ref{almostcons} implies that
\begin{equation*}
    \mathcal{M}(\psi)(t)\leq \mathcal{M}(\psi)(0).
\end{equation*}
\par Furthermore, using Lemma \ref{linearpart}, estimates \eqref{est}, \eqref{est2} and Taylor's Expansion Theorem, we can verify that
\begin{multline}
    \begin{aligned}\mathcal{M}(\psi)(t)=\mathcal{M}(\phi_{\omega}(x))+\Big\langle &{-}\partial^{2}_{x}r_{1}(t,x)+\omega r_{1}-F^{'}\left(\phi_{\omega}(x-\zeta(t))^{2}\right)r_{1}(t,x)\\&{-}F^{''}\left(\phi_{\omega}(x-\zeta(t))^{2}\right)\phi_{\omega}(x-\zeta(t))e^{i(\gamma(t)-\frac{v_{0}x}{2})}\overline{r_{1}(t,x)}\\
    &{-}F^{''}\left(\phi_{\omega}(x-\zeta(t))^{2}\right)\phi_{\omega}(x-\zeta(t))e^{{-}i(\gamma(t)-\frac{v_{0}x}{2})}r_{1}(t),r_{1}(t,x)\Big\rangle_{L^{2}(0,{+}\infty)}
    \\&{+}O\left(\norm{r_{1}(t)}_{H^{1}}^{2}(e^{{-}\sqrt{\omega}\zeta(t)}+v_{0})+\norm{r_{1}(t)}_{H^{1}}^{3}+\zeta(t)e^{{-}2\sqrt{\omega}\zeta(t)}\right),
\end{aligned}
\end{multline}
while $t$ is in $B.$
Consequently, since $r_{1}$ satisfies the orthogonality conditions \eqref{mod1iden} and \eqref{mod2iden}, and Lemma \ref{asymptlemma} implies that
\begin{equation*}
    \left\langle r_{1}(t),e^{i(\gamma(t)-\frac{v_{0}x}{2})}\phi_{\omega}(x-\zeta(t))\right\rangle=O\left(\max\{\norm{r_{1}(t)}_{H^{1}}^{2},\zeta(t)e^{{-}2\sqrt{\omega}\zeta(t)}\}\right),
\end{equation*}
we can deduce from Remark \ref{coerremark} of Lemma \ref{coer} the existence of constants $C,\,c>0$ independent of $v_{0}$ satisfying
\begin{equation*}
    \mathcal{M}(\psi)(t)+C\zeta(t)e^{{-}2\sqrt{\omega}\zeta(t)}\geq  \mathcal{M}(\phi_{\omega}(x))+c\norm{r_{1}(t)}_{H^{1}}^{2},
\end{equation*}
while $t$ is in $B.$ Therefore, since $\mathcal{M}(\psi)(t)\leq \mathcal{M}(\psi)(0)$ and $\psi(0)$ satisfies the hypotheses of Theorem \ref{orbitstab}, we obtain the following inequality
\begin{equation}\label{PO}
C\norm{r_{j,1}}_{H^{1}}^{2}+Ce^{{-}\sqrt{\omega}\zeta_{0}}+Ce^{{-}\sqrt{\omega}\zeta(t)}\geq  \norm{r_{1}(t)}_{H^{1}}^{2},
\end{equation}
for some constant $C>0,$ while $t<0$ is an element of $B.$
\par Moreover, using identity
 $\Lambda(\psi)=0$ and Lemma \ref{interactt}, we can verify that
\begin{align*}
    i\partial_{t}r_{1}+\partial^{2}_{x}r_{1}=&\left(\dot \gamma(t)-\omega -\frac{v_{0}\dot \zeta(t)}{4}\right)Sym\left(e^{i\left(\gamma(t)-\frac{v_{0}(\cdot)}{2}\right)}\phi_{\omega}(\cdot-\zeta(t))\right)(x)\\
    &{+}i(\dot \zeta(t)+v_{0})Sym\left(e^{i\left(\gamma(t)-\frac{v_{0}(\cdot)}{2}\right)}\phi^{'}_{\omega}(\cdot-\zeta(t))\right)(x)\\
&{-}F^{'}\left(\phi_{\omega}(x-\zeta(t))^{2}+\phi_{\omega}(x+\zeta(t))^{2}\right)r_{1}(t,x)
\\&{-}F^{''}\left(\phi_{\omega}(x-\zeta(t))^{2}\right)\phi_{\omega}(x-\zeta(t))e^{i(\gamma(t)-\frac{v_{0}x}{2})}\overline{r_{1}(t,x)}\\
&{-}F^{''}\left(\phi_{\omega}(x-\zeta(t))^{2}\right)\phi_{\omega}(x-\zeta(t))e^{{-}i(\gamma(t)-\frac{v_{0}x}{2})}r_{1}(t)\\&{+}O\left(e^{{-}2\sqrt{\omega}y(t)}+\norm{r_{1}(t)}\zeta(t)e^{{-}2\sqrt{\omega}\zeta(t)}+\norm{r_{1}(t)}_{H^{1}}^{2}\right).
\end{align*}
Consequently, from the time derivative of the orthogonality conditions \eqref{mod1iden} and \eqref{mod2iden}, we obtain the existence of a constant $C>0$ satisfying the following estimates
\begin{align}\label{dynmod1}
    \left\vert \dot \gamma(t)-\omega +\frac{v_{0}^{2}}{4} \right\vert<&C\left[\norm{r_{1}(t)}_{H^{1}}+e^{{-}2\sqrt{\omega}\zeta(t)}\right],\\ \label{dynmod2}
    \vert \dot \zeta(t)+v_{0} \vert\leq &C\left[\norm{r_{1}(t)}_{H^{1}}+e^{{-}2\sqrt{\omega}\zeta(t)}\right],
\end{align}
while $t$ is in $B.$
\par In particular, from the definition of the set $B,$ we have that if $v_{0}>0$ is sufficiently small, the estimate \eqref{dynmod2} implies that $\dot \zeta(t)>\frac{3}{4}v_{0}>0,$ while $t$ is in $B.$ Therefore $\zeta(t)>\zeta(0)$ for any $t<0$ in $B,$ from which, using estimate \eqref{PO}, we deduce that
\begin{align}\label{esttt1}
  \norm{r_{1}(t)}_{H^{1}}^{2}\leq &C\norm{r_{j,1}}_{H^{1}}^{2}+Ce^{{-}\sqrt{\omega}\zeta_{0}}<v_{0}^{2},\\ \nonumber
  \zeta(t)>&\zeta(0)\geq \frac{8}{\sqrt{\omega}}\ln{\frac{1}{v_{0}}}.
\end{align}
\par Consequently, since $\norm{r_{1}(t)}_{H^{1}},\, \zeta(t)$ are continuous functions on $t,$ we can verify similarly as in the Step $6$ of the proof of Theorem $1.3$ from \cite{third} that $B$ is equal to $\mathbb{R}_{<0}$ and that estimate \eqref{esttt1} is true for all $t\leq 0.$ The estimates \eqref{estdynm} follows from \eqref{dynmod1} and \eqref{dynamicsmod} using \eqref{esttt1}.
\end{proof}
\section{Proof of Theorem \ref{main}}
\par From now on, let $u$ be the unique solution of \eqref{NLS3} satisfying \eqref{initialcondition} for $v_{1}=0.$ From Theorems \ref{uniq}, \ref{app lemma} and Remark \ref{uniq}, we can verify for any $l\in\mathbb{N}$ that there exist $k\in\mathbb{N}$ sufficiently large and $v>0$ small enough that there exists $T_{0,k}\in\mathbb{R}$ satisfying $ \frac{\left(\ln{\frac{1}{v}}\right)^{\frac{4}{3}}}{20 v}<T_{0,k}<\frac{\left(\ln{\frac{1}{v}}\right)^{\frac{4}{3}}}{10 v}$ and there exists  $t_{0}$ satisfying $\vert t_{0}\vert \leq C_{0}\frac{\ln{\frac{1}{v}}}{v}$ for some constant $C_{0}>1$ independent of $k$ such that
\begin{align*}
    \norm{(1+\vert x\vert^{2})^{l}\left[u(T_{0,k}+t_{0,k},x)-P_{k}(T_{0,k},x,\sigma_{k}(T_{0,k}))\right]}_{H^{2l}}<&v^{10k}
\end{align*}
Consequently, Theorem \ref{energyestimatetheor} implies that
\begin{align}\label{final10}
    \norm{u(t+t_{0},x)-P_{k}(t,x,\sigma_{k}(t))}_{H^{1}}<&v^{k},\\
\label{final11}
    \norm{x^{l}\left[u(t+t_{0},x)-P_{k}(t,x,\sigma_{k}(t))\right]}_{H^{1}}<&v^{k},
\end{align}
for any $t$ satisfying
${-}2\frac{\left(\ln{\frac{1}{v}}\right)^{\frac{4}{3}}}{v}\leq t \leq \frac{\left(\ln{\frac{1}{v}}\right)^{\frac{4}{3}}}{v}.
$
The estimate \eqref{final11} and Theorem \ref{app lemma} imply the inequality \eqref{weightednorm} for $v_{1}=0$ in Theorem \ref{main}. Consequently, it remains to prove the estimates \eqref{almoelres} in Theorem \ref{main}, this will follow from estimate \eqref{final10} and Theorem \ref{orbitstab} from Section \ref{orbsec}.
\par Moreover, for $v_{1}=0,$ $u(t)\in H^{1}(\mathbb{R})$ is an odd function because of the uniqueness result of Theorem \ref{uniq}, and $P_{k}(t,x,\sigma_{k}(t))$ is also odd functions on $x.$ Therefore, $u(t+t_{0},x)-P_{k}(t,x,\sigma_{k}(t))$ is an odd function on $x$ for any $t\in\mathbb{R}.$
\par Consequently, using the information of the asymptotics of $P_{k}$ in Theorem \ref{app lemma}, we can apply Theorem \ref{orbitstab} for the initial data $\psi(0,x)=u({-}t_{0}-\frac{10k\ln{\frac{1}{v}}}{v\sqrt{\omega}},x).$ From the local Well-Posedness of \eqref{NLS3} in $H^{1},$ we have that $\psi(t,x)=u(t-t_{0}-\frac{10k\ln{\frac{1}{v}}}{v\sqrt{\omega}},x).$ Therefore, Theorem \ref{orbitstab} implies that $u(t,x)$ satisfies
\eqref{almoelres} and \eqref{modass} for any $t< \frac{{-}\left(\ln{\frac{1}{v}}\right)^{\frac{4}{3}}}{v}$ when $v>0$ is sufficiently small. In conclusion, Theorem \ref{main} is true when $v_{1}=0$ and Remark \ref{v1=0} implies that Theorem \ref{main} is true for any $v_{1}\in\mathbb{R}.$

\appendix
\section{Proof of Theorem \ref{uniq}}
 Before we start the proof, we consider for any $v,\, T>0$ the following norms
\begin{equation*}
\norm{f}_{L^{2}_{T,v}}=\sup_{t\geq T}e^{vt}\norm{f(t,x)}_{L^{2}_{x}(\mathbb{R})},\, \norm{f}_{H^{1}_{T,v}}=\sup_{t\geq T}e^{vt}\norm{f(t,x)}_{H^{1}_{x}(\mathbb{R})}.
\end{equation*}
It is not difficult to verify that the set of complex functions $f(t,x)$ with domain $[T,{+}\infty)\times\mathbb{R}$ having bounded norm $\norm{\cdot}_{L^{2}_{T,v}}$ is a Banach space $L^{2}_{T,v}$ and the same happens for the norm $\norm{\cdot}_{H^{1}_{T,v}}$ which forms a Banach Space $H^{1}_{T,v}.$ 
\par Furthermore, Theorem \ref{uniq} is a consequence of the following proposition.
\begin{lemma}\label{edl}
 There exist $C>1$ and $\delta>0$ such that if $0<v<\delta,\,T\geq \frac{4\ln{\frac{1}{v}}}{\sqrt{\omega}v}$ and $f(t,x)$ is any function odd on $x$ satisfying
 \begin{equation*}
     \norm{f}_{H^{1}_{T,\frac{3\sqrt{\omega}v}{4}}}<{+}\infty,
 \end{equation*}
then there exists a unique function $r(t,x)\in H^{1}_{T,\delta}$ odd on $x$ 
satisfying
\begin{align}\label{linearnls}
       i\partial_{t}r+\partial^{2}_{x}r-\omega r(t,x)
{+}&F^{'}\left(\phi_{\omega}(x-vt)^{2}+\phi_{\omega}(x+vt)^{2}\right)r(t,x)
\\ \nonumber {+}&F^{''}\left(\phi_{\omega}(x-vt)^{2}\right)\phi_{\omega}(x-vt)e^{i({-}\frac{v^{2}t}{2}+\frac{vx}{2})}\overline{r(t,x)}\\ \nonumber
{-}&F^{''}\left(\phi_{\omega}(x+vt)^{2}\right)\phi_{\omega}(x+vt)e^{i(\frac{{-}v^{2}t}{2}-\frac{vx}{2})}\overline{r(t)}
\\ \nonumber
{+}&F^{''}\left(\phi_{\omega}(x-vt)^{2}\right)\phi_{\omega}(x-vt)e^{{-}i({-}\frac{v^{2}t}{2}+\frac{vx}{2})}r(t)\\ \nonumber
{-}&F^{''}\left(\phi_{\omega}(x+vt)^{2}\right)\phi_{\omega}(x+vt)e^{{-}i({-}\frac{v^{2}t}{2}-\frac{vx}{2})}r(t)
= e^{{-}i\omega t}f(t,x).
\end{align}
Furthermore,

\begin{equation}\label{e000}
    \norm{r(t)}_{H^{1}_{T,\frac{3\sqrt{\omega}v}{4}}}<\frac{C}{v^{2}}\norm{f(t)}_{H^{1}_{T,\frac{3\sqrt{\omega}v}{4}}},
\end{equation}
and for any $m\in\mathbb{N},$ there is $C_{m}$ such that 
\begin{equation}
    \norm{x^{m}r(t)}_{H^{l}_{T,\frac{3\sqrt{\omega}v}{4}}}<\frac{C_{m}}{v^{2(2m+l)}}\max_{d_{2}\leq m, d_{1}+d_{2}=m+l+1   }\norm{f(t)(1+ x^{2})^{\frac{d_{2}}{2}}}_{H^{d_{1}}_{T,\frac{3\sqrt{\omega}v}{4}}}
\end{equation}
\end{lemma}
\begin{lemma}\label{ortll}
Let $r$ be a solution in $C([T,+\infty],H^{1})$ of \eqref{linearnls}, and let
\begin{align*}
    Ort_{1,\pm}(t)=&\left\langle r(t),i\phi^{'}_{\omega}\left(\pm x-vt\right)e^{i(\frac{v(\pm x)}{2}-\frac{v^{2}t}{4})} \right\rangle,\,\\ Ort_{2,\pm}(t)=&\left\langle r(t),\phi_{\omega}\left(\pm x-vt\right)e^{i(\frac{v(\pm x)}{2}-\frac{v^{2}t}{4})} \right\rangle,\\
     Ort_{3,\pm}(t)=&\left\langle r(t),(\pm x-vt)\phi_{\omega}\left(\pm x-vt\right)e^{i(\frac{v(\pm x)}{2}-\frac{v^{2}t}{4})} \right\rangle,\\ Ort_{4,\pm}(t)=&\left\langle r(t),i\partial_{\omega}\phi_{\omega}\left(\pm x-vt\right)e^{i(\frac{v(\pm x)}{2}-\frac{v^{2}t}{4})} \right\rangle.
\end{align*}    
There exist $c>1,\,\delta>0$ such that if $0<v<\delta,\,T\geq \frac{4\ln{\left(\frac{1}{v}\right)}}{\sqrt{\omega}v},$
then the following estimates are true for any $t\in[T,\,{+}\infty]$
\begin{align*}
    \frac{dOrt_{1,\pm}(t)}{dt}=&O\left((1+vt)^{c}e^{{-}2\sqrt{\omega}v\vert t\vert}\norm{r(t)}_{H^{1}}+\norm{f(t)}_{L^{2}}\right),\\
    \frac{dOrt_{2,\pm}(t)}{dt}=&O\left((1+vt)^{c}e^{{-}2\sqrt{\omega}v\vert t\vert}\norm{r(t)}_{H^{1}}+\norm{f(t)}_{L^{2}}\right),\\
    \frac{dOrt_{3,\pm}(t)}{dt}=&\mp 2Ort_{1,\pm}(t)+O\left((1+vt)^{c}e^{{-}2\sqrt{\omega}v\vert t\vert}\norm{r(t)}_{H^{1}}+\norm{f(t)}_{L^{2}}\right),\\
    \frac{dOrt_{4,\pm}(t)}{dt}=&Ort_{2,\pm}(t)+O\left((1+vt)^{c}e^{{-}2\sqrt{\omega}v\vert t\vert}\norm{r(t)}_{H^{1}}+\norm{f(t)}_{L^{2}}\right).
\end{align*}
\end{lemma}
\begin{proof}[Proof of Lemma \ref{ortll}.]
Let $\varphi_{\omega,\zeta},\,\varphi_{\omega,\gamma},\,\varphi_{\omega,v}$ and $\varphi_{\omega,\omega}$ be the following functions
\begin{align*}
    \varphi_{\omega,\zeta}(t,x)=&\phi^{'}_{\omega}\left( x-vt\right)e^{i(\omega t+\frac{v( x)}{2}-\frac{v^{2}t}{4})},\,\varphi_{\omega,\gamma}(t,x)=i\phi_{\omega}\left( x-vt\right)e^{i(\omega t+\frac{v x}{2}-\frac{v^{2}t}{4})},\\
    \varphi_{\omega,v}(t,x)=&i(x-vt)\phi_{\omega}\left( x-vt\right)e^{i(\omega t+\frac{v x}{2}-\frac{v^{2}t}{4})},\,\varphi_{\omega,\omega}(t,x)=\partial_{\omega}\phi_{\omega}(x-vt)e^{i(\omega t+\frac{v x}{2}-\frac{v^{2}t}{4})}.
\end{align*}
From the formula \ref{linearnls} and integration by parts, we deduce using Lemma \ref{interactt} that
\begin{align*}
    \frac{d}{dt}\left\langle r(t),i\varphi_{\omega,j}(t,x) \right\rangle=&{-}\left\langle r(t),S_{\omega}\left(\frac{\partial}{\partial j}\big\vert_{t=0}\phi_{\omega}(\cdot-vt)e^{i(\omega t+\frac{v x}{2}-\frac{v^{2}t}{4})}\right)(x-vt)e^{i(\frac{v x}{2}-\frac{v^{2}t}{4})} \right\rangle\\&{+}O\left((1+vt)^{c}e^{{-}2\sqrt{\omega}v\vert t\vert}\norm{r(t)}_{H^{1}}+\norm{f(t)}_{L^{2}}\right),
\end{align*}
for some constant $c>1$ and any $j\in\{\zeta,v,\gamma,\omega\}.$ 
\par Consequently, using identities $S_{\omega}(\phi^{'}_{\omega})=0,\,S_{\omega}(i\phi_{\omega})=0,$ \eqref{somega1} and \eqref{somega2}, we conclude all the estimates above for each $j\in\{\zeta,v,\gamma,\omega\}.$

\end{proof}
\begin{proof}[Proof of Lemma \ref{edl}.]
We will divide the proof of Theorem \ref{uniq} into two parts, first the existence part and later the uniqueness part.\\
\textbf{Step 1.}(Ansatz for the solution $u.$) First, it is not difficult to verify that the right-hand side of the equation \eqref{linearnls} comes from the linearized equation \eqref{NLS3} for a solution $u(t)$ of the following form
\begin{equation}\label{pos}
    u(t,x)=\phi_{\omega}(x-vt)e^{i(\omega t+\frac{vx}{2}-\frac{v^{2}t}{4})}-\phi_{\omega}(x+vt)e^{i(\omega t+\frac{{-}vx}{2}-\frac{v^{2}t}{4})}+e^{i\omega t}r(t,x),
\end{equation}
such that $Q(u)=2Q(\phi_{\omega}).$ We are going to apply a fixed-point method to prove the existence of a unique $r$ having $\norm{r}_{H^{1}_{T,\delta}}<1$ for some large $T>0$ and small $\delta>0$ such that $u(t,x)$ in \eqref{pos} is the unique solution of \eqref{NLS3}.  
\par Furthermore, since $Q(u)=2Q(\phi_{\omega}),$ we obtain from Lemma \ref{asymptlemma} that if $t>0$ is sufficiently large and $\norm{r(t)}_{H^{1}}$ is small enough, then
\begin{equation}\label{oooo}
    \left\langle e^{i\omega t}r(t,x),\phi_{\omega}(\pm x-vt)e^{i(\omega t+\frac{v(\pm x)}{2}-\frac{v^{2}t}{4})}  \right\rangle=O\left(\norm{r(t)}_{H^{1}}^{2}+e^{{-}2\sqrt{\omega}vt}\right).
\end{equation}
\textbf{Step 2.}(Lyapunov function for energy Estimate.)\\
The energy estimate method used here is completely similar to the one used in the proof of Theorem \ref{energyestimatetheor}.
\par First, we consider the Hessian of the energy $H$ of the two solitary waves 
\begin{multline*}
 \begin{aligned}
 L_{v}(r,t)=&\int_{\mathbb{R}}\left\vert \partial_{x}r(t,x)\right\vert^{2}+\omega \left\vert r(t,x)\right\vert^{2}\,dx\\ \nonumber
    &{-}\int_{\mathbb{R}}F^{'}\left(\phi_{\omega}(x-vt)^{2}\right)\left\vert  r(t,x)\right\vert^{2}+F^{'}\left(\phi_{\omega}(x+vt)^{2}\right)\left\vert  r(t,x)\right\vert^{2}\,dx\\ \nonumber
    &{-}\Ree\int_{\mathbb{R}} F^{''}\left(\phi_{\omega}(x-vt)^{2}\right)\phi_{\omega}(x-vt)^{2}e^{i(vx-\frac{v^{2}t}{2})}\overline{r(t,x)^{2}}\,dx\\ \nonumber
    &{-}\Ree\int_{\mathbb{R}}F^{''}\left(\phi_{\omega}(x+vt)^{2}\right)\phi_{\omega}(x+vt)^{2}e^{i(vx-\frac{v^{2}t}{2})}\overline{r(t,x)^{2}}\,dx\\ \nonumber
    &{-}\int_{\mathbb{R}} F^{''}\left(\phi_{\omega}(x-vt)^{2}\right)\phi_{\omega}(x-vt)^{2}\vert r(t,x)\vert^{2}\,dx\\ \nonumber
    &{-}\int_{\mathbb{R}} F^{''}\left(\phi_{\omega}(x+vt)^{2}\right)\phi_{\omega}(x+vt)^{2}\vert r(t,x)\vert^{2}\,dx.
\end{aligned}
\end{multline*}
Moreover, using a smooth cut-off function $\chi$ satisfying \eqref{chi}, we consider the localized momentum corrections
\begin{equation*}
M_{v}(r,t)=v \I \int_{\mathbb{R}}\chi\left(\frac{x+vt}{2vt}\right)\overline{r(t,x)}\partial_{x}r(t,x)\,dx- v\I \int_{\mathbb{R}}\left[1-\chi\left(\frac{x+vt}{2vt}\right)\right]\overline{r(t,x)}\partial_{x}r(t,x)\,dx.
\end{equation*}
\par From now on, we consider the Lyapunov function $E_{v}(r,t)=L_{v}(r,t)+M_{v}(r,t)$ to estimate $\norm{r}_{H^{1}_{T,\delta}}$ for parameters $\delta\in (0,1)$ and $T>1$ to be chosen later.
\par First, using Lemma \ref{coer} and estimate \eqref{oooo}, we can verify the existence of positive constants $c,\,C$ such that if $v>0$ is small enough, $t>4\frac{\ln{\left(\frac{1}{v}\right)}}{v},$ then
\begin{equation}\label{coerLv}
    E_{v}(r,t)>c\norm{r(t)}_{H^{1}}^{2}-C\left[\sum_{j=1}^{4}Ort_{j,+}(t)^{2}+Ort_{j,-}(t)^{2}\right].
\end{equation}
\textbf{Step 3.}(Derivative of $E_{v}(r,t).$) 
%From the formula of \eqref{pos} satisfied by $u(t,x),$ we can verify using Taylor's Expansion Theorem and Lemma \ref{interactt} that
%\begin{align}\label{eqrr}
 %   i\partial_{t}r+\partial^{2}_{x}r-\omega r(t,x)=&
%{-}F^{'}\left(\phi_{\omega}(x-vt)^{2}+\phi_{\omega}(x+vt)^{2}\right)r(t,x)
%\\ \nonumber %&{-}F^{''}\left(\phi_{\omega}(x-vt)^{2}\right)\phi_{\omega}(x-vt)e^{i({-}\frac{v^{2}t}{2}+\frac{vx}{2})}\overline{r(t,x)}\\ \nonumber
%&{+}F^{''}\left(\phi_{\omega}(x+vt)^{2}\right)\phi_{\omega}(x+vt)e^{i(\frac{{-}v^{2}t}{2}-\frac{vx}{2})}\overline{r(t)}\,dx
%\\ \nonumber
%&{-}F^{''}\left(\phi_{\omega}(x-vt)^{2}\right)\phi_{\omega}(x-vt)e^{{-}i({-}\frac{v^{2}t}{2}+\frac{vx}{2})}r(t)\\ \nonumber
%&{+}F^{''}\left(\phi_{\omega}(x+vt)^{2}\right)\phi_{\omega}(x+vt)e^{{-}i({-}\frac{v^{2}t}{2}-\frac{vx}{2})}r(t)
%\\ \nonumber &{+}O\left(e^{{-}2\sqrt{\omega}vt}+\norm{r(t)}_{H^{1}}(1+vt)e^{{-}2\sqrt{\omega}vt}+\norm{r(t)}_{H^{1}}^{2}\right).
%\end{align}
Similarly to the proofs of Lemmas \ref{energyestimate1} and \ref{dpj} but using $\zeta(t)=vt,\, \gamma(t)=0$ in the place of $\zeta_{k}(t)$ and $\gamma_{k}(t),$ we can verify the existence of a constant $C>1$ satisfying
\begin{equation}\label{energyestimate2solitons}
    \left\vert \dot E_{v}(r,t)\right\vert<C\left(\frac{1}{t}\norm{r(t)}_{H^{1}}^{2}+\norm{r(t)}_{H^{1}}\norm{f(t)}_{H^{1}}\right),
\end{equation}
if $\norm{r(t)}_{H^{1}}<1,$ and $t$ is large enough, in this case the condition $t\geq \frac{4\ln{\frac{1}{v}}}{\sqrt{\omega}v}$ is sufficient when $v>0$ is small enough. \\
\textbf{Step 4.}(Estimate of the spectral projection of $r(t)$ in non-positive part of $S_{\omega}.$) \\
In this step, we focus on the estimate of the expressions $Ort_{1,\pm}(t),\,Ort_{2,\pm}(t),\,Ort_{3,\pm}(t)$ and $Ort_{4,\pm}(t)$ defined in Lemma \ref{ortll} when $t$ is close to ${+}\infty.$ 
 \par If $r\in H^{1}_{T,\delta},$ we can verify that
\begin{equation*}
    \lim_{t\to{+}\infty} Ort_{j,\pm}(t)=0 \text{, for any $j\in\{1,2,3,4\}.$}
\end{equation*}
Consequently, using Lemma \ref{ortll} and the Fundamental Theorem of Calculus, we deduce the existence of a constant $C>1$ satisfying
\begin{equation*}
    \left\vert Ort_{j,{\pm}}(t) \right\vert\leq C \left[\frac{e^{{-}\frac{7}{4}\sqrt{\omega}vt}}{v}\norm{r}_{H^{1}_{T,\frac{\sqrt{3}\omega v}{4}}}+\frac{e^{{-}\frac{3}{4}\sqrt{\omega}vt}}{v}\norm{f}_{H^{1}_{T,\frac{\sqrt{3}\omega v}{4}}}\right],
\end{equation*}
if $j\in\{1,2\}$ and $t\geq T.$
\par Therefore, using Lemma \ref{ortll} for $j\in\{3,4\}$ and the Fundamental Theorem of Calculus, we deduce from the estimate above that
\begin{equation*}
     \left\vert Ort_{j,{\pm}}(t) \right\vert\leq C \left[\frac{e^{{-}\frac{7}{4}\sqrt{\omega}vt}}{v^{2}}\norm{r}_{H^{1}_{T,\frac{\sqrt{3}\omega v}{4}}}+\frac{e^{{-}\frac{3}{4}\sqrt{\omega}vt}}{v^{2}}\norm{f}_{H^{1}_{T,\frac{\sqrt{3}\omega v}{4}}}\right],
\end{equation*}
if $t\geq T.$
\\
\textbf{Step 5.}(Estimate of $\norm{r}_{H^{1}_{T,\frac{\sqrt{3}\omega v}{4}}}.$)
First, from the estimate \eqref{coerLv} of Step $2$ and the estimates of Step $4,$ we can verify using $\norm{f}_{L^{2}_{T,\frac{3}{4}\sqrt{\omega}v}}< 1$ that if $v>0$ is small enough and $\norm{r}_{H^{1}_{T,\frac{3}{4}\sqrt{\omega}v}}< 1$ for a sufficiently large $T\geq \frac{4\ln{\left(\frac{1}{v}\right)}}{\sqrt{\omega} v},$ then there is a constant $c>0$ independent of $v$ satisfying
\begin{equation}\label{coerwe}
    \frac{1}{v^{4}}\norm{f}_{H^{1}_{T,\frac{3}{4}\sqrt{\omega}v}}^{2}+e^{\frac{3}{2}\sqrt{\omega}vt}E_{v}(r,t)>ce^{\frac{3}{2}\sqrt{\omega}vt}\norm{r(t)}_{H^{1}}^{2} \text{, for all $t\geq T. $}
\end{equation}
\par Next, from the estimate \eqref{energyestimate2solitons} of Step $3,$ we deduce that if $\norm{r}_{H^{1}_{T,\frac{3}{4}\sqrt{\omega}v}}\leq 1$ for a $T\geq \frac{4\ln{\left(\frac{1}{v}\right)}}{\sqrt{\omega} v},$ then, for the same constant $C>1$ in \eqref{energyestimate2solitons}, the estimate
\begin{equation}\label{dlvt}
    \left\vert \dot E_{v}(r,t) \right\vert<C\left(\frac{e^{{-}\frac{3}{2}\sqrt{\omega}vt}}{t}\norm{r}_{H^{1}_{T,\frac{3}{4}\sqrt{\omega}v}}^{2}+e^{{-}\frac{3\sqrt{\omega}vt}{2}}\norm{r(t)}_{H^{1}_{T,\frac{3}{4}\sqrt{\omega}v}}\norm{f(t)}_{H^{1}_{T,\frac{3}{4}\sqrt{\omega}v}}\right)
    %+\frac{e^{\frac{9}{4}\sqrt{\omega}vt}}\norm{r(t)}_{H^{1}_{T,\frac{3}{4}\sqrt{\omega}v}}^{3}\right)
\end{equation}
holds for all $t\geq T.$ Therefore, integrating estimate \eqref{dlvt} from $t\geq T$ through ${+}\infty,$ we obtain the existence of a constant $C>1$ satisfying
\begin{equation}\label{upperboundlv}
    \vert E_{v}(r,t)\vert\leq C\left(\frac{e^{{-}\frac{3}{2}\sqrt{\omega}vt}}{vt}\norm{r}_{H^{1}_{T,\frac{3}{4}\sqrt{\omega}v}}^{2}+\frac{e^{{-}\frac{3\sqrt{\omega}vt}{2}}}{v}\norm{r(t)}_{H^{1}_{T,\frac{3}{4}\sqrt{\omega}v}}\norm{f(t)}_{H^{1}_{T,\frac{3}{4}\sqrt{\omega}v}}\right),
\end{equation}
for all $t\geq T.$
\par Consequently, using estimates \eqref{coerwe}, \eqref{upperboundlv} and Young Inequality, we can deduce the existence of a constant $C_{\omega}>1$ satisfying
\begin{equation}\label{finalest}
    \norm{r}_{H^{1}_{T,\frac{3}{4}\sqrt{\omega}v}}\leq \frac{C_{\omega}}{v^{2}}\norm{f}_{H^{1}_{T,\frac{3}{4}\sqrt{\omega}v}}.
\end{equation}
\textbf{Step 6.}(Existence of a solution $r\in H^{1}_{T,\delta}.$) The proof of existence of a solution $r\in H^{1}_{T,\delta}$ of \eqref{linearnls} is similar to the one in the demonstration of Lemma $3.1$ of \cite{multison}.
\par First, for some $T_{n}>T,$ we consider a smooth cut-off function $\chi_{0}$ such that
\begin{equation*}
    \chi_{0}(t)=
    \begin{cases}
        1 \text{, if $t\leq {-}1,$}\\
        0 \text{, if $t\geq 0$}
    \end{cases}.
\end{equation*}
Let $r_{n}$ be the solution of
\begin{multline}\label{rn}
  \begin{aligned}
      i\partial_{t}r_{n}+\partial^{2}_{x}r_{n}-\omega r_{n}
{+}&F^{'}\left(\phi_{\omega}(x-vt)^{2}+\phi_{\omega}(x+vt)^{2}\right)r_{n}
\\  {+}&F^{''}\left(\phi_{\omega}(x-vt)^{2}\right)\phi_{\omega}(x-vt)e^{i({-}\frac{v^{2}t}{2}+\frac{vx}{2})}\overline{r_{n}}\\ 
{-}&F^{''}\left(\phi_{\omega}(x+vt)^{2}\right)\phi_{\omega}(x+vt)e^{i(\frac{{-}v^{2}t}{2}-\frac{vx}{2})}\overline{r_{n}}
\\ 
{+}&F^{''}\left(\phi_{\omega}(x-vt)^{2}\right)\phi_{\omega}(x-vt)e^{{-}i({-}\frac{v^{2}t}{2}+\frac{vx}{2})}r_{n}\\
{-}&F^{''}\left(\phi_{\omega}(x+vt)^{2}\right)\phi_{\omega}(x+vt)e^{{-}i({-}\frac{v^{2}t}{2}-\frac{vx}{2})}r_{n}
= e^{{-}i\omega t}\chi_{0}(t-T_{n})f(t,x),   
\end{aligned}
\end{multline}
with initial condition $r_{n}(T_{n})=0.$ Since the partial differential equation above is locally Well-Posed in $H^{1}(\mathbb{R}),$ there is a unique solution $r_{n}.$ Moreover, since $\chi_{0}(t-T_{n})f(t,x)=0$ for any $t\geq T_{n},$ we have that $r_{n}(t)=0$ for all $t\geq T_{n},$ so $r_{n}$ is in $H^{1}_{T,\frac{3\sqrt{\omega}v}{4}}.$  
\par Next, we consider a sequence $T_{n}$ converging to ${+}\infty.$ Since $f\in H^{1}_{T,\frac{3\sqrt{\omega}v}{4}},$ we have that 
\begin{equation*}
\lim_{n\to{+\infty}}\norm{f-\chi_{0}(t-T_{n})f}_{H^{1}_{T,\frac{3\sqrt{\omega}v}{4}}}=0.
\end{equation*}
Consequently, from Step $5,$ $r_{n}$ is a Cauchy-Sequence in $H^{1}_{T,\frac{3\sqrt{\omega}v}{4}}.$ This Cauchy-Sequence converges to the solution $r$ of \eqref{linearnls}, and this solution is unique because of the estimate \eqref{finalest} from the previous step.    \\
\textbf{Step 7.}(Weighted estimates.)
Furthermore, if for any $l\in\mathbb{N}$
\begin{equation}\label{lw}
    \sup_{t\geq T}e^{\frac{3\sqrt{\omega}vt}{4}}\left[\norm{x^{l}f(t,x)}_{H^{1}}+\norm{\partial^{l}_{x}f(t,x)}_{H^{1}}\right]<{+}\infty,
\end{equation}
the solution $r_{n}$ of \eqref{rn} with initial condition $r_{n}(T_{n})=0$ shall be defined in $C(([T,{+}\infty),H^{l}(\mathbb{R}))\cap C(([T,{+}\infty),x^{l}H^{1}(\mathbb{R})),$ because of the Local Well-Posedness of the $1d$ linear Schrödinger equation in these weighted spaces, and since $\phi_{\omega}\in\mathscr{S}(\mathbb{R})$ satisfies Remark \ref{phis+}.
\par First, since $f$ satisfies \eqref{lw} and the estimate \eqref{finalest} implies that
\begin{equation*}
    \norm{r_{n}}_{H^{1}_{T,\frac{3\sqrt{\omega}v}{4}}}<\frac{C_{\omega}}{v^{2}}\norm{f}_{H^{1}_{T,\frac{3\sqrt{\omega}v}{4}}},
\end{equation*} 
we can verify after we differentiate the partial differential equation \eqref{rn} on $x$ using \eqref{finalest} the existence of a constant $C_{2,\omega}>1$ satisfying
\begin{equation*}
    \norm{\partial_{x}r_{n}}_{H^{1}_{T,\frac{3\sqrt{\omega}v}{4}}}<\frac{C_{2,\omega}}{v^{4}}\norm{f}_{H^{2}_{T,\frac{3\sqrt{\omega}v}{4}}},
\end{equation*}
for any $n\in\mathbb{N}.$ 
\par Similarly, we can verify by induction for any $l\in\mathbb{N}$ the existence of a constant $C_{l+1,\omega}>1$ satisfying
\begin{equation}\label{gl0}
    \norm{r_{n}}_{H^{l+1}_{T,\frac{3\sqrt{\omega}v}{4}}}<\frac{C_{l+1,\omega}}{v^{2(l+1)}}\norm{f}_{H^{l+1}_{T,\frac{3\sqrt{\omega}v}{4}}},
\end{equation}
for any $n\in\mathbb{N}.$
Moreover, we are going to verify for any natural number $m\geq0$ that 
\begin{equation}\label{gl0m}
\norm{x^{m}r_{n}}_{H^{l+1}_{T,\frac{3\sqrt{\omega}v}{4}}}<\frac{C_{l+1,m,\omega}}{v^{2(2m+l+1)}}\max_{d_{1}+d_{2}=m+l+1,d_{1}\leq m}\norm{(1+ x^{2})^{\frac{d_{1}}{2}}f}_{H^{d_{2}}_{T,\frac{3\sqrt{\omega}v}{4}}},
\end{equation}
for any $n\in\mathbb{N}$ such that the constant $C_{l+1,m,\omega}>1$ does not depend on $n.$ We already verified that \eqref{gl0m} is true when $m=0,$ so we assume that is true from $m=0$ through $m=m_{1}\geq 0.$ 
\par Furthermore, the function $r_{n,m}(t,x)\coloneqq x^{m}r_{n}(t,x)$ is a strong solution of the following partial differential equation
\begin{multline}\label{rnnm}
  \begin{aligned}
      i\partial_{t}r_{n,m}+\partial^{2}_{x}r_{n,m}-\omega r_{n,m}
{+}&F^{'}\left(\phi_{\omega}(x-vt)^{2}+\phi_{\omega}(x+vt)^{2}\right)r_{n,m}
\\  {+}&F^{''}\left(\phi_{\omega}(x-vt)^{2}\right)\phi_{\omega}(x-vt)e^{i({-}\frac{v^{2}t}{2}+\frac{vx}{2})}\overline{r_{n,m}}\\ 
{-}&F^{''}\left(\phi_{\omega}(x+vt)^{2}\right)\phi_{\omega}(x+vt)e^{i(\frac{{-}v^{2}t}{2}-\frac{vx}{2})}\overline{r_{n,m}}
\\
{+}&F^{''}\left(\phi_{\omega}(x-vt)^{2}\right)\phi_{\omega}(x-vt)e^{{-}i({-}\frac{v^{2}t}{2}+\frac{vx}{2})}r_{n,m}\\ 
{-}&F^{''}\left(\phi_{\omega}(x+vt)^{2}\right)\phi_{\omega}(x+vt)e^{{-}i({-}\frac{v^{2}t}{2}-\frac{vx}{2})}r_{n,m}
\\= &e^{{-}i\omega t}\chi_{0}(t-T_{n})x^{m}f(t,x)\\&{+}m(m-1)r_{n,m-2}-2m\partial_{x}r_{n,m-1},   
\end{aligned}
\end{multline}
Consequently, from estimate \eqref{finalest},  we deduce after we differentiate \eqref{rnnm} on $x$ $l$ times the existence of constants  $C_{\omega,m,l},\,C_{\omega,m,l,0}>1$
satisfying
\begin{align}\nonumber
     \norm{r_{n,1}}_{H^{1}_{T,\frac{3\sqrt{\omega}v}{4}}}<&\frac{C_{\omega,m,0}}{v^{2}}\left[\norm{x f}_{H^{1}_{T,\frac{3\sqrt{\omega}v}{4}}}+\norm{r_{n}}_{H^{2}_{T,\frac{3\sqrt{\omega}v}{4}}}\right],\\ \nonumber
     \norm{r_{n,m}}_{H^{1}_{T,\frac{3\sqrt{\omega}v}{4}}}<&\frac{C_{\omega,m,0}}{v^{2}}\left[\norm{x^{m}f}_{H^{1}_{T,\frac{3\sqrt{\omega}v}{4}}}+\norm{r_{n,m-2}}_{H^{1}_{T,\frac{3\sqrt{\omega}v}{4}}}+\norm{r_{n,m-1}}_{H^{2}_{T,\frac{3\sqrt{\omega}v}{4}}}\right]
    \\ \nonumber\norm{\partial^{l}_{x}r_{n,m}}_{H^{1}_{T,\frac{3\sqrt{\omega}v}{4}}}<&\frac{C_{\omega,m,l,0}}{v^{2}}\Bigg[\norm{\partial^{l}_{x}\left[x^{m}f\right]}_{H^{1}_{T,\frac{3\sqrt{\omega}v}{4}}}+\norm{\partial^{l}_{x}r_{n,m-2}}_{H^{1}_{T,\frac{3\sqrt{\omega}v}{4}}}+\norm{\partial^{l+1}_{x}r_{n,m-1}}_{H^{1}_{T,\frac{3\sqrt{\omega}v}{4}}}\\ \nonumber &{+}\norm{r_{n,m}}_{H^{l}_{T,\frac{3\sqrt{\omega}v}{4}}}\Bigg],
\end{align}
for any $n\in\mathbb{N},l\in\mathbb{N}_{\geq 1}$  and $m\in\mathbb{N}_{\geq 2}.$ 
\par Therefore, for any $l\in\mathbb{N}_{\geq 1},$ the second and third estimates above imply the existence of a constant $C_{\omega,m_{1},l,1}>1$ satisfying
\begin{align}\label{id90}
\norm{r_{n,m}}_{H^{l+1}_{T,\frac{3\sqrt{\omega}v}{4}}}<&\frac{C_{\omega,m,l,1}}{v^{2}}\Bigg[\norm{(1+x^{2})^{\frac{m}{2}}f}_{H^{l+1}_{T,\frac{3\sqrt{\omega}v}{4}}}+\max(m-2,0)\norm{r_{n,m-2}}_{H^{l+1}_{T,\frac{3\sqrt{\omega}v}{4}}}\\& \nonumber{+}\norm{r_{n,m-1}}_{H^{l+2}_{T,\frac{3\sqrt{\omega}v}{4}}}+\norm{r_{n,m}}_{H^{l}_{T,\frac{3\sqrt{\omega}v}{4}}}\Bigg],
\end{align}
for any $n\in\mathbb{N}.$   
\par Consequently, if \eqref{gl0m} is true form $m=m_{1},$ then the estimate of $\norm{r_{n,1}}_{}$ \eqref{id90} imply that \eqref{gl0m} is true for $m=m_{1}+1,$ which finishes the proof by induction of the weighted norm of the remainder $r_{n}.$ 
\par In conclusion, since $r_{n}$ converge to $r$ in $H^{1}_{T,\frac{3\sqrt{\omega}v}{4}},$ we can verify from the estimates \eqref{gl0m} and Banach-Alaoglu Theorem that
\begin{equation*}
\norm{x^{m}r}_{H^{l}_{T,\frac{3\sqrt{\omega}v}{4}}}<\frac{C_{\omega,m,l}}{v^{2(2m+l)}}\max_{d_{1}+d_{2}=m+l,d_{1}\leq m}\norm{(1+x^{2})^{\frac{d_{1}}{2}}f}_{H^{d_{2}}_{T,\frac{3\sqrt{\omega}v}{4}}}.
\end{equation*}

\end{proof}
\begin{proof}[Proof of Theorem \ref{uniq}.]
\textbf{Step 1.}(Exponential Decay of the remainder.)
Let $r(t)=e^{{-}i\omega t}u+\phi_{\omega}(x-vt)e^{i(\frac{vx}{2}-\frac{v^{2}t}{4})}-\phi_{\omega}(x+vt)e^{i(\frac{{-}vx}{2}-\frac{v^{2}t}{4})}.$ Since $u$ is a solution of the partial differential equation \eqref{NLS3}, we have that
\begin{align}\label{linearnlss}
       i\partial_{t}r+\partial^{2}_{x}r-\omega r(t,x)
{+}&F^{'}\left(\phi_{\omega}(x-vt)^{2}+\phi_{\omega}(x+vt)^{2}\right)r(t,x)
\\ \nonumber {+}&F^{''}\left(\phi_{\omega}(x-vt)^{2}\right)\phi_{\omega}(x-vt)e^{i({-}\frac{v^{2}t}{2}+\frac{vx}{2})}\overline{r(t)}\\ \nonumber
{-}&F^{''}\left(\phi_{\omega}(x+vt)^{2}\right)\phi_{\omega}(x+vt)e^{i(\frac{{-}v^{2}t}{2}-\frac{vx}{2})}\overline{r(t)}
\\ \nonumber
{+}&F^{''}\left(\phi_{\omega}(x-vt)^{2}\right)\phi_{\omega}(x-vt)e^{{-}i({-}\frac{v^{2}t}{2}+\frac{vx}{2})}r(t)\\ \nonumber
{-}&F^{''}\left(\phi_{\omega}(x+vt)^{2}\right)\phi_{\omega}(x+vt)e^{{-}i({-}\frac{v^{2}t}{2}-\frac{vx}{2})}r(t)
= \left[I(t)+N(r)\right],
\end{align}
such that
\begin{align*}
    I(t)=&F^{'}\left(\phi_{\omega}(x-vt)^{2}+\phi_{\omega}(x+vt)^{2}\right)\left[\phi_{\omega}(x-vt)e^{i\frac{vx}{2}}-\phi_{\omega}(x+vt)e^{{-}i\frac{vx}{2}}\right]\\&{-}F^{'}\left(\phi_{\omega}(x-vt)^{2}\right)\phi_{\omega}(x-vt)e^{i\frac{vx}{2}}+F^{'}\left(\phi_{\omega}(x+vt)^{2}\right)\phi_{\omega}(x+vt)e^{{-}i\frac{vx}{2}},
\end{align*}
and the function $N(r)(t,x)$ is bounded in $H^{1}(\mathbb{R})$ satisfying for some constant $C$ any $t\geq 1$
\begin{equation*}
    \norm{N(r)(t)}_{H^{1}}\leq C\left[\norm{r(t)}_{H^{1}}e^{{-}2\sqrt{\omega}vt}(1+\vert t\vert v)+\norm{r(t)}_{H^{1}}^{2}\right],
\end{equation*}
this follows from the fact that $H^{1}(\mathbb{R})$ is an algebra and from the Lemma \ref{interactt} with Remark \ref{phis+}. 
\par Furthermore, for any two functions $r_{1},\,r_{2}\in C\left([T,{+}\infty),H^{1}(\mathbb{R})\right),$ we can verify from the partial differential equations \eqref{NLS3} and \eqref{linearnlss} that
\begin{align}\label{p0}
    \norm{N(r_{1})(t)-N(r_{2})(t)}_{H^{1}}\leq C\norm{r_{1}(t)-r_{2}(t)}_{H^{1}}\Bigg[&\max_{j\in\{1,2\}}\norm{r_{j}(t)}_{H^{1}}^{2q-2}+\max_{j\in\{1,2\}}\norm{r_{j}(t)}_{H^{1}}\\ \nonumber
    &{+}e^{{-}2\sqrt{\omega}vt}(1+\vert t\vert v)\Bigg],
\end{align}
for any $t\geq 1,$ the number $q\in\mathbb{N}_{\geq 2}$ is the degree of the polynomial $F.$  
\par Consequently, since Lemma \ref{interactt} and Remark \ref{phis+} imply the existence of a constant $C>1$ such that
\begin{equation*}
    \norm{I(t)}_{H^{1}}\leq Ce^{{-}2\sqrt{\omega }vt},
\end{equation*}
for any $t\geq 1,$ we deduce using Young Inequality that $r$ is a solution of a partial differential equation \eqref{linearnls} such that
\begin{equation*}
    \norm{f(t)}_{H^{1}}\leq C\left[\norm{r(t)}_{H^{1}}^{2}+e^{{-}2\sqrt{\omega }vt}\right],
\end{equation*}
for some constant $C>1$ when $v>0$ is sufficiently small and $t\geq \frac{4\ln{\left(\frac{1}{v}\right)}}{\sqrt{\omega}v}.$
\par Moreover, since $F$ is a real polynomial satisfying \eqref{H1} and $\phi_{\omega}$ satisfies Remark \eqref{ordt}, we deduce using Lemma \ref{interactt} the following estimate
\begin{equation}\label{In}
    \norm{(1+ x^{2})^{\frac{m}{2}}I(t)}_{H^{n}}\leq C_{m,n} \left(1+\vert t \vert v\right)^{m}e^{{-}2\sqrt{\omega}\vert t\vert v},
\end{equation}
for any $m,\,n\in\mathbb{N},$ if $t\geq 1.$  
\par Furthermore, since $H^{n}(\mathbb{R})$ is an algebra for any $n\geq 1$ and
\begin{equation*}
    \norm{f g}_{H^{n}}\leq C_{n}\norm{f}_{H^{n}}\norm{g}_{H^{n}},
\end{equation*}
for any $f,\,g\in H^{n}(\mathbb{R})$ for a constant $C_{n}>1,$ we can verify from Lemma \ref{interactt}, the definition of $N,$ and the fact that $F$ is a real polynomial satisfying \eqref{H1} the following estimate
\begin{equation}\label{Nnorm}
    \norm{(1+ x ^{2})^{\frac{m}{2}} N(r)}_{H^{n}}\leq C_{m,n}\left[\norm{(1+ x^{2})^{\frac{m}{2}}r(t)}_{H^{n}}e^{{-}2\sqrt{\omega}vt}(1+\vert t\vert v)+\norm{(1+x^{2})^{\frac{m}{2}}r(t)}_{H^{n}}\norm{r(t)}_{H^{n}}\right],
\end{equation}
if $\norm{r(t)}_{H^{n}}\leq 1$ and $t>1$ is sufficiently large, for some constant $C_{m,n}>1$ depending only on $m,n.$ 
\par Consequently, using Lemma \ref{edl} and estimates \eqref{In}, \eqref{Nnorm}, we conclude that there exists $C_{m,n}>1$ satisfying
\begin{equation}\label{op}
\max_{d_{1}+d_{2}=m+n,d_{1}\leq m}\norm{(1+\vert x \vert^{2})^{\frac{d_{1}}{2}}r(t)}_{H^{d_{2}}_{T,\frac{3\sqrt{\omega}v}{4}}}\leq \frac{C_{m,n}}{v^{2(2m+n)}}e^{{-}5\frac{\sqrt{\omega}\vert T\vert v}{4}}\left(1+v T\right)^{m}<1,
\end{equation}
for any $T\geq \frac{4(m+n)\ln{\frac{1}{v}}}{\sqrt{\omega}v}$ when $v>0$ is small enough.
\par Next, using estimate \eqref{e000} of Lemma \ref{edl} and \eqref{p0}, we can verify using the Picard iteration method that there exists a unique solution $u_{0}(t)$ of \eqref{NLS3} satisfying \eqref{moq1} when $c=\frac{3\sqrt{\omega}v}{4},$ and from the estimate \eqref{op} we obtain \eqref{pureweight}.\\
\textbf{Step 2.}(Uniqueness of the two solitary waves solution.)
\par It remains to prove that if a solution $u$ satisfies \eqref{moq1} for some $c_{1}>0,$ then it satisfies \eqref{moq1} for $c_{1}\geq \frac{3\sqrt{\omega}v}{4}.$ We assume that there exists solution $u_{0}$ of \eqref{NLS3} satisfying \eqref{moq1} for $0<c<\frac{3\sqrt{\omega}v}{4}.$ Let $u$ be the solution of \eqref{NLS3} satisfying \eqref{moq1} constructed in Step $1.$ Obviously, $u_{0}$ satisfies \eqref{moq1} for any $c<\frac{3\sqrt{\omega}v}{4},$ because $u_{0}\in H^{1}_{T,\frac{3\sqrt{\omega}v}{4}}$ for a $T>1$ large enough. 
\par From now on, we consider the following partial differential equation satisfied by $r_{d}(t)\coloneqq e^{{-}i\omega t}\left(u(t)-u_{0}(t)\right)$ when $t>0$ is sufficiently large.
\begin{align*}
     i\partial_{t}r_{d}+\partial^{2}_{x}r_{d}-\omega r_{d}
{+}&F^{'}\left(\phi_{\omega}(x-vt)^{2}+\phi_{\omega}(x+vt)^{2}\right)r_{d}
\\ \nonumber {+}&F^{''}\left(\phi_{\omega}(x-vt)^{2}\right)\phi_{\omega}(x-vt)e^{i({-}\frac{v^{2}t}{2}+\frac{vx}{2})}\overline{r_{d}}\\ \nonumber
{-}&F^{''}\left(\phi_{\omega}(x+vt)^{2}\right)\phi_{\omega}(x+vt)e^{i(\frac{{-}v^{2}t}{2}-\frac{vx}{2})}\overline{r_{d}}
\\ \nonumber
{+}&F^{''}\left(\phi_{\omega}(x-vt)^{2}\right)\phi_{\omega}(x-vt)e^{{-}i({-}\frac{v^{2}t}{2}+\frac{vx}{2})}r_{d}\\ \nonumber
{-}&F^{''}\left(\phi_{\omega}(x+vt)^{2}\right)\phi_{\omega}(x+vt)e^{{-}i({-}\frac{v^{2}t}{2}-\frac{vx}{2})}r_{d}
= N(u)-N(u_{0}).
\end{align*}
\par The next steps are completely similar to the argument in the proof of Lemma \ref{linearnls} for $f=N(u)-N(u_{0})$ but now considering the space $H^{1}_{T,c}$ for a $T>1$ sufficiently large. 
\par Since $f=N(u)-N(u_{0}),$ we deduce from \eqref{p0} and triangle inequality that
\begin{equation}\label{b0}
    \norm{f(t)}_{H^{1}}\leq C\left[\norm{r_{d}(t)}_{H^{1}}^{2}+\norm{r_{d}(t)}_{H^{1}}\norm{r(t)}_{H^{1}}+e^{{-}\frac{39}{20}\sqrt{\omega}vt}\norm{r_{d}(t)}_{H^{1}}\right],
\end{equation}
if $v>0$ is small enough, $t\geq 1$ is sufficiently large, and $r(t)$ is the remainder denoted in Step $1.$ From this, Lemma \ref{ortll}, and estimate \eqref{coerLv}, we can verify from the Fundamental Theorem of Calculus the existence of constants $k>0$ satisfying for any $T>1$ large enough the following inequality
\begin{equation}\label{en000}
    E_{v}(r_{d},t)\geq k\norm{r_{d}}_{H^{1}_{T,c}}^{2}e^{{-}2cT}.
\end{equation}

\par Next, using estimates \eqref{energyestimate2solitons}, \eqref{b0}, \eqref{en000}, and the fact that $\norm{r}_{H^{1}_{T,\frac{3}{4}\sqrt{\omega}v}}<{+}\infty,$ we can verify from the Fundamental Theorem of Calculus and Minkowski inequality the existence of constants $k>0,\,K>1$ depending on $c$ and $v$ satisfying the following inequality
\begin{equation}\label{b000}
    k\norm{r_{d}(t)}_{H^{1}_{T,c}}^{2}e^{{-}2cT}\leq K\norm{r_{d}(t)}_{H^{1}_{T,c}}^{2}e^{{-}\frac{3}{2}\sqrt{\omega}vT-2cT}, 
\end{equation}
for all $T> 1$  sufficiently large. Consequently,
\begin{equation*}
\norm{r_{d}}_{H^{1}_{T,c}}=0,
\end{equation*}
when $T>1$ is large enough, since $c,\,v>0.$
\par Therefore, $\norm{r_{d}(t)}_{H^{1}}\equiv 0,$ when $t>1$ is large enough.  In conclusion, the uniqueness of \eqref{NLS3} in $H^{1}_{T,\frac{3}{4}\sqrt{\omega}v}$ implies that $u(t)\equiv u_{0}(t),$ which finishes the proof of Theorem \ref{uniq}.  

%\par Consequently, we deduce that
%\begin{equation*}
%\norm{r_{d}(t)}_{H^{1}_{T,\frac{3}{4}\sqrt{\omega}v}}^{2}<\norm{r_{d}(t)}_{H^{1}_{T,\sqrt{\omega}v}}^{2}< \frac{k}{K}.
%\end{equation*}
%In conclusion, since $r_{d}=r-r_{j,1}$ and $r\in H^{1}_{T,\frac{3}{4}\sqrt{\omega}v}$ for any $T\geq \frac{4\ln{\left(\frac{1}{v}\right)}}{\sqrt{\omega}v},$  $r_{j,1}\in H^{1}_{T,\frac{3}{4}\sqrt{\omega}v},$ and so $r_{j,1}\equiv r,$ because we proved uniqueness of $u$ satisfying \eqref{moq1} when $c=\frac{3}{4}\sqrt{\omega}v.$
\end{proof}

\bibliographystyle{plain}
\bibliography{main}

\begin{thebibliography}{10}

\bibitem{optics}
Govind Agrawal.
\newblock Nonlinear fiber optics.
\newblock {\em Lecture Notes in Physics}, 18:195--211, 01 2007.

\bibitem{solitoneq}
H.~{Berestycki} and P.~L. {Lions}.
\newblock {Nonlinear scalar field equations, I existence of a ground state}.
\newblock {\em Archive for Rational Mechanics and Analysis}, 82(4):313--345, December 1983.

\bibitem{asympt1}
Vladimir Buslaev and Catherine Sulem.
\newblock On asymptotic stability of solitary waves for nonlinear schrödinger equations.
\newblock {\em Annales de l'Institut Henri Poincare (C) Non Linear Analysis}, 20:419--475, 05 2003.

\bibitem{phy2}
Rémi Carles and Christof Sparber.
\newblock Orbital stability vs. scattering in the cubic-quintic schrödinger equation.
\newblock {\em Reviews in Mathematical Physics}, 33:2150004, 09 2020.

\bibitem{asympt0}
Charles Collot and Pierre Germain.
\newblock Asymptotic stability of solitary waves for one dimensional nonlinear schr\"odinger equations, 06 2023.

\bibitem{solisch}
Manoussos Grillakis, Jalal Shatah, and Walter Strauss.
\newblock Stability theory of solitary waves in the presence of symmetry, i.
\newblock {\em Journal of Functional Analysis}, 74(3):160--197, 1987.

\bibitem{solisch2}
Manoussos Grillakis, Jalal Shatah, and Walter Strauss.
\newblock Stability theory of solitary waves in the presence of symmetry, ii.
\newblock {\em Journal of Functional Analysis}, 94(2):308--348, 1990.

\bibitem{specbook}
P.D. Hislop and I.~M. Sigal.
\newblock {\em Introduction to Spectral Theory}, volume 113 of {\em Applied Mathematical Sciences}.
\newblock Springer, 1996.

\bibitem{holmerlin}
Justin Holmer and Quanhui Lin.
\newblock Phase-driven interaction of widely separated nonlinear schrÖdinger solitons.
\newblock {\em Journal of Hyperbolic Differential Equations}, 9, 2012.

\bibitem{multison}
Jacek Jendrej and Gong Chen.
\newblock Kink networks for scalar fields in dimension $1+1$.
\newblock {\em Nonlinear Analysis}, 215, 2022.

\bibitem{phy5}
Karima Khusnutdinova.
\newblock Nonlinear waves in integrable and nonintegrable systems (mathematical modeling and computation 16) by jianke yang.
\newblock {\em Bulletin of the London Mathematical Society}, 47:188--190, 01 2015.

\bibitem{phy1}
Yuri Kivshar.
\newblock Dynamics of solitons in nearly integrable systems.
\newblock {\em Reviews of Modern Physics - REV MOD PHYS}, 61:763--915, 10 1989.

\bibitem{phy3}
Yuri Kivshar, Dmitry Pelinovsky, Thierry Cretegny, and Michel Peyrard.
\newblock Internal modes of solitary waves.
\newblock {\em Physical Review Letters - PHYS REV LETT}, 80:5032--5035, 06 1998.

\bibitem{phy4}
Swapan Konar and Manoj Mishra.
\newblock The effect of quintic nonlinearity on the propagation characteristics of dispersion managed optical solitons.
\newblock {\em Chaos Solitons and Fractals}, 29:823, 08 2006.

\bibitem{asympt3}
Ze~Li.
\newblock Asymptotic stability of solitons to 1d nonlinear schrödinger equations in subcritical case.
\newblock {\em Frontiers of Mathematics in China}, 15:1--35, 09 2020.

\bibitem{Martel}
Yvan Martel.
\newblock Asymptotic stability of solitary waves for the 1d cubic-quintic schr{\"o}dinger equation with no internal mode.
\newblock {\em Probability and Mathematical Physics}, 2021.

\bibitem{asympt2}
Yvan Martel.
\newblock Asymptotic stability of small standing solitary waves of the one-dimensional cubic-quintic schrödinger equation.
\newblock {\em Inventiones mathematicae}, 237:1253–1328, 05 2024.

\bibitem{mnls}
Yvan Martel and Frank Merle.
\newblock Multi solitary waves for nonlinear schrödinger equations.
\newblock {\em Annales de l'Institut Henri Poincare (C) Non Linear Analysis}, 23:849--864, 11 2006.

\bibitem{stabcol}
Yvan Martel and Frank Merle.
\newblock Stability of two soliton colision for non-integrable \textit{gKdV} equations.
\newblock {\em Communications in Mathematical Physics}, 286:39--79, 2009.

\bibitem{collision1}
Yvan Martel and Frank Merle.
\newblock Inelastic interaction of nearly equal solitons for the quartic \textit{gKdV} equation.
\newblock {\em Inventiones Mathematicae}, 183(3):563--648, 2011.

\bibitem{asympt5}
Satoshi Masaki, Jason Murphy, and Jun-Ichi Segata.
\newblock Asymptotic stability of solitary waves for the $ 1d $ nls with an attractive delta potential.
\newblock {\em Discrete and Continuous Dynamical Systems}, 43, 01 2023.

\bibitem{asympt4}
Tetsu Mizumachi.
\newblock Asymptotic stability of small solitary waves to 1d nonlinear schrödinger equations with potential.
\newblock {\em Journal of Mathematics of Kyoto University - J MATH KYOTO UNIV}, 48, 01 2008.

\bibitem{third}
Abdon Moutinho.
\newblock On the kink-kink collision problem for the $\phi^{6}$ model with low speed.
\newblock {\em Accepted in Analysis and PDE}, 2022.
\newblock {arXiv:2211.09749 }.

\bibitem{first}
Abdon Moutinho.
\newblock {Dynamics of Two Interacting Kinks for the $\phi^{6}$ Model}.
\newblock {\em Commun. Math. Phys.}, 401(2):1163--1235, 2023.

\bibitem{second}
Abdon Moutinho.
\newblock Approximate kink-kink solutions for the $\phi^{6}$ model in the low-speed limit.
\newblock {\em Asymptotic Analysis}, 140:1--90, 05 2024.

\bibitem{munozkdv2}
Claudio Munoz.
\newblock Inelastic character of solitons of slowly varying \textit{gKdV} equations.
\newblock {\em Communications in Mathematical Physics}, 314(3):817--852, 2012.

\bibitem{munozin}
Claudio Muñoz.
\newblock On the inelastic two-soliton collision for gkdv equations with general nonlinearity.
\newblock {\em International Mathematics Research Notices - INT MATH RES NOTICES}, 2010, 03 2009.

\bibitem{ohta}
Masahito Ohta.
\newblock Stability and instability of standing waves for one dimensional nonlinear schr{\"o}dinger equations with double power nonlinearity.
\newblock {\em Kodai Mathematical Journal}, 18(1):68--74, 1995.

\bibitem{perelman}
Galina Perelman.
\newblock Two soliton collision for nonlinear schrödinger equations in dimension $1$.
\newblock {\em Annales de l'IHP Anal. Non Linéaire}, 28(3):357--384, 2011.

\bibitem{collisiondidier}
Didier Pilod and Frédéric Valet.
\newblock Dynamics of the collision of two nearly equal solitary waves for the zakharov-kuznetsov equation.
\newblock 2023.
\newblock {arXiv:2403.02262 }.

\bibitem{asympt6}
Guillaume Rialland.
\newblock Asymptotic stability of solitary waves for the 1d near-cubic non-linear schrödinger equation in the absence of internal modes.
\newblock {\em Nonlinear Analysis}, 241:113474, 04 2024.

\bibitem{sigcoll}
W.~Salem, Jürg Fröhlich, and Israel~Michael Sigal.
\newblock Colliding solitons for the nonlinear schrödinger equation.
\newblock {\em Communications in Mathematical Physics}, 291:151--176, 10 2009.

\bibitem{fourpower}
Islam Samir, Ahmed Arnous, Yakup Yildirim, Anjan Biswas, Luminita Moraru, and Simona Moldovanu.
\newblock Optical solitons with cubic-quintic-septic-nonic nonlinearities and quadrupled power-law nonlinearity: An observation.
\newblock {\em Mathematics}, 10:4085, 11 2022.

\bibitem{opticssoltion}
Aly Seadawy and Bayan Alsaedi.
\newblock Soliton solutions of nonlinear schrödinger dynamical equation with exotic law nonlinearity by variational principle method.
\newblock {\em Optical and Quantum Electronics}, 56, 02 2024.

\bibitem{dsipersivebook}
Terence Tao.
\newblock {\em Nonlinear dispersive equations: local and global analysis}.
\newblock AMS, 2006.

\bibitem{modstability}
Michael~I. Weinstein.
\newblock Modulational stability of ground states of nonlinear schrödinger equations.
\newblock {\em SIAM Journal on Mathematical Analysis}, 16(3):472--491, 1985.

\bibitem{cubicinv}
V.~E. Zakharov and A.~B. Shabat.
\newblock Exact theory of two-dimensional self-focusing and one-dimensional self-modulation of waves in nonlinear media.
\newblock {\em Journal of Experimental and Theoretical Physics}, 34:62--69, 1970.

\bibitem{opticstriple1}
Elsayed Zayed, Mohamed Alngar, Anjan Biswas, Abdul Kara, Luminita Moraru, Mehmet Ekici, A.K. Alzahrani, and Milivoj Belić.
\newblock Solitons and conservation laws in magneto-optic waveguides with triple-power law nonlinearity.
\newblock {\em Journal of Optics}, 49:1--7, 09 2020.

\end{thebibliography}
\end{document}